\documentclass[journal]{IEEEtran}
\ifCLASSINFOpdf
\else
\fi
\usepackage[latin9]{inputenc}
\usepackage[a4paper]{geometry}
\usepackage{amsmath}
\usepackage{amsthm}
\usepackage{amssymb}
\usepackage{graphicx}
\usepackage{color, colortbl}
\usepackage{caption}
\usepackage[normalem]{ulem}
\usepackage[table]{xcolor}
\usepackage{multirow}
\usepackage{array}
\usepackage{booktabs}
\usepackage{longtable}
\makeatletter
\theoremstyle{plain}
\newtheorem{thm}{\protect\theoremname}[section]
\theoremstyle{plain}
\newtheorem{lem}[thm]{\protect\lemmaname}
\theoremstyle{plain}
\newtheorem{rem}[thm]{\protect\remarkname}
\theoremstyle{plain}

\theoremstyle{plain}

\usepackage{bbm}
\usepackage{enumerate}
\usepackage[hyphens]{url}
\usepackage{nicefrac}\newcommand{\R}{\mathbb{R}}
\newcommand{\N}{\mathbb{N}}

\newcommand{\cNt}{\mathcal{N}^{\Theta}}

\newcommand{\rank}{\operatorname{rank}}
\newcommand{\st}{\operatorname{s.\!t.\,}}
\newcommand{\sign}{\operatorname{sign}}

\title{\vspace{-1.7cm}
	Low-Rank plus Sparse Decomposition of Covariance Matrices\\ using Neural Network Parametrization}
\author{Michel Baes\thanks{RiskLab, Department of Mathematics, ETH Zurich, e-mail: $michel.baes@math.ethz.ch$}\and Calypso Herrera\thanks{Department of Mathematics, ETH Zurich, e-mail: $calypso.herrera@math.ethz.ch$}\and Ariel Neufeld\thanks{Division of Mathematical Sciences, NTU Singapore, e-mail: $ariel.neufeld@ntu.edu.sg$}  \and Pierre Ruyssen\thanks{
		Google Brain, Google Zurich, 
		e-mail: 
		$pierrot@google.com$
}}

\providecommand{\lemmaname}{Lemma}
\providecommand{\theoremname}{Theorem}
\providecommand{\remarkname}{Remark}
\providecommand{\definitionname}{Definition}
\providecommand{\examplename}{Example}

\makeatother

\providecommand{\lemmaname}{Lemma}
\providecommand{\theoremname}{Theorem}
\providecommand{\remarkname}{Remark}
\providecommand{\definitionname}{Definition}
\providecommand{\examplename}{Example}

\definecolor{light-gray}{gray}{0.90}

\begin{document}
%
\title{Low-Rank plus Sparse Decomposition \\of Covariance Matrices\\ using Neural Network Parametrization}
%
%
%

\author{Michel Baes, 
        Calypso Herrera, 
        Ariel Neufeld, 
        and~Pierre Ruyssen
\thanks{M. Baes is with the RiskLab, Department of Mathematics,
	
	 ETH Zurich. 
	e-mail: michel.baes@math.ethz.ch.}
\thanks{C. Herrera is with the Department of Mathematics, ETH Zurich. 
	
	e-mail: calypso.herrera@math.ethz.ch.}
\thanks{
A. Neufeld 
is with the Division of Mathematical Sciences, 

NTU Singapore.
e-mail: ariel.neufeld@ntu.edu.sg}
\thanks{P. Ruyssen is with 	Google Brain, Google Zurich.  
	
			e-mail: pierrot@google.com.}
\thanks{Manuscript uploaded June 15, 2021. 
}
}

%
%

\markboth{
}%
{Shell \MakeLowercase{\textit{et al.}}: Bare Demo of IEEEtran.cls for IEEE Journals}
%



\maketitle

\begin{abstract}
This paper revisits the problem of decomposing a positive semidefinite
matrix as a sum of a matrix with a given rank plus a sparse matrix.
An immediate application can be found in portfolio optimization, when
the matrix to be decomposed is the covariance between the different
assets in the portfolio. Our approach consists in representing the
low-rank part of the solution as the product $MM^{T}$, where $M$
is a rectangular matrix of appropriate size, parametrized by the
coefficients of a deep neural network. We then use a gradient descent
algorithm to minimize an appropriate loss function over the parameters
of the network. We deduce its convergence rate to a local optimum
from the Lipschitz smoothness of our loss function. We show that the rate of convergence grows polynomially in the dimensions of the input, output, and the size of each of the hidden layers.
\end{abstract}


%
\IEEEpeerreviewmaketitle


\section{Introduction}

We present a  new approach to decompose a possibly large covariance matrix $\Sigma$ into the sum of a positive semidefinite low-rank matrix $L$ plus a sparse matrix $S$. Our approach consists  in fixing an (upper bound for the) rank $k$ of $L$ by defining 
$L:=MM^T$ for a suitable $M \in \R^{n\times k}$, where one parametrizes $M$ using a deep neural network, whose coefficients are minimized using a gradient descent method.

{Albeit our method can be used in any context where such a problem occurs, our primary application of interest is rooted in Finance.}
When studying the correlation matrix between the returns of financial assets, it is important for the design of a well-diversified portfolio to identify groups of heavily correlated assets, or more generally, to identify a few ad-hoc features, or economic factors, that describe some dependencies between these assets. To this effect, the most natural tool is to determine the few first dominant eigenspaces of the correlation matrix and to interpret them as the main features driving the behavior of the portfolio. This procedure, generally termed Principal Component Analysis (PCA), is widely used. However, { PCA does not ensure any sparsity between the original matrix $\Sigma$ and its approximation $A$.
As it turns out, many coefficients of $\Sigma-A$ can be relatively large with respect to the others; these indicate pairs of assets that present an ignored large correlation between themselves, beyond the dominant features revealed by PCA.
Following \cite{lisa2001}, to reveal this extra structure present in $\Sigma$, we decompose it into the sum of a low-rank matrix $L$, which describes the dominant economic factors of a portfolio, plus a sparse matrix $S$, to identify hidden large correlations between assets pairs. The number of those economic factors is set according to the investor's views on the market, and coincides with the rank of $L$. The sparse part $S$ can be seen as a list of economic abnormalities, which can be exploited by the investor.}

{Beyond covariance matrices,} this decomposition is a procedure abundantly used in image and video processing  for compression and interpretative purposes  \cite{Bouwmans:2016:HRL:2994445}, but also in latent variable model selection \cite{candes2009}, in latent semantic indexing \cite{Kerui2010}, in graphical model selection \cite{Aybat2013AlternatingDM}, in graphs \cite{Kalofolias2015}, and in gene expression \cite{Liu2015}, among others. 
A rich collection of algorithms exist to compute such decomposition, see  \cite{Chandrasekaran2010LatentVG,  He2012,Lin2010, Wang2012, Zhou2011} to cite but a few, most of which are reviewed in \cite{bouwmans2015} and implemented in the Matlab LRS library \cite{lrslibrary2015}. {However, our method can only address the decomposition of a \emph{symmetric positive semidefinite} matrix, as it uses explicitly and takes full advantage of this very particular structure. }

{The Principal Component Pursuit (PCP) reformulation of this problem has been proposed and studied by \cite{candes2009,Chandrasekaran2010LatentVG,Lin2010} as a robust alternative to PCA, and generated a number of efficient algorithms. For a given $\delta>0$, the PCP problem is formulated as 
\begin{equation}
\min_{L,S} ||L||_{*}+ \delta ||S||_1 \quad \st \Sigma = L + S\,.
\label{min_pb}
\end{equation}
where $\Sigma \in \mathbb{R}^{n\times n}$ is the observed matrix, $||L||_{*}$ is the nuclear
norm of matrix $L$ (i.e.\ the sum of all singular
values of $L$) and $||S||_1 $ is the $l^1$-norm of matrix $S$. 
To solve \eqref{min_pb}, an approach consists in constructing its Augmented Lagrange Multiplier (ALM) \cite{Hestenes}. By incorporating the constraints of \eqref{min_pb} into the objective multiplied by their Lagrange multiplier $Y\in\mathbb{R}^{n\times n}$, the ALM is
\begin{equation}
\min_{L,S,Y} ||L||_{*}+ \delta ||S||_1 + \langle Y,\Sigma - L - S\rangle + \frac{\mu}{2}||\Sigma - L - S||^2_{F}\,,
\label{ALM}
\end{equation}
which coincides with the original problem when $\mu\to\infty$.
We denote by $\mathcal{L}_{\mu}(L,S,Y)$ the above objective function.
In \cite{Lin2010}, it is solved with an alternating direction method:
\begin{align}
L_{t+1} &= \arg\min_{L} \mathcal{L}_{\mu_t}(L,S_t,Y_t)\label{IALM_L}\\
S_{t+1} &= \arg\min_{S} \mathcal{L}_{\mu_t}(L_{t+1},S,Y_t)\label{IALM_S}\\
Y_{t+1} &= Y_t + \mu_t(\Sigma - L_{t+1} - S_{t+1})\nonumber\\
\mu_{t+1} &=\rho\mu_t.\nonumber
\end{align}
for some $\rho>1$. The resulting method is called Inexact ALM (IALM); in the Exact ALM, $\mathcal{L}_{\mu}(L,S,Y_t)$ is minimized on $L$ and $S$ simultaneously, a considerably more difficult task.
The problem \eqref{IALM_S} can be solved explicitly at modest cost. In contrast, \eqref{IALM_L} requires an expensive Singular Value Decomposition (SVD). In \cite{NonConvex_RPCA}, the authors replace the nuclear norm in  
\eqref{ALM} by the non-convex function $||L||_\gamma:=\sum_i(1+\gamma)\sigma_i(L)/(\gamma+\sigma_i(L))$ that interpolates between $\rank(L)$ and $||X||_*$ as $\gamma$ goes from $0$ to $1$, in order to depart from the convex PCP approximation of the original problem. Then they apply the same alternating direction method to the resulting function. This method is referred to as Non-Convex RCPA or NC-RCPA.
In \cite{Rodriguez2013}, the authors rather solve a variant of \eqref{min_pb} by incorporating the constraint into the objective, removing the costly nuclear norm term, and imposing a rank constraint on $L$: 
\begin{equation}
\min_{L,S} \delta ||S||_1 + \frac{1}{2}||\Sigma - L - S||^2_{F}\quad \st \rank(L)=k.
\label{constraints_rank}
\end{equation}
We denote by $\widehat{\mathcal{L}}(L,S)$ the above objective function. Using also an alternating direction strategy, the authors have devised the Fast PCP (FPCP) method as
\begin{align}
L_{t+1} &= \arg\min_{L} \widehat{\mathcal{L}}(L,S_t) \ \ \st \rank(L)=k \label{L_k1}\\
S_{t+1} &= \arg\min_{S} \widehat{\mathcal{L}}(L_{t+1},S). \label{S_k1}
\end{align}
The problem \eqref{S_k1} is easy to solve as for \eqref{IALM_S}. In contrast to \eqref{IALM_L}, the sub-problem \eqref{L_k1} can be solved by applying a faster partial SVD to $\Sigma - S_t$, with the only necessity of computing the $k$ first singular values and their associated eigenvectors.  These authors have further improved their algorithm in \cite{Rodriguez2016} with an accelerated variant of the partial SVD algorithm. Their methods are considered as state-of-the-art in the field  \cite{bouwmans2015} and their solution is of comparable quality to the one of \eqref{min_pb}. An alternative approach, designated here as (RPCA-GD) to solve \eqref{constraints_rank} was proposed in \cite{Yi2016RPCA_GS}, where, as in our setting, the rank constraint is enforced by setting $L:=MM^T$ for $M\in\R^{n\times k}$. Then a projected gradient method is used to solve \eqref{L_k1} in $M$. In order to guarantee that $S$ has a prescribed sparsity, they use an ad-hoc projector on an appropriate space of sparse matrices.}

{The solution to the PCP problem \eqref{min_pb} depends on the hyperparameter $\delta$, from which we cannot infer the value of the rank $k$ of the resulting optimal $L$ with more accuracy than in the classical results of Cand\`es et al. in Theorem~1.1 of \cite{candes2009}. In view of the particular financial application we have in mind, we would prefer a method for which we can chose the desired $k$ in advance. In both the IALM and the Non-Convex RPCA methods, neither the rank of $L$ nor its expressivity --- that is, the portion of the spectrum of $\Sigma$ covered by the low-rank matrix --- can be chosen in advance. In RPCA-GD, the rank is chosen in advance, but the sparsity of $S$ is set in advance by inspection of the given matrix $\Sigma$, a limitation that we would particularly like to avoid in our application. From this point of view, FPCP seems the most appropriate method.} In our approach, we must first select a rank for $L$, based e.g.\ on a prior spectral decomposition of $\Sigma$ or based on exogenous considerations. We then apply a gradient descent method with a well-chosen loss function, using Tensorflow \cite{tensorflow2015-whitepaper} or Pytorch {\cite{paszke2017automatic}.
	
	In Section~\ref{sec:main}, we introduce the construction of our low-rank matrix $L=MM^T$, where $M$, {in contrast with the RPCA-GD method,} is parametrized by the coefficients of a
	multi-layered neural network. {A potential advantage of this parametrization has been pointed in \cite{lopez-paz2018easing}, albeit in a different context than ours.} We also describe the corresponding loss function that we seek to minimize. Moreover, we analyze the regularity properties of the objective function leading to an estimate of the convergence rate of a standard gradient descent method to a stationary point of the method; see Theorem~\ref{thm:convergence}.
	We show that the convergence rate of our algorithm  grows polynomially in the dimension of each layer.  In Section~\ref{sec:numerics}, we conduct a series of experiments first on artificially generated data, that is matrices $\Sigma$ with a given decomposition $L+S$, to assess the accuracy of our method and to compare it with the algorithms presented in this section.
	We apply also our algorithm to real data sets: the correlation matrix of the stocks from the S\&P500 index and an estimate of the correlation matrix between real estate market indices of 44 countries. We show that our method achieves a higher accuracy than the other algorithms.
	Moreover, by its construction of $L:=MM^T$, we can guarantee the positive semidefiniteness of $L$, although empirical covariance matrices tend to not satisfy this property. We refer to \cite{higham2016bounds} for a detailed discussion of this issue.
	{ By contrast, most algorithms do not ensure that $L$ is kept positive semidefinite, which forces them to correct their output at the expense of their accuracy.} 
	  The proof of Theorem~\ref{thm:convergence} is provided in Section~\ref{sec:proof}.
	%
	%
	\section{Neural network parametrized optimization and its convergence rate}\label{sec:main}
	
	Let $\mathbb{S}^{n}$ be the set of $n$-by-$n$ real symmetric matrices and $\mathbb{S}^{n}_+\subset\mathbb{S}^{n}$ be the cone of positive semidefinite matrices.
	We are given a matrix $\Sigma=\left[\Sigma_{i,j}\right]_{i,j}\in\mathbb{S}^{n}$ and a number $1\le k\le n$. Our task is to decompose $\Sigma$
	as the sum of $L=\left[L_{i,j}\right]_{i,j}\in\mathbb{S}^{n}_+$
	of rank at most equal $k$ and a sparse matrix $S=\left[S_{i,j}\right]_{i,j}\in\mathbb{R}^{n\times n}$ in some optimal way.
	Observe that the matrix $S$ is also a symmetric matrix. It is well-known
	that the matrix $L$ can be represented as $L=MM^{T}$, where $M=\left[M_{i,j}\right]_{i,j}\in\mathbb{R}^{n\times k}$;
	thus $\Sigma=MM^{T}+S$.
	
	For practical purposes, we shall represent every symmetric $n$-by-$n$
	matrix by a vector of dimension $r:=n(n+1)/2$; formally, we define
	the linear operator
\begingroup\makeatletter\def\f@size{7}\check@mathfonts
\def\maketag@@@#1{\hbox{\m@th\normalsize\normalfont#1}}
\begin{eqnarray}
	\label{eq:def-h}
		h :\mathbb{S}^{n}&\to&\mathbb{R}^{n(n+1)/2},\\
	\Sigma&\mapsto& h(\Sigma):=(\Sigma_{1,1},\dots,\Sigma_{1,n},\Sigma_{2,2},\dots,\Sigma_{2,n},\dots,\Sigma_{n,n})^{T}. \nonumber
	\end{eqnarray}
	\endgroup
	
	The operator $h$ is obviously invertible, and its inverse shall be denoted by $h^{-1}$. Similarly, every  
	vector of dimension  $nk$ shall be represented by a $n$-by-$k$
	matrix
	by
	the linear operator $g:\mathbb{R}^{nk} \to \mathbb{R}^{n\times k}$, which maps $v = (v_1,\ldots,v_{nk})^T$ to $g(v):=W\equiv[w_{i,j}]_{i,j}$ with $w_{i,j}=v_{(i-1)k+j}$, in a kind of row-after-row filling of $W$ by $v$. This operator has clearly an inverse $g^{-1}$.
	
	We construct a neural network with $n(n+1)/2$ inputs and $nk$
	outputs; these outputs are meant to represent the coefficients of
	the matrix $M$ with whom we shall construct the rank $k$ matrix
	$L$ in the decomposition of the input matrix~$\Sigma$. However,
	we do not use this neural network in its feed-forward mode as a heuristic
	to compute $M$ from an input $\Sigma$; we merely use the neural
	network framework as a way to parametrize a tentative solution $M$
	to our decomposition problem.
	
	We construct our neural
	network with $m$ layers of respectively $\ell_1,\ldots,\ell_m$ neurons, each with the same
	activation function $\sigma:\mathbb{R}\to[-1,1]$. We assume that the
	first and the second derivative of $\sigma$ are uniformly bounded
	from above by the constants $\sigma'_{\textrm{max}}>0$ and $\sigma''_{\textrm{max}}$,
	respectively. We let $\ell_0:=n(n+1)/2$ and $\ell_{m+1}:=nk$. In accordance with the standard architecture of fully connected multi-layered neural networks, for $0\leq u\leq m$ we let
	$A^{(u)}=\left[A^{(u)}_{i,j}\right]_{i,j}\in\mathbb{R}^{\ell_{u+1}\times \ell_{u}}$ be layer $u$'s weights,\\ $b^{(u)}=\left[b^{(u)}_{i}\right]_{i}\in\mathbb{R}^{\ell_{u+1}}$ be
	its bias, and for all $v\in\R^{\ell_{u}}$
	\begin{equation*}
	f_{u}^{A^{(u)},b^{(u)}}(v):=A^{(u)}v+b^{(u)}.
	\end{equation*}
	For each $1\leq i\leq m$ and each $v\in \mathbb{R}^{\ell_i}$, we write $\sigma^{(i)}(v):=\left(\sigma(v_{1}),\ldots,\sigma(v_{\ell_i})\right)^{T}$.
	We 
	denote  the parameters $\Theta:=(A^{(0)},b^{(0)},\dots,A^{(m)},b^{(m)})$ and
	define the $m$-layered neural network by the function $\mathcal{N}^{\Theta}_m$ from $\R^{\ell_0}$ to $\R^{\ell_{m+1}}$ for which $\mathcal{N}^{\Theta}_m(x)$ equals
	\begingroup\makeatletter\def\f@size{9}\check@mathfonts
	\def\maketag@@@#1{\hbox{\m@th\normalsize\normalfont#1}}
	\begin{equation*}
	\begin{split}
	f_{m}^{A^{(m)},b^{(m)}} \!\circ\! \sigma^{(m)} \circ f_{m-1}^{A^{(m-1)},b^{(m-1)}} \!\circ\! \dots \!\circ\! \sigma^{(1)}\!\circ\! f_{0}^{A^{(0)},b^{(0)}}(x).
	\end{split}
	\end{equation*}
	\endgroup
	We therefore have to specify $\mathbb{L}_m:=\sum_{u=0}^{m} \ell_u \ell_{u+1} + \ell_{u+1}$ many parameters  to
	describe the neural network $\mathcal{N}^{\Theta}_m$ completely.
	
	Now, we are ready to define the cost function to minimize. We write the $1$-norm of some $X \in \R^{n\times n}$ as $||X||_{1}:=\sum_{i,j=1}^{n}|X_{i,j}|$.
	Our objective function is, for a given $\Sigma\in\mathbb{S}^{n}$, the function
	\begin{equation*}
	\varphi_{\mathrm{obj}}(\Theta):= \left\Vert g\big(\mathcal{N}^{\Theta}_m\left(h(\Sigma)\right)\!\big)\,g\big(\mathcal{N}^{\Theta}_m\left(h(\Sigma)\right)\!\big)^{T}-\Sigma\right\Vert _{1}.
	\end{equation*}
	Since $M=g\big(\mathcal{N}^{\Theta}_m\left(h(\Sigma)\right)\!\big)$ is
	our tentative solution to the matrix decomposition problem, this objective
	function consists in minimizing $||MM^{T}-\Sigma||_{1}=||S||_{1}$ over the parameters $\Theta$ that define $M$.
	
	As the function $\Theta \mapsto \varphi_{\mathrm{obj}}(\Theta)$  is neither differentiable nor convex, we do not have access neither to its gradient nor to a subgradient. We shall thus approximate it by
	\begingroup\makeatletter\def\f@size{7}\check@mathfonts
	\def\maketag@@@#1{\hbox{\m@th\normalsize\normalfont#1}}
	\begin{equation}
{\varphi}(\Theta) :=  \sum_{i,j=1}^{n}\mu\left(\left[g\big(\mathcal{N}^{\Theta}_m\left(h(\Sigma)\right)\!\big)\,g\big(\mathcal{N}^{\Theta}_m\left(h(\Sigma)\right)\!\big)^{T}-\Sigma\right]_{i,j}\right),\label{eq:phi}
	\end{equation}
	\endgroup
	where $\mu:\mathbb{R}\to[0,\infty)$ is a smooth approximation of
	the absolute value function with a derivative uniformly bounded by
	$1$ and its second derivative bounded by $\mu''_{\max}$. A widely
	used example of such a function is given by
	\[
	\mu(t):=
	\begin{cases}
	\frac{t^{2}}{2\varepsilon}+\frac{\varepsilon}{2} & \textrm{if }|t|\le\varepsilon\\
	|t| & \textrm{if }|t|>\varepsilon,
	\end{cases}
	\]
	where $\varepsilon$ is a small positive constant. With this choice
	for $\mu$, we have $\mu''_{\max}=1/\varepsilon$. Another example,
	coming from the theory of smoothing techniques in convex optimization \cite{nesterov2005smooth},
	is given by $\mu(t):=\varepsilon\ln(2\cosh(t/\varepsilon)),$ also
	with $\mu''_{\max}=1/\varepsilon$.

	We apply a gradient method to minimize the objective function ${\varphi}$, whose general scheme can be written as follows.	
	\begin{equation}
	\label{gradient-method}
	\begin{split}
	\begin{array}{l}
	\textrm{Fix}\;\Theta_{0}\\
	\textrm{For}\;j\ge0:\\
	\quad\quad\textrm{Compute }\nabla{\varphi}(\Theta_{j})\\
	\quad\quad\textrm{Determine a step-size } h_{j}>0\\
	\quad\quad\text{Set}\;\Theta_{j+1}=\Theta_{j}-h_{j}\nabla{\varphi}(\Theta_{j}).
	\end{array}
	\end{split}
	\end{equation}
	
	
	The norm we shall use in the sequel is a natural extension of the standard
	Frobenius norm to finite lists of matrices of diverse sizes (the Frobenius norm of a vector coinciding with its Euclidean norm). Specifically,
	for any $\gamma\in\mathbb{N}_0$, $m_{1},\ldots,m_{\gamma},n_{1},\ldots,n_{\gamma}\in\mathbb{N}_{0}$,
	and $(X^{1},\ldots,X^{\gamma})\in\mathbb{R}^{m_{1}\times n_{1}}\times\dots\times\mathbb{R}^{m_{\gamma}\times n_{\gamma}}$,
	we let 
	\begin{equation}\label{eq:norm_definition}
	\left\Vert \left(X^{1},\dots,X^{\gamma}\right)\right\Vert :=\left(||X^{1}||_F^{2}+\dots+||X^{\gamma}||_F^{2}\right)^{1/2}.
	\end{equation}
	This norm is merely the standard Euclidean norm of the vector obtained
	by concatenating all the columns of $X^{1},\ldots,X^{\gamma}$. 
	
	Since the objective function in \eqref{eq:phi} is non-convex, this method can only realistically
	converge to one of its stationary point or to stop close enough from
	one, that is, at point $\Theta^{*}$ for which $||\nabla\varphi(\Theta^*)||$
	is smaller than a given tolerance. The complexity of many variants
	of this method can be established if the function ${\varphi}$
	has a Lipschitz continuous gradient (see \cite{ghadimi2016accelerated} and references therein).
	We have the following convergence result.
	\begin{thm}\label{thm:convergence}
		Let $\Sigma \in \mathbb S^n$ and assume that there exists $D>0$  
		such that the sequence $(\Theta_j)_{j\in \N_0}$ of parameters constructed in \eqref{gradient-method} satisfies
		\begin{equation}\label{ass-D}
		\sup_{j \in \N_0}\|\Theta_j\|\leq D.
		\end{equation}
		Then, the gradient of the function $\varphi$ defined in \eqref{eq:phi} is Lipschitz continuous on $\mathcal{D}:=\{\Theta \in \R^{\mathbb{L}_m}\colon \|\Theta \| \leq D\}$ with Lipschitz constant bound $L_m$ satisfying for $m\ge 1$:
\begingroup\makeatletter\def\f@size{7}\check@mathfonts	\def\maketag@@@#1{\hbox{\m@th\normalsize\normalfont#1}} 
	\begin{align*}
    L^2_m = \mathcal{C}_m\max\Big\{kn^3D^4\ell_{\max},&nD^{4m+2}\max\{\ell_{\max},||h(\Sigma)||^2)\}\times\\
    & \max\{nD^2||h(\Sigma)||^2,nD^2\ell_{\max}
    \Big\},
	\end{align*}
    \endgroup   
		where $\mathcal{C}_m$ is a constant that only depends polynomially on  $\sigma'_{\max},\sigma''_{\max},\mu''_{\max}$, with powers in $O(m)$ and $\ell_{\max}:=\max\{\ell_1,\ldots,\ell_m\}$.
		When $m=1$, we have an alternative Lipschitz constant bound, more favourable for large $D$:
	\begingroup\makeatletter\def\f@size{7}\check@mathfonts	\def\maketag@@@#1{\hbox{\m@th\normalsize\normalfont#1}}
	\begin{equation*}
	\begin{split}
	\hat L^2_{1}
	=
	\mathcal{C}_{1} n^2\ell_1
	\max\!\bigg\{& D^4 ||h(\Sigma)||^4, \ell_1 D^8||h(\Sigma)||^4,\\
	&\ell_1^2D^6||h(\Sigma)||^2,n D^2||h(\Sigma)||^2,nk\ell_1,\ell_1^3D^4\bigg\}.
	\end{split}
	\end{equation*}
	\endgroup		
		As a consequence, if for the gradient method \eqref{gradient-method}  there exists a constant $K>0$ such that for all $j\geq 0$ 
		\begin{equation}\label{condition-gradient}
		\varphi(\Theta_j)-\varphi(\Theta_{j+1})\geq \tfrac{K}{L_m} \| \nabla \varphi (\Theta_j)\|^2,
		\end{equation}
		then for every $N \in \N$ we have that
		\begin{equation}\label{eq:thm-conv}
		\min_{0\leq j \leq N} \| \nabla \varphi (\Theta_j)\| \leq \tfrac{1}{\sqrt{N+1}}\Big[\tfrac{L_m}{K} \big(\varphi(\Theta_0)-\varphi^*\big)\Big]^{1/2}, 
		\end{equation}
		where $\varphi^*:=\min_{\Theta \in  \mathcal{D}} \varphi(\Theta)$.
		In particular, for every tolerance level $\varepsilon>0$ we have
		\begin{equation*}
		N+1 \geq \tfrac{L_m}{K\varepsilon^2} \big(\varphi(\Theta_0)-\varphi^*\big) \ \ \Longrightarrow \ \ \min_{0\leq j \leq N} \| \nabla \varphi (\Theta_j)\| \leq \varepsilon.
		\end{equation*}
	\end{thm}
 	\begin{rem}[Choosing $(h_j)$]\label{ex:gradient-method}
		Notice that the condition \eqref{condition-gradient} in Theorem~\ref{thm:convergence} imposed on the gradient method, or more precisely on the step-size strategy $(h_j)$,  is not very restrictive. We provide several examples that are frequently used.
		\begin{enumerate}
			\item The sequence $(h_j)$ is chosen in advance, independently of the minimization problem. This includes, e.g., the common  constant step-size strategy $h_j=h$ or $h_j=\tfrac{h}{\sqrt{j+1}}$  for some constant $h>0$. With these choices, one can show that \eqref{condition-gradient} is satisfied for $K=1$.
			%
			%
			%
			%
			\item The Goldstein-Armijo rule, in which, given  $0<\alpha<\beta<1$, one needs to find $(h_j)$ such that
			\begin{equation*}
			\begin{split}
			\alpha \big\langle \nabla \varphi(\Theta_j), \Theta_j-\Theta_{j+1} \big \rangle
			&\leq \varphi(\Theta_j)-\varphi(\Theta_{j+1})\\
			\beta \big\langle \nabla \varphi(\Theta_j), \Theta_j-\Theta_{j+1} \big \rangle
			&\geq \varphi(\Theta_j)-\varphi(\Theta_{j+1}).
			\end{split}
			\end{equation*}
			This strategy satisfies \eqref{condition-gradient}  with $K=2\alpha(1-\beta)$.
		\end{enumerate}
		We refer to \cite[Section~1.2.3]{nesterov2013introductory} and to \cite[Chapter 3]{Nocedal_Wright} for further details and more examples.
	\end{rem}
	\begin{rem}[On Assumption \eqref{ass-D}]
		The convergence rate \eqref{eq:thm-conv} obtained in Theorem~\ref{thm:convergence} relies fundamentally on the Lipschitz property of the gradient of the (approximated) objective function $\varphi$ of the algorithm in \eqref{gradient-method}. However, due to its structure, we see that the global Lipschitz property  of $\nabla \varphi$ fails already for a single-layered neural network, as it grows polynomially of degree~4 in the parameters; see also Section~\ref{sec:proof}. Yet, it is enough to ensure  the Lipschitz property of $\nabla \varphi$ on the domain of the sequence of parameters $(\Theta_j)_{j \in \N_0}$ generated by the algorithm in \eqref{gradient-method}, which explains the significance of assumption \eqref{ass-D}. Nevertheless, assumption \eqref{ass-D} is not very restrictive as one might expect that the algorithm  \eqref{gradient-method} converges and hence automatically forces assumption \eqref{ass-D} to hold true.
		Empirically, we verify in Subsection~\ref{subsec:bound} that this assumption holds for our algorithm when used with our two non-synthetic data sets.
	\end{rem}
	\begin{rem}[Polynomiality of our method]\label{rem:curse_dimensionality}
		While the second part of Theorem~\ref{thm:convergence} is standard in optimization (see, e.g., in \cite[Section~1.2.3]{nesterov2013introductory}), we notice that for a fixed depth $m$ of the neural network the constant $L$ in the rate of convergence of the sequence $(\min_{1\leq j \leq N} \| \nabla \varphi (\Theta_j)\|)$ grows polynomially in the parameters $\ell_{\max}:=\max\{\ell_1,\dots,\ell_{m}\}$, $n$, and $k$. These parameters describe the dimensions of the input, the output, and the hidden layers of the neural network. A rough estimate shows that
		\begin{equation*}
		L_m \leq O\left(k^{1/2}n^{3/2}D^{2m+2}\ell_{\max}||h(\Sigma)||^2\right).
		\end{equation*}
	\end{rem}
	\begin{rem}[A simplified version of \eqref{eq:thm-conv}]
		Note that, since the function $\varphi$ is nonnegative, we have $\varphi^*\geq 0$, so that \eqref{eq:thm-conv} can be simplified by
		\begin{equation*}
		\min_{0\leq j \leq N} \| \nabla \varphi (\Theta_j)\| \leq \tfrac{1}{\sqrt{N+1}}\left(\tfrac{L_m}{K} \varphi(\Theta_0)\right)^{1/2}\,.
		\end{equation*}
	\end{rem}
	%
\begin{rem}[Choice of activation function]
We require the activation function $\sigma\colon \R \to [-1,1]$ to be a non-constant, bounded, and smooth function. The following table provides the most common activation functions satisfying the above conditions.
\begin{table}[h!]
	\begin{center}
		\label{tab:activationFct}
		\begin{tabular}{c|c|c|c} 
			\textbf{Name} & \textbf{Definition} & $\boldsymbol{\sigma'_{\max}}$ & $\boldsymbol{\sigma''_{\max}}$
				 \\
			\hline 
		Logistic  & $\frac{1}{1+e^{-x}}$ & 0.25& $\tfrac{1}{6\sqrt{3}}$\\
			Hyperbolic tangent & $\frac{e^{2x}-1}{e^{2x}+1}$ & 1& $\tfrac{4}{3\sqrt{3}}$\\		(Scaled)	$\arctan$  & $\frac{2}{\pi}\tan^{-1}(x)$ & $\tfrac{2}{\pi}$& $\tfrac{3\sqrt{3}}{4\pi}$ \\
		\end{tabular}
 \caption{Choice of activation function $\sigma$.}
	\end{center}
\end{table}
\end{rem}	
	We provide the proof of Theorem~\ref{thm:convergence} in Section~\ref{sec:proof}.
%
%
\section{Numerical results}\label{sec:numerics}

\subsection{Numerical results based on simulated data}

We start our numerical tests\footnote{We refer to  \url{https://github.com/CalypsoHerrera/lrs-cov-nn} for the codes of our simulations.}
 with a series of experiments on artificially generated data.
 We construct a collection of $n$-by-$n$ positive semidefinite matrices $\Sigma$ that can be written as $\Sigma = L_0+S_0$ for a known matrix $L_0$ of rank $k_0\leq n$ and a known matrix $S_0$ of given sparsity $s_0$. {We understand by \emph{sparsity} the number of null elements of $S_0$ divided by the number of coefficients of $S_0$; when a sparse matrix is determined by an algorithm, we consider that every component smaller in absolute value than $\epsilon=0.01$ is null.}
To construct one matrix $L_0$, we first sample $nk_0$ independent standard normal random  variables that we arrange into an $n$-by-$k_0$ matrix $M$. Then $L_0$ is simply taken as $MM^T$. To construct a symmetric positive semidefinite sparse matrix $S_0$ 
 we first select uniformly randomly a pair $(i,j)$ with $1\leq i<j\leq n$. We then construct an $n$-by-$n$ matrix $A$ that has only four non-zero coefficients{: its} off-diagonal elements $(i,j)$ and $(j,i)$ are set to a number $b$ drawn uniformly randomly in $[-1,1]$, whereas the diagonal elements $(i,i)$ and $(j,j)$ are set to a number $a$ drawn uniformly randomly in $[|b|,1]$. This way, the matrix $A$ is positive semidefinite. The matrix $S_0$ is obtained by summing such matrices $A$, each corresponding to a different pair $(i,j)$, until the desired sparsity $s_0$ is reached.

Given an  artificially generated matrix $\Sigma = L_0+S_0$, where $L_0$ has a prescribed rank {$k_0$} and $S_0$ {a sparsity $s_0$},
we run our algorithm to construct a matrix $M\in\R^{n\times k}$. With $L:=MM^T$ and $S:=\Sigma-L$, we determine {the \emph{approximated rank}} $r(L)$ of $L$ by counting the number of eigenvalues of $L$ that are larger than $\epsilon = 0.01$. We also determine the sparsity $s(S)$ {as specified above, by taking as null every coefficient smaller than $\epsilon = 0.01$ in absolute value}.
 We compute the discrepancy between the calculated low-rank part $L$ and the correct one $L_0$ by $\textrm{rel.error}(L):=||L-L_0||_F/||L_0||_F$ and between $S$ and the true $S_0$ by $\textrm{rel.error}(S):=||S-S_0||_F/||S_0||_F$.
Table~\ref{table} reports the average of these quantities over ten runs of {our algorithm DNN (short for Deep Neural Network)}, as well as their standard deviation (in parenthesis). {We carried our experiments on  various values for the dimension $n$ of the matrix $\Sigma$, for the given rank $k_0$ of $L_0$, for the given sparsity $s_0$ of $S_0$ and for the chosen forced (upper bound for the) rank $k$ in the construction of $L$ introduced in Section~\ref{sec:main}. }

When choosing  $n=100$, our algorithm unsurprisingly achieves the maximal rank $k$ for the output matrix $L$, unlike IALM, Non-Convex RPCA, and FPCP.
The sparsities are comparable when $n=100$, even though the different methods have different strategies to sparsify their matrices; some methods, such as FPCP, apply a shrinkage after optimization by replacing every matrix entries $S_{ij}$ by $\sign(S_{ij})[|S_{ij}|-1/\sqrt{n}]^+$. By forcing sparsity, FPCP makes its output violate Equation \eqref{S_k1}, so the sparsity of its output and ours might not be comparable. We however display $s(S)$ and its relative error $\textrm{rel.error}(S)$ even for that method, to give some insight on how the algorithms behave.

When the given rank $k_0$ 
coincides the forced rank $k$, our algorithm achieves a much higher accuracy than
all the other algorithms for small $n$. When $n$ gets larger, our algorithm still compares favourably with the other ones, except on a few outlier instances. Of course, when there is a discrepancy between $k_0$ and $k$, our algorithm cannot recover $k_0$. Nevertheless, the relative error $\textrm{rel.error}(L)$ compares favorably with the other methods, especially when $k_0>k$. We acknowledge however that in circumstances where one needs to minimize the rank of $L$, e.g., to avoid overfitting past data, the forced rank $k$ can only be considered as the maximum rank of $L$ returned by the algorithm; in such case, the PCP, IALM, and FPCP algorithms could be more appropriate.

Various network architectures 
 and corresponding activation functions 
have been tested. They only marginally influence the numerical results.
%
%
\begin{table}[h]
	\begin{center}
		\resizebox{8.2cm}{!}{
		\begin{tabular}{| c c c c c | r r r r |}
			\toprule
			$n$ & $k_0$ & $k$ & $s_0$ & Algorithm &         $r(L)\quad\;\;$ &       $s(S)\quad\;\;$ & $\textrm{rel.error}(S)$ & $\textrm{rel.error}(L)$ \\
			\midrule
			\multirow{30}{*}{100} & \multirow{30}{*}{10} & \multirow{10}{*}{5} & \multirow{5}{*}{0.60} & IALM &   11.00 (0.45) &  0.17 (0.02) &  0.98 (0.00) &  0.93 (0.01) \\
			&    &    &      & NC-RCPA &   10.00 (0.00) &  1.00 (0.00) &  1.00 (0.00) &  0.95 (0.01) \\
			&    &    &      & RPCA-GD &    5.00 (0.00) &  0.01 (0.00) &  0.99 (0.00) &  0.94 (0.01) \\
			&    &    &      & FPCP &    5.00 (0.00) &  0.04 (0.00) &  1.01 (0.01) &  0.97 (0.01) \\
			&    &    &      & \cellcolor{light-gray}DNN & \cellcolor{light-gray}   5.00 (0.00) & \cellcolor{light-gray} 0.01 (0.00) & \cellcolor{light-gray} 0.36 (0.02) & \cellcolor{light-gray} 0.53 (0.02) \\
			\cline{4-9}
			&    &    & \multirow{5}{*}{0.95} & IALM &   11.30 (0.64) &  0.16 (0.01) &  2.21 (0.16) &  0.28 (0.02) \\
			&    &    &      & NC-RCPA &   10.00 (0.00) &  1.00 (0.00) &  1.00 (0.00) &  0.13 (0.00) \\
			&    &    &      & RPCA-GD &    5.00 (0.00) &  0.01 (0.00) &  4.18 (0.15) &  0.53 (0.02) \\
			&    &    &      & FPCP &    5.00 (0.00) &  0.04 (0.00) &  6.41 (0.21) &  0.84 (0.01) \\
			&    &    &      & \cellcolor{light-gray}DNN & \cellcolor{light-gray}   5.00 (0.00) & \cellcolor{light-gray} 0.01 (0.00) & \cellcolor{light-gray} 4.35 (0.18) & \cellcolor{light-gray} 0.54 (0.02) \\
			\cline{3-9}
			\cline{4-9}
			&    & \multirow{10}{*}{10} & \multirow{5}{*}{0.60} & IALM &   11.00 (0.45) &  0.17 (0.02) &  0.98 (0.00) &  0.93 (0.01) \\
			&    &    &      & NC-RCPA &   10.00 (0.00) &  1.00 (0.00) &  1.00 (0.00) &  0.95 (0.01) \\
			&    &    &      & RPCA-GD &   10.00 (0.00) &  0.03 (0.01) &  0.99 (0.00) &  0.94 (0.01) \\
			&    &    &      & FPCP &    9.00 (0.00) &  0.07 (0.01) &  0.99 (0.01) &  0.95 (0.01) \\
			&    &    &      & \cellcolor{light-gray}DNN & \cellcolor{light-gray}  10.00 (0.00) & \cellcolor{light-gray} 0.06 (0.02) & \cellcolor{light-gray} 0.03 (0.01) & \cellcolor{light-gray} 0.04 (0.01) \\
			\cline{4-9}
			&    &    & \multirow{5}{*}{0.95} &  IALM &   11.30 (0.64) &  0.16 (0.01) &  2.21 (0.16) &  0.28 (0.02) \\
			&    &    &      & NC-RCPA &   10.00 (0.00) &  1.00 (0.00) &  1.00 (0.00) &  0.13 (0.00) \\
			&    &    &      & RPCA-GD &   10.00 (0.00) &  0.09 (0.18) &  1.87 (0.38) &  0.24 (0.05) \\
			&    &    &      & FPCP &    9.00 (0.00) &  0.07 (0.00) &  5.64 (0.20) &  0.74 (0.01) \\
			&    &    &      & \cellcolor{light-gray}DNN & \cellcolor{light-gray}  10.00 (0.00) & \cellcolor{light-gray} 0.16 (0.05) & \cellcolor{light-gray} 0.14 (0.05) & \cellcolor{light-gray} 0.02 (0.01) \\
			\cline{3-9}
			\cline{4-9}
			&    & \multirow{10}{*}{20} & \multirow{5}{*}{0.60} & IALM &   11.00 (0.45) &  0.17 (0.02) &  0.98 (0.00) &  0.93 (0.01) \\
			&    &    &      & NC-RCPA &   10.00 (0.00) &  1.00 (0.00) &  1.00 (0.00) &  0.95 (0.01) \\
			&    &    &      & RPCA-GD &   20.00 (0.00) &  0.06 (0.01) &  0.97 (0.00) &  0.92 (0.01) \\
			&    &    &      & FPCP &   12.00 (0.00) &  0.19 (0.02) &  0.97 (0.00) &  0.93 (0.01) \\
			&    &    &      & \cellcolor{light-gray}DNN & \cellcolor{light-gray}  20.00 (0.00) & \cellcolor{light-gray} 0.06 (0.01) & \cellcolor{light-gray} 0.05 (0.01) & \cellcolor{light-gray} 0.08 (0.01) \\
			\cline{4-9}
			&    &    & \multirow{5}{*}{0.95} & IALM &   11.30 (0.64) &  0.16 (0.01) &  2.21 (0.16) &  0.28 (0.02) \\
			&    &    &      & NC-RCPA &   10.00 (0.00) &  1.00 (0.00) &  1.00 (0.00) &  0.13 (0.00) \\
			&    &    &      & RPCA-GD &   20.00 (0.00) &  0.06 (0.01) &  2.03 (0.18) &  0.26 (0.02) \\
			&    &    &      & FPCP &   12.00 (0.00) &  0.19 (0.01) &  3.99 (0.19) &  0.53 (0.01) \\
			&    &    &      & \cellcolor{light-gray}DNN & \cellcolor{light-gray}  20.00 (0.00) & \cellcolor{light-gray} 0.17 (0.04) & \cellcolor{light-gray} 0.32 (0.08) & \cellcolor{light-gray} 0.04 (0.01) \\
			\cline{1-9}
			\cline{2-9}
			\cline{3-9}
			\cline{4-9}
			\multirow{5}{*}{200} & \multirow{5}{*}{5} & \multirow{5}{*}{5} & \multirow{5}{*}{0.95} & IALM &    5.00 (0.00) &  1.00 (0.00) &  1.00 (0.00) &  0.27 (0.01) \\
			&    &    &      & NC-RCPA &    5.00 (0.00) &  1.00 (0.00) &  1.00 (0.00) &  0.27 (0.01) \\
			&    &    &      & RPCA-GD &    5.00 (0.00) &  0.99 (0.00) &  1.00 (0.00) &  0.27 (0.01) \\
			&    &    &      & FPCP &    5.00 (0.00) &  0.06 (0.00) &  2.90 (0.14) &  0.79 (0.01) \\
			&    &    &      & \cellcolor{light-gray}DNN & \cellcolor{light-gray}   5.00 (0.00) & \cellcolor{light-gray} 0.18 (0.02) & \cellcolor{light-gray} 0.08 (0.01) & \cellcolor{light-gray} 0.02 (0.00) \\
			\cline{1-9}
			\cline{2-9}
			\cline{3-9}
			\cline{4-9}
			\multirow{5}{*}{400} & \multirow{5}{*}{10} & \multirow{5}{*}{10} & \multirow{5}{*}{0.95} & IALM &   10.00 (0.00) &  1.00 (0.00) &  1.00 (0.00) &  0.27 (0.01) \\
			&    &    &      & NC-RCPA &   10.00 (0.00) &  1.00 (0.00) &  1.00 (0.00) &  0.27 (0.01) \\
			&    &    &      & RPCA-GD &   10.00 (0.00) &  0.99 (0.00) &  1.00 (0.00) &  0.27 (0.01) \\
			&    &    &      & FPCP &    7.50 (0.50) &  0.02 (0.00) &  3.13 (0.07) &  0.87 (0.00) \\
			&    &    &      & \cellcolor{light-gray}DNN & \cellcolor{light-gray}  10.00 (0.00) & \cellcolor{light-gray} 0.12 (0.05) & \cellcolor{light-gray} 0.08 (0.03) & \cellcolor{light-gray} 0.02 (0.01) \\
			\cline{1-9}
			\cline{2-9}
			\cline{3-9}
			\cline{4-9}
			\multirow{5}{*}{800} & \multirow{5}{*}{20} & \multirow{5}{*}{20} & \multirow{5}{*}{0.95} & IALM &   20.00 (0.00) &  1.00 (0.00) &  1.00 (0.00) &  0.28 (0.00) \\
			&    &    &      & NC-RCPA &   20.00 (0.00) &  1.00 (0.00) &  1.00 (0.00) &  0.28 (0.00) \\
			&    &    &      & RPCA-GD &   20.00 (0.00) &  1.00 (0.00) &  1.00 (0.00) &  0.28 (0.00) \\
			&    &    &      & FPCP &    9.00 (0.00) &  0.01 (0.00) &  3.30 (0.03) &  0.92 (0.00) \\
			&    &    &      & \cellcolor{light-gray}DNN & \cellcolor{light-gray}  18.10 (5.70) & \cellcolor{light-gray} 0.04 (0.02) & \cellcolor{light-gray} 1.16 (1.59) & \cellcolor{light-gray} 0.32 (0.44) \\
			\cline{1-9}
			\cline{2-9}
			\cline{3-9}
			\cline{4-9}
		\end{tabular}}
		
	\end{center}
	\cellcolor{light-gray}\caption{Comparison between our method (DNN) and several other algorithms: IALM \cite{Lin2010}, Non-convex RCPA (NC-RCPA) \cite{NonConvex_RPCA}, RPCA-GD \cite{Yi2016RPCA_GS}, and FPCP \cite{Rodriguez2013}. The input parameters are the dimension $n$ of the matrix $\Sigma$ 
		the given rank $k_0$  of $L_0$,  the given sparsity $s_0$ of $S_0$, and  the chosen forced rank $k$ of $L$. We report the estimated rank $r(L)$ of the output matrix $L$, the estimated sparsity $s(S)$ of the output $S$, and their respective relative errors.
	}
	\label{table}
\end{table}
%
%
\subsection{Application on a five hundred S\&P500 stocks portfolio}
In this section, we evaluate our algorithm on real market data and compare it the other algorithms we have selected to demonstrate its capability also when the low-rank plus sparse matrix decomposition is not known. {A natural candidate for our experiment is the correlation matrix of stocks in the S\&P500, due to its relatively large size and the abundant, easily available data.}
Five hundred S\&P500 stocks were part of the index between 2017 and 2018. To make the representation {more readable}, we have sorted these stocks in eleven sectors according to the global industry classification standard\footnote{First, we have those belonging to the energy sector, second, those from the materials sector, then, in order, those from industrials, real estate, consumer discretionary, consumer staples, health care, financials, information technology, communication services, and finally utilities. We may notice that utilities seem almost uncorrelated to the other sectors, and that real estate and health care present a significantly lower level of correlation to the other sectors than the rest.}. We have constructed the correlation matrix $\Sigma$ from the daily returns of these 500 stocks during 250 consecutive trading days (see Figure~\ref{sp500}). As the data used to construct $\Sigma$ are available at an identical frequency, the matrix $\Sigma$ is indeed positive semidefinite, with 146 eigenvalues larger than $10^{-10}$. {The 70 largest eigenvalues account for 90\% of $\Sigma$'s trace, that is, the sum of all its 500 eigenvalues.}


In Figure~\ref{sp500}{,} we display the resulting matrices $L$ and $S$ for
all algorithms 
with respect to the same  input {$\Sigma$}. In our method and in RPCA-GD, we have set the rank of $L$ to $k=3$. Coincidentally, we have also obtained a rank of $3$ with FPCP. Among the three output matrices $L$, the one returned by our method matches the input more closely as it contains more relevant eigenspaces. The other algorithms have transferred this information to the sparse matrix $S$. {Note that the scale of values for the correlation matrix ranges between $-0.2$ and $0.9$.}
The ranks of $L$ obtained with the two other algorithms are $61$ for IALM and $5$ for Non-Convex RPCA, showing the difficulty of tuning these methods to obtain some desired rank.
The matrix $S$ from DNN would normally be slightly less sparse than the matrix $S$ of FPCP, as FPCP applies shrinkage. However, for visualization and comparison purposes, shrinkage is also applied on $S$ returned by DNN algorithm in Figure~\ref{sp500}.
%
\begin{center}
	\begin{figure}[h!]
		\begin{minipage}[c]{0.33\linewidth}%
			\centering \includegraphics[width=1\linewidth]{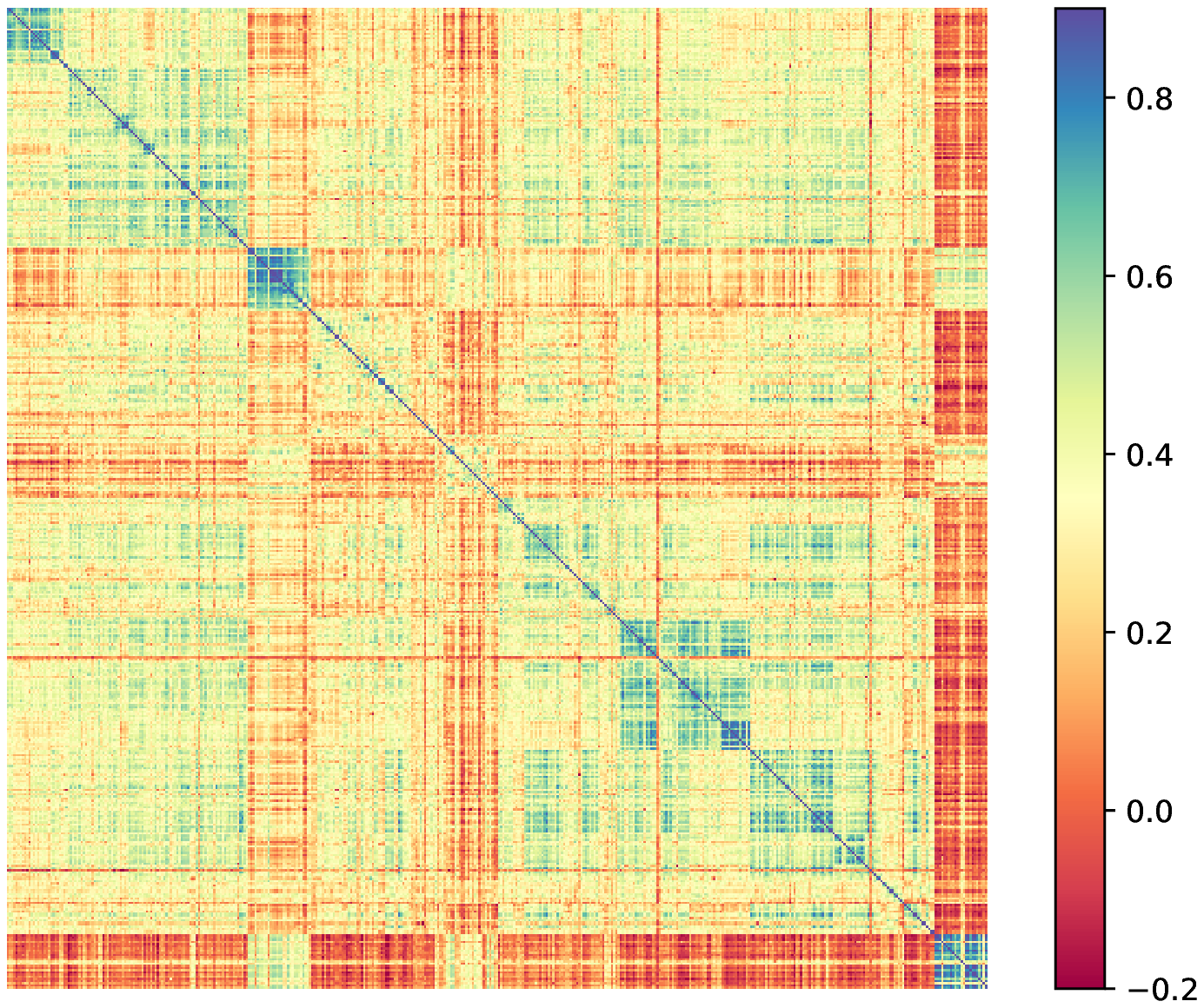}
			\vspace{-6mm}
			\caption*{{$\Sigma$}}
			\vspace{2mm}
			\includegraphics[width=1\linewidth]{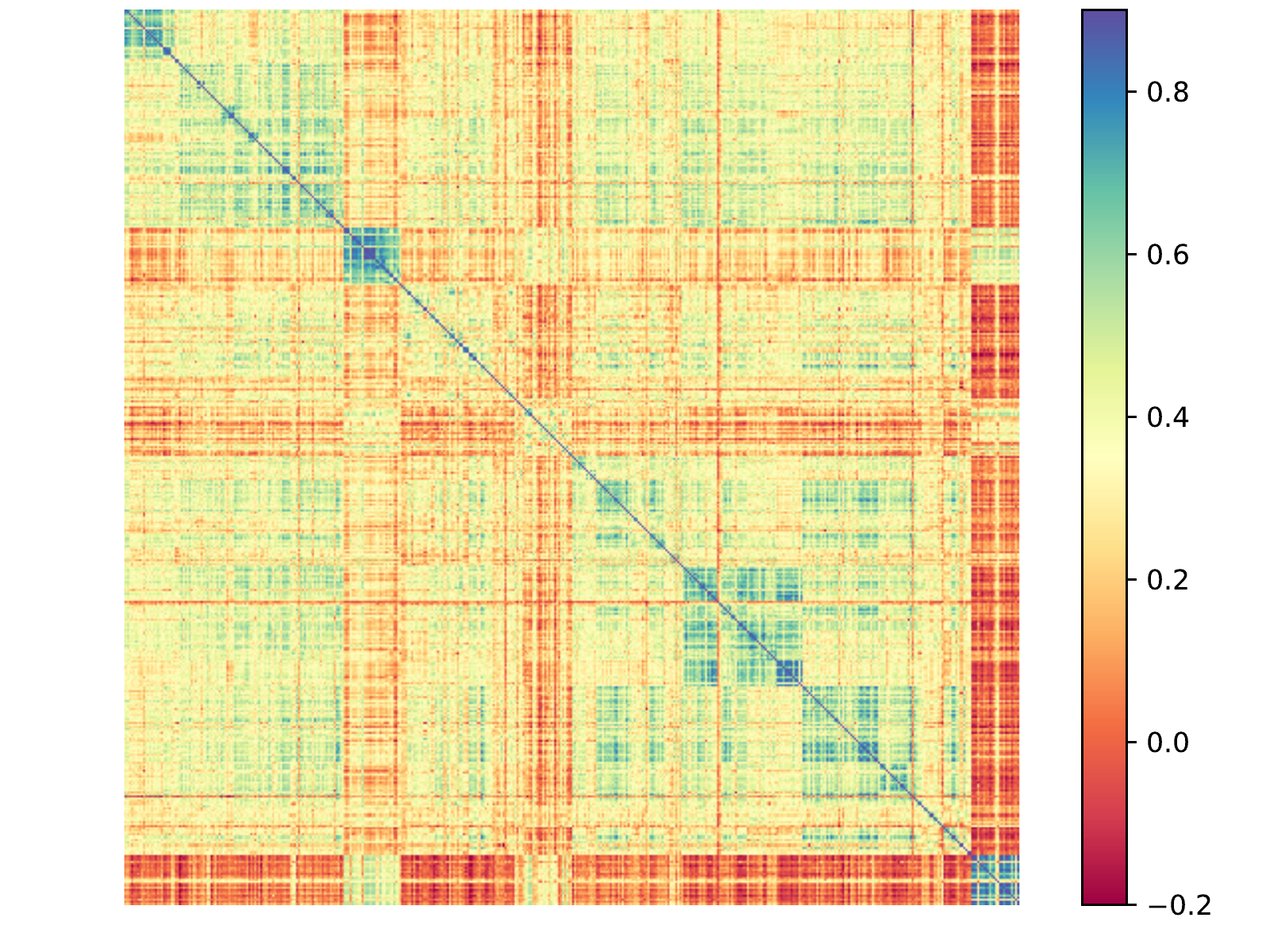} 
			\vspace{-6mm}
			\caption*{{$\Sigma$}}
			\vspace{2mm}
			\includegraphics[width=1\linewidth]{images/sp500/sp500_0} 
			\vspace{-6mm}
			\caption*{{$\Sigma$}}
			\vspace{2mm}
			\includegraphics[width=1\linewidth]{images/sp500/sp500_0} 
			\vspace{-6mm}
			\caption*{{$\Sigma$}}
			\vspace{2mm}
			\includegraphics[width=1\linewidth]{images/sp500/sp500_0} 
			\vspace{-6mm}
			\caption*{{$\Sigma$}}
			\vspace{2mm}
		\end{minipage}%
		\begin{minipage}[c]{0.33\linewidth}%
\centering \includegraphics[width=1\linewidth]{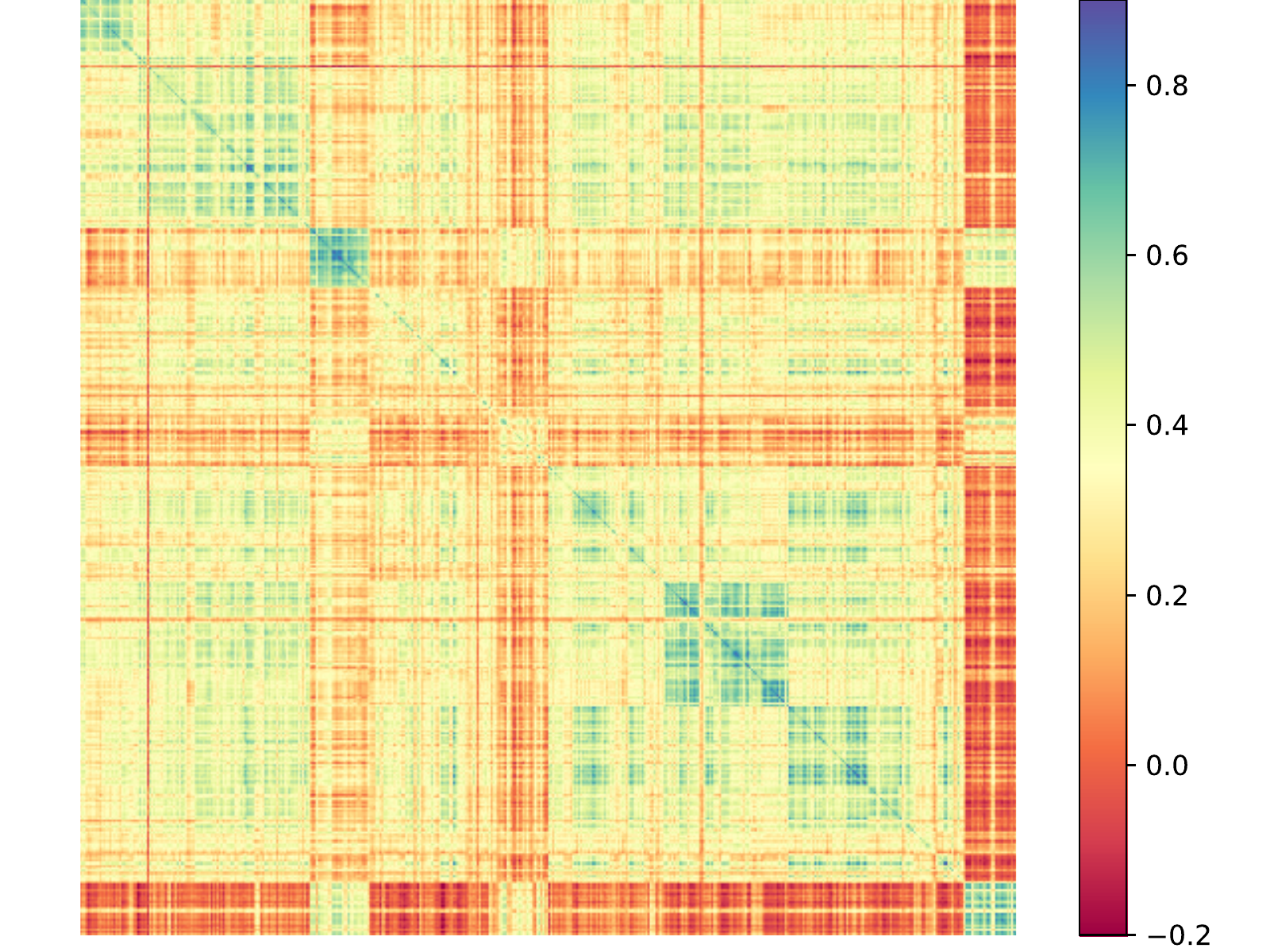}
\vspace{-6mm}
 \caption*{$L$ (IALM)}
\vspace{2mm}
 \includegraphics[width=1\linewidth]{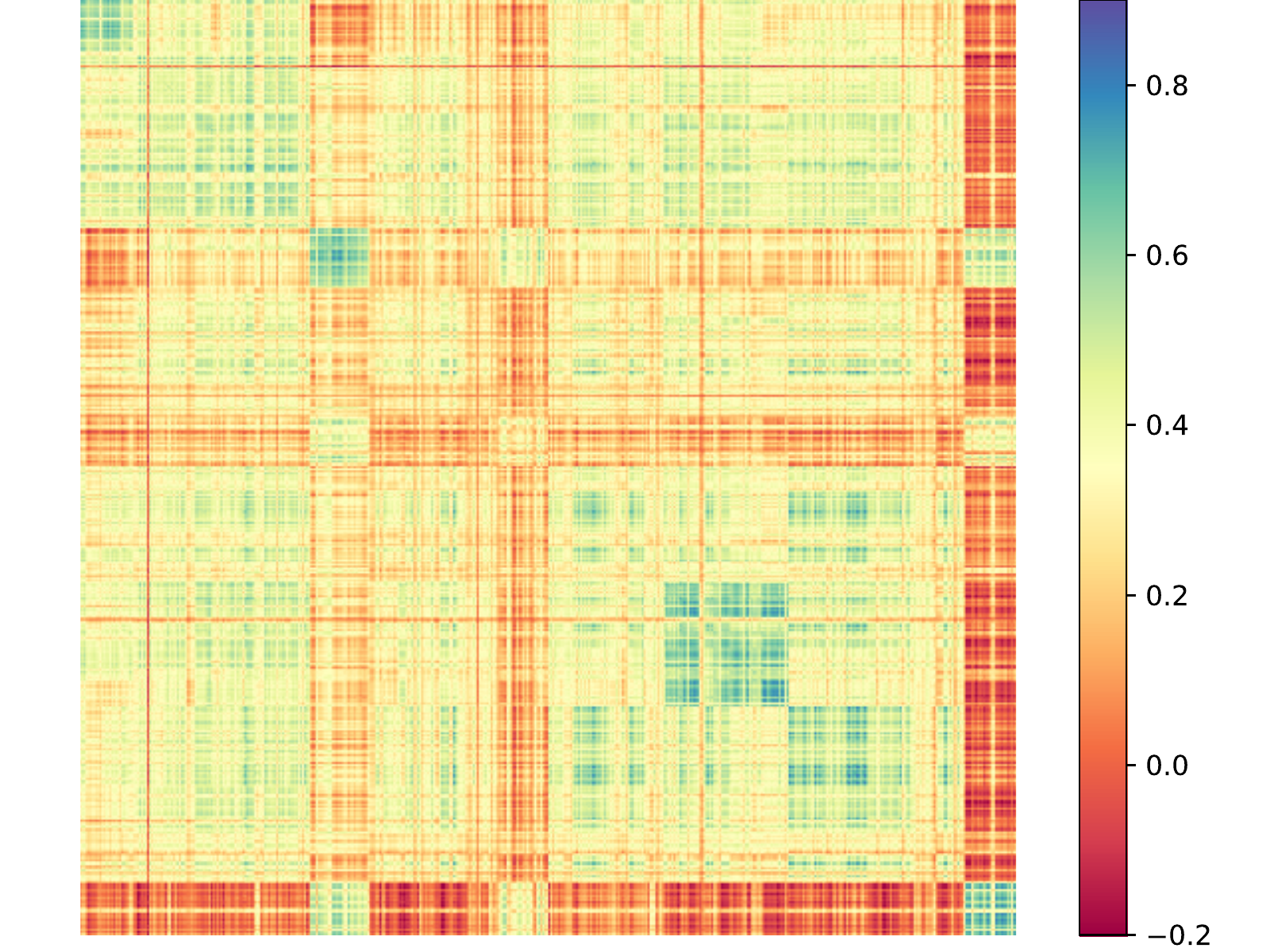} \vspace{-6mm}
 \caption*{$L$ (NC-RCPA)}
\vspace{2mm}
			\includegraphics[width=1\linewidth]{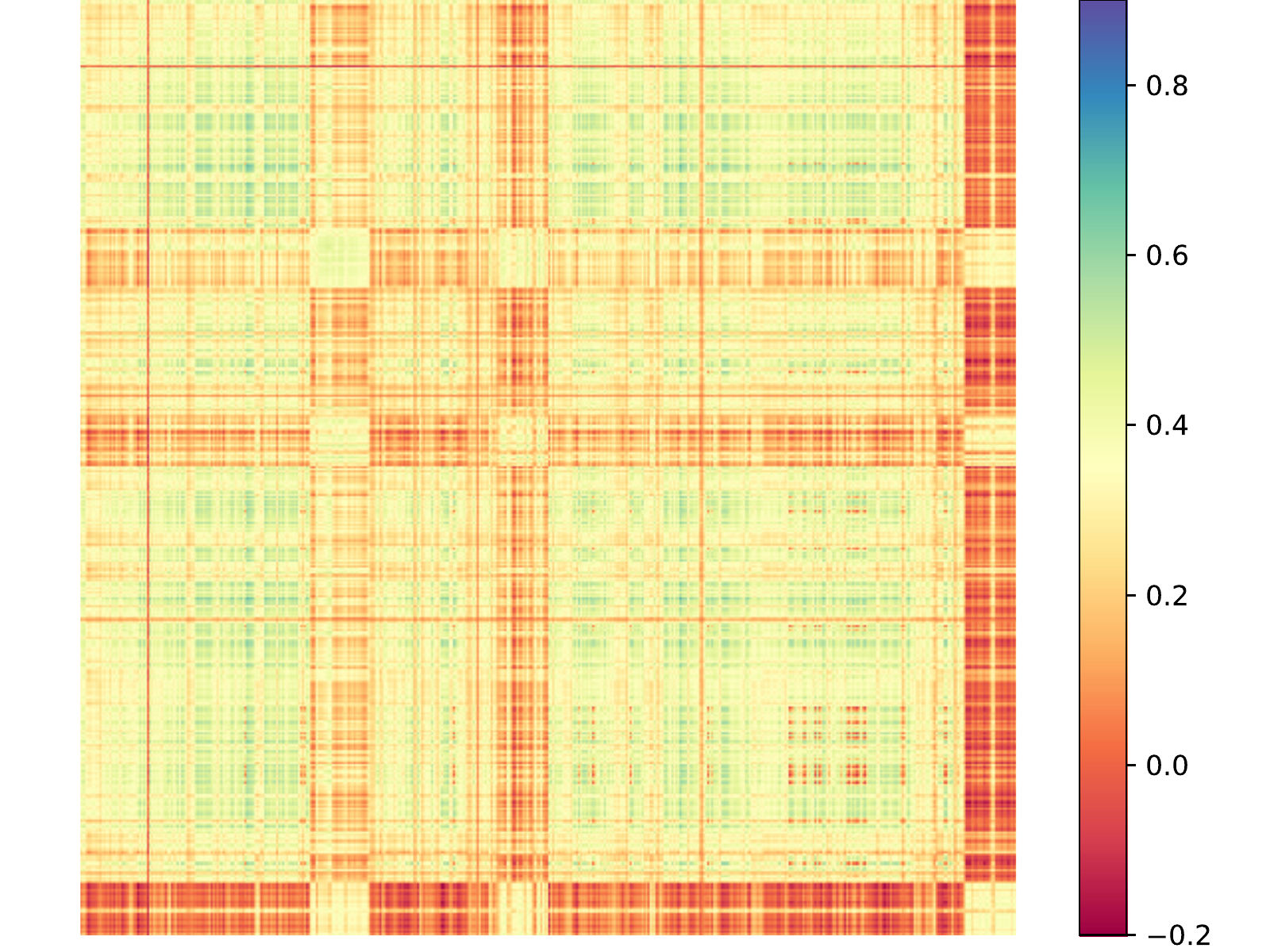} \vspace{-6mm}
			\caption*{$L$ (RPCA-GD)}
			\vspace{2mm}
			\includegraphics[width=1\linewidth]{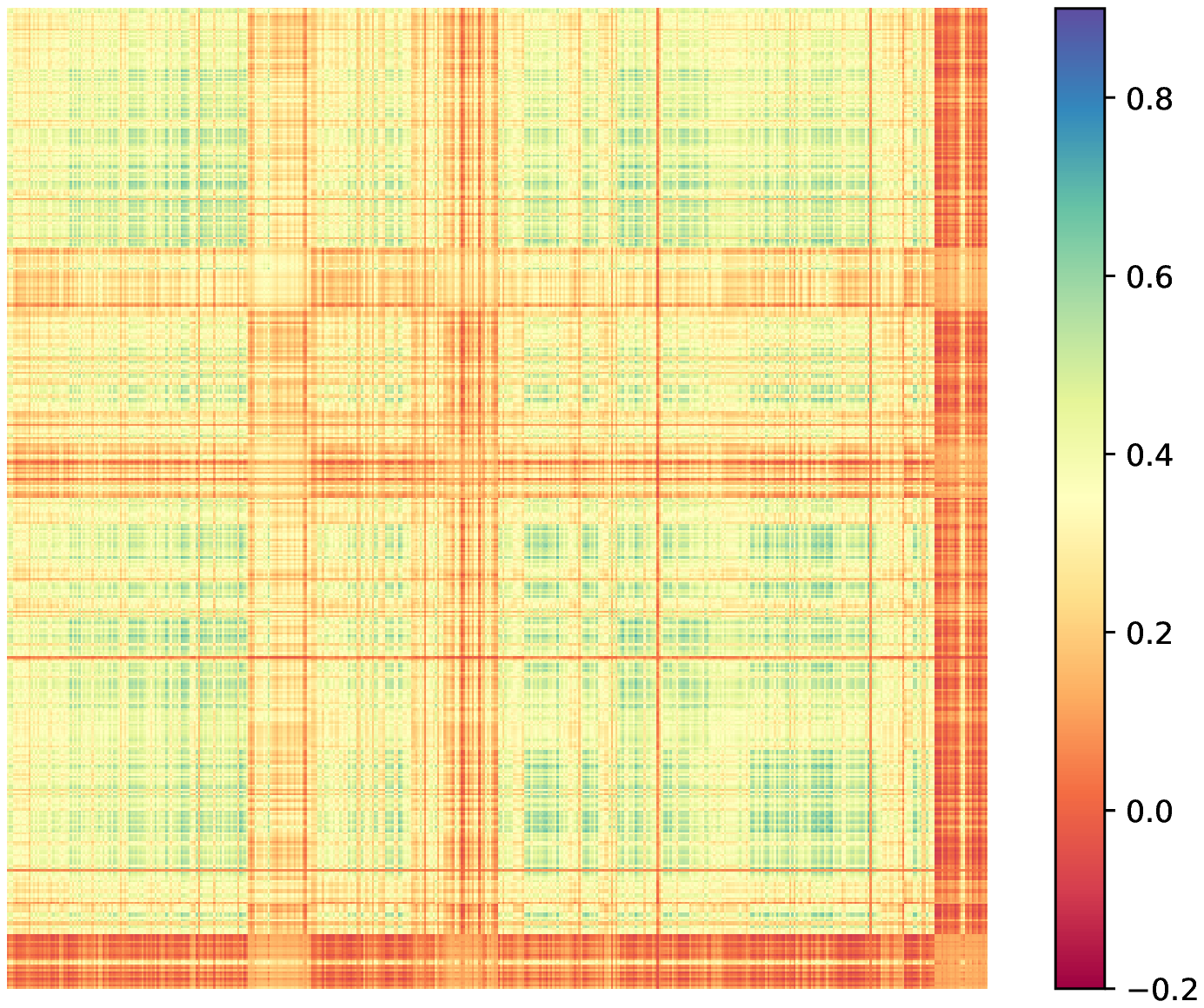} \vspace{-6mm}
			\caption*{$L$ (FPCP)}
			\vspace{2mm}
			\includegraphics[width=1\linewidth]{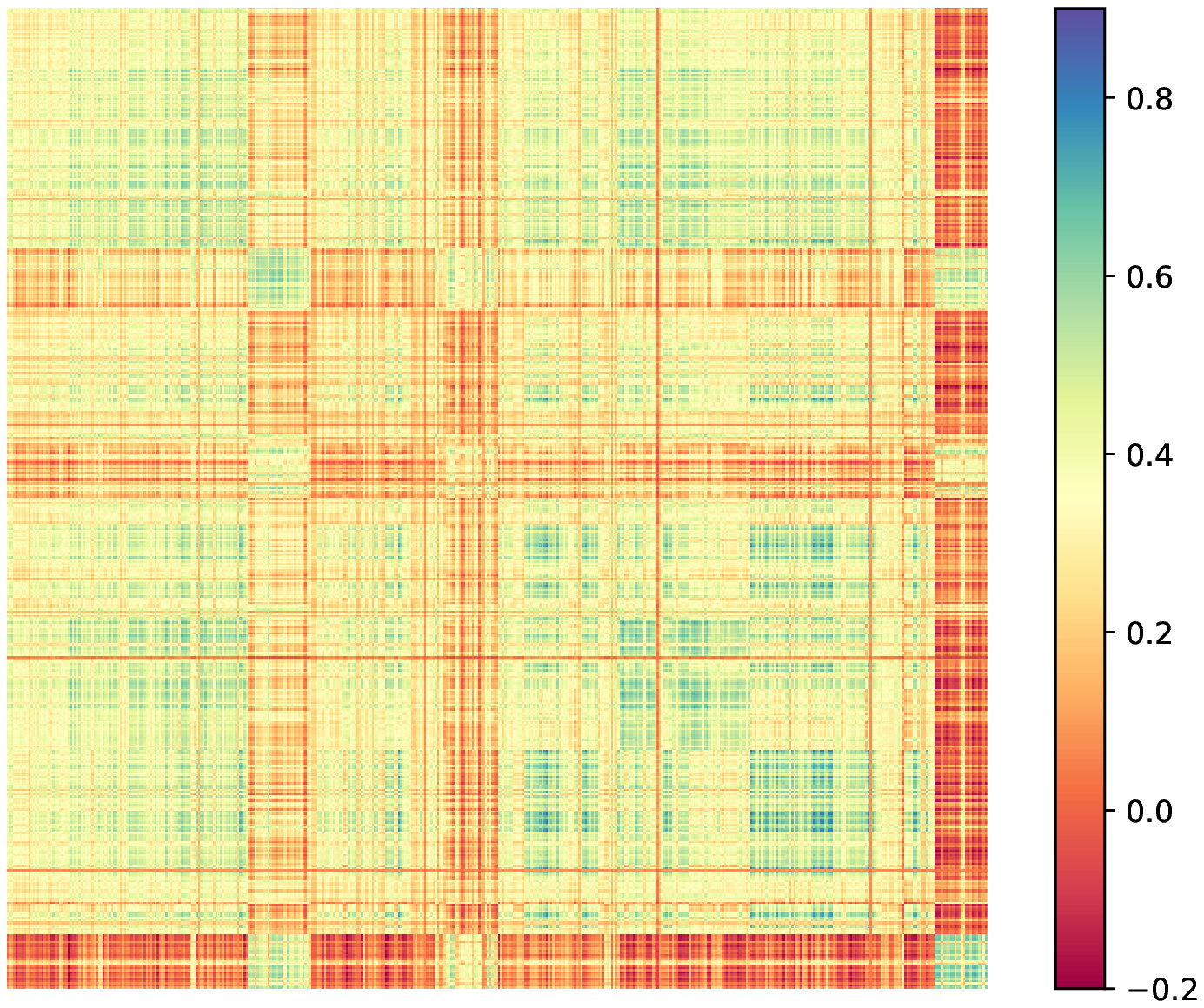} \vspace{-6mm}
			\caption*{$L$ (DNN)}
			\vspace{2mm}
		\end{minipage}%
		\begin{minipage}[c]{0.33\linewidth}%
			\centering
			\vspace{0.35cm} \includegraphics[width=1\linewidth]{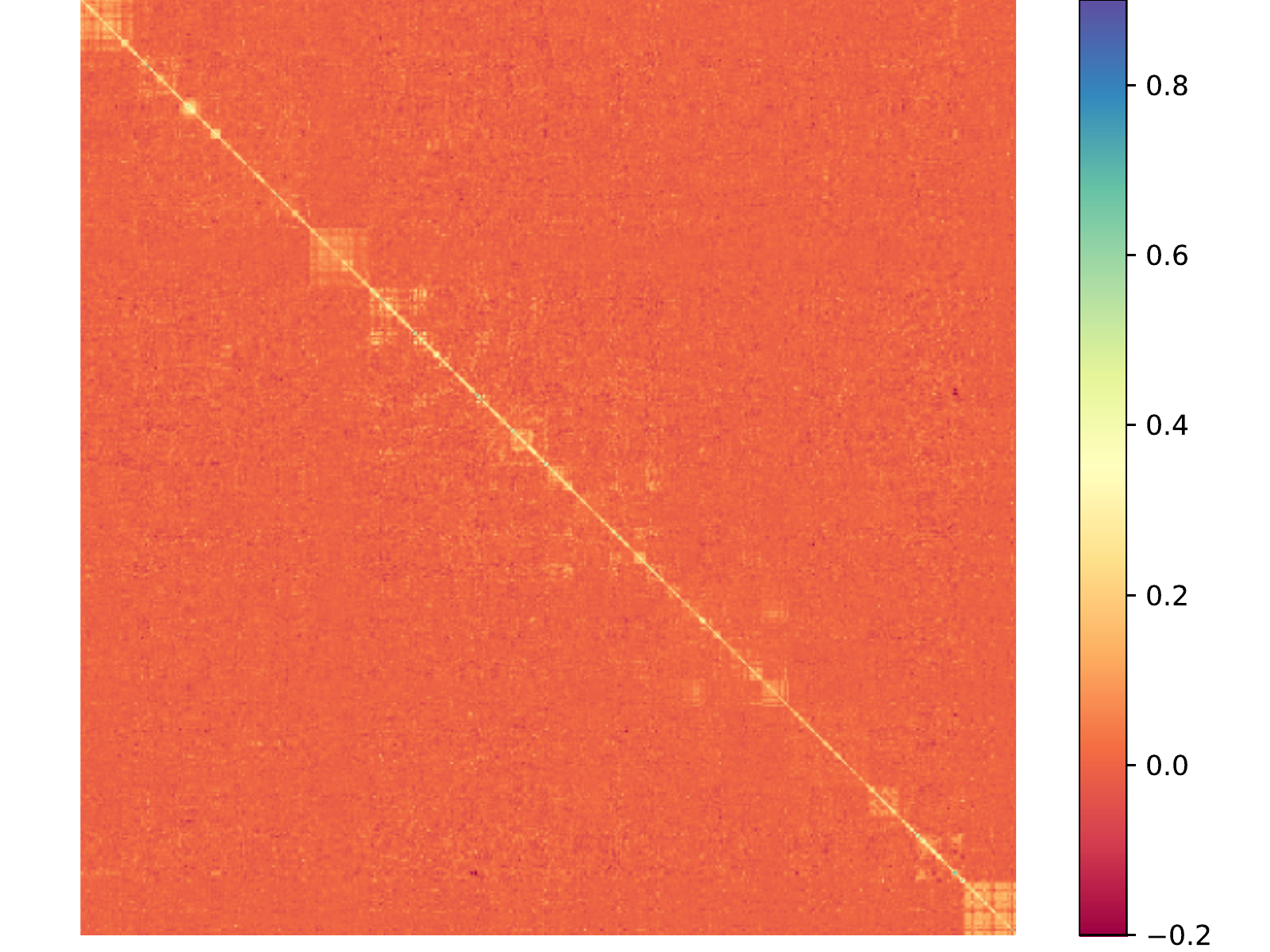}
			\vspace{-6mm}
			\caption*{$S$ (IALM)}
			\vspace{2mm}
			\includegraphics[width=1\linewidth]{images/sp500/sp500_0_S_ialm} \vspace{-6mm}
			\caption*{ {$S$ (NC-RCPA)}}
			\vspace{2mm}
			\includegraphics[width=1\linewidth]{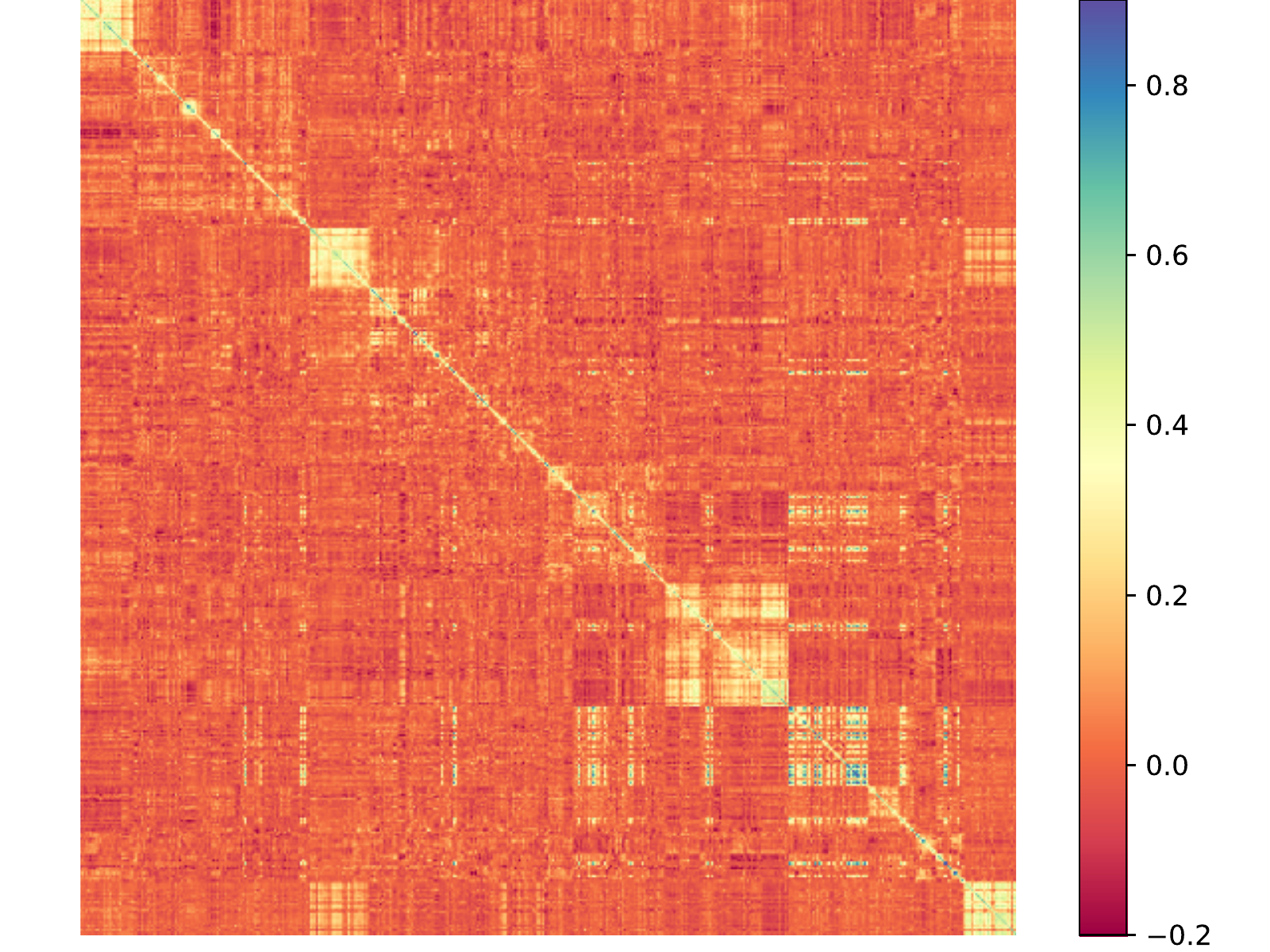} \vspace{-6mm}
			\caption*{ {$S$ (RPCA-GD)}}
			\vspace{2mm}
			\includegraphics[width=1\linewidth]{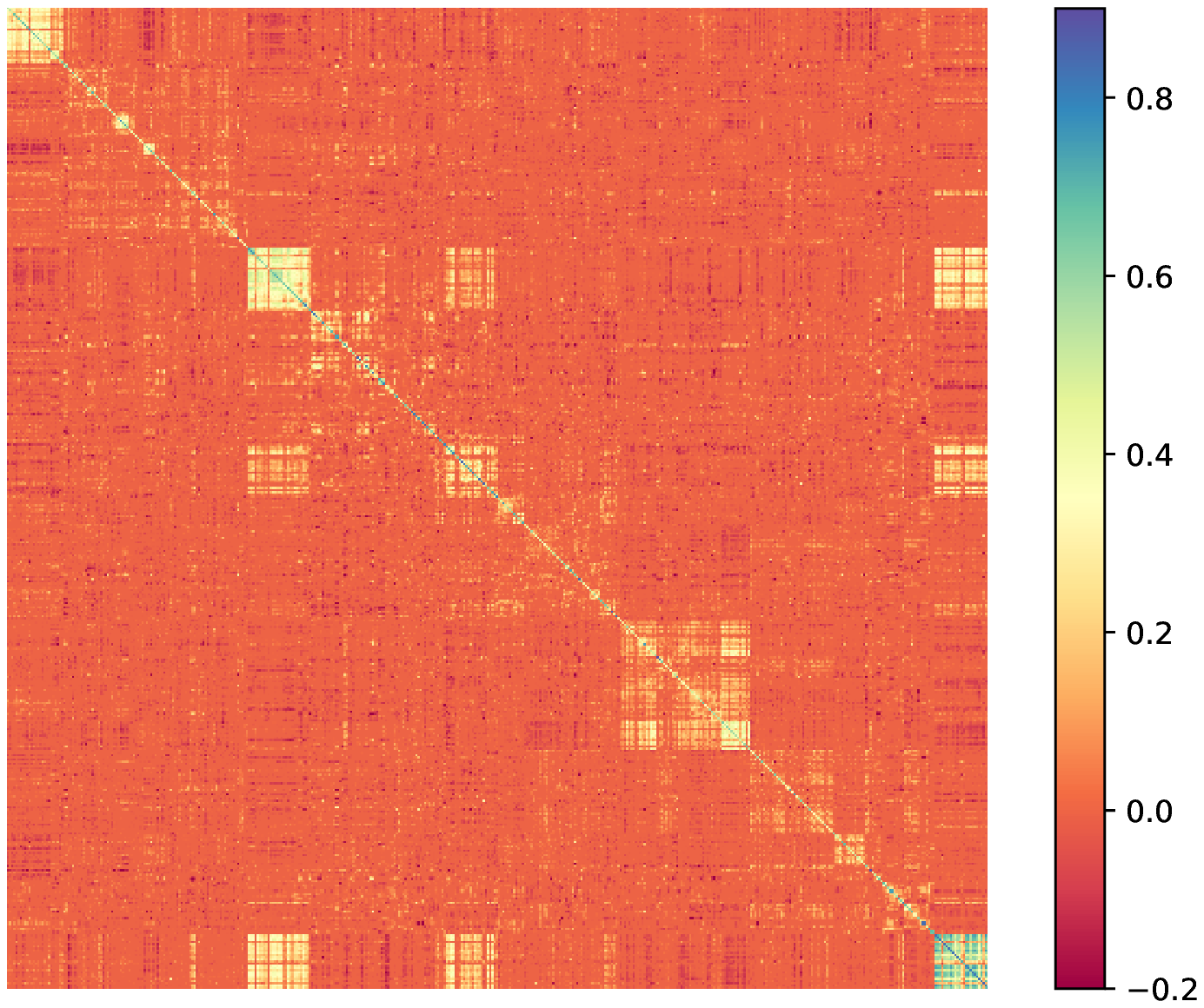} \vspace{-6mm}
			\caption*{ {$S$ (FPCP)}}
			\vspace{2mm}
			\includegraphics[width=1\linewidth]{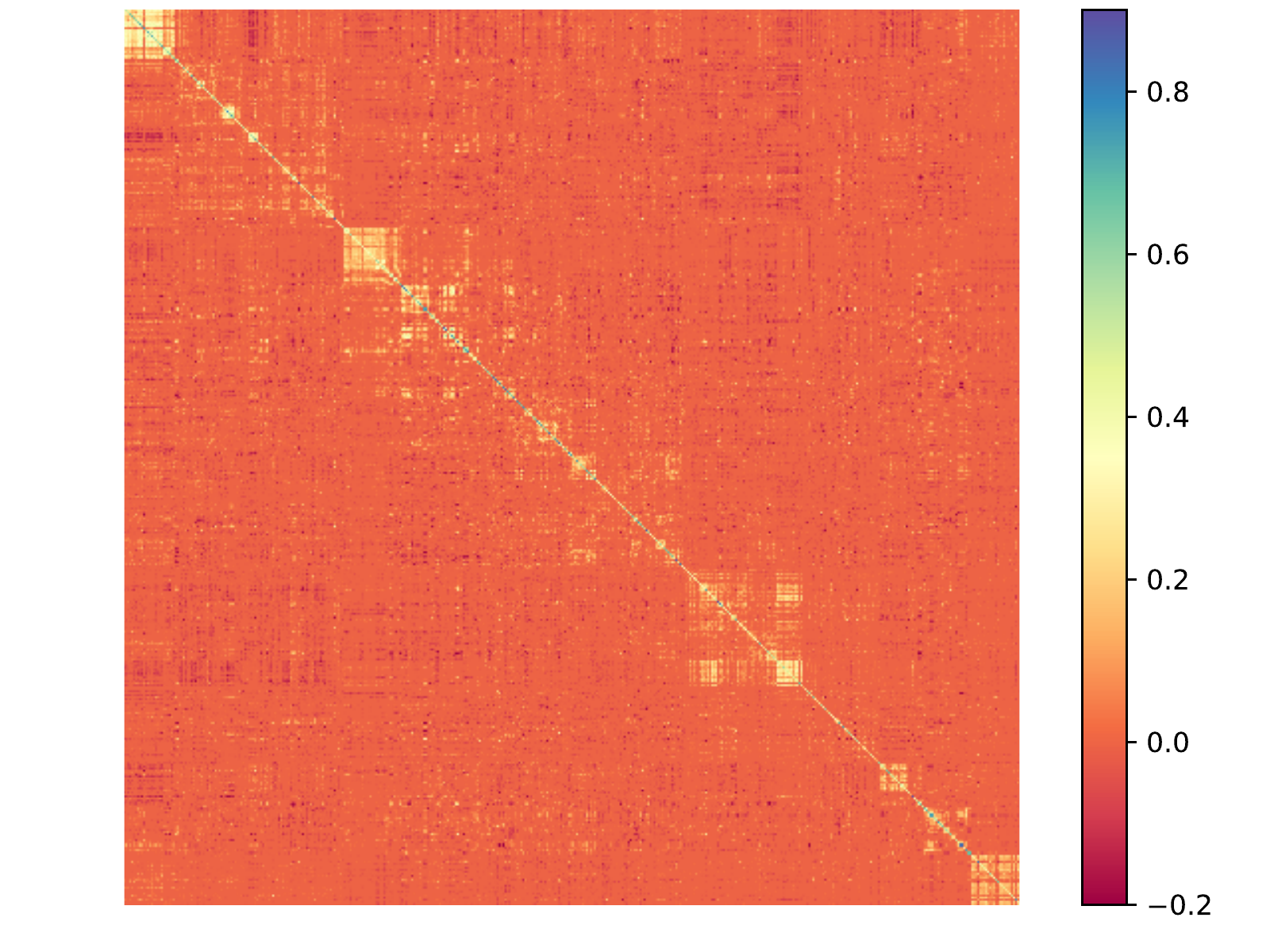} \vspace{-6mm}
			\caption*{ $\phantom{ii}$ $S$ (DNN),  $\phantom{ fd}$ \,  with shrinkage}
			\vspace{2mm}
		\end{minipage}
		\caption{Decomposition into a low-rank plus a sparse matrix of the correlation matrix of 500 stocks among the S\&P500 stocks. The forced rank is set to $k=3$. {For fair comparison with FPCP, shrinkage is applied on $S$ returned by the DNN algorithm. The obtained rank of the matrix $L$ are 61 for IALM, 5 for NC-RPCA, and 3 for the three other methods. We have $||\Sigma - L||_F/||L||_F$ at $0.16$ for DNN, and at $0.24$ for FPCP. The matrix $S$ has sparsity $0.69$ for DNN and $0.63$ for FPCP. The relative error of $L + S$ compared to $\Sigma$, where $S$ is subject to a shrinkage, is $0.076$ for DNN and $0.079$ for FPCP. When shrinkage is not applied, this error drops to $2.3 \cdot 10^{-8}$ for DNN.}}
		\label{sp500} 
	\end{figure}
\end{center}
%
\subsection{Application on real estate return}
We have computed the low-rank plus sparse decomposition of the real estate return matrix for $44$ countries\footnote{The countries are ordered by continent {and subcontinents}: {Western} Europe (Belgium, 
	Luxembourg, 
	Netherlands, 
	France, 
	Germany,  
	Switzerland, 
	Austria, 
	Denmark, 
	Norway, 
	Sweden, 
	Finland,
	United Kingdom, 
	Ireland, 
	Italy,
	Spain,
	Portugal), 
	{Eastern} Europe (Croatia, Estonia, 
	Latvia, 
	Lithuania, 
	Russia, 
	Poland, 
	Bulgaria, 
	Hungary, 
	Romania, 
	Slovak Republic, 
	Czech Republic), Near East (Turkey, Saudi Arabia), 
	{Southern} America (Brazil, Chile, 
	Colombia,
	Peru, 
	Mexico),
	{Northern} America (United States,
	Canada), {Eastern} Asia (India, China,
	Hong Kong, 
	Singapore,
	Japan)
	Oceania (Australia,
	New Zealand)
	and Africa (South Africa).}. The correlation matrix contains $88$ returns, alternating the residential returns and the corporate returns of each country\footnote{
	We thank Eric Schaanning for providing us this correlation matrix.}; see Figure~\ref{realestate}.
{We impose the rank of the output matrix to be equal to $3$.}
Similar to the previous section, the correlation color scale in Figure~\ref{realestate} is cropped between $-0.5$ and $0.5$ for a better visualization. The sparse matrix of FPCP is normally {sparser than the one returned by our DNN algorithm}. 
However in Figure~\ref{realestate}, and for a fair comparison, shrinkage is also applied to the matrix $S$ returned by our algorithm.
The low-rank matrices $L$ exhibit a variety of ranks from low (1 for NC-RPCA) to high (14 for IALM), indicating the difficulty of tuning the hyperparameters of these methods to a desired rank. 

In Figure~\ref{realestate_eigen}, we plot in the first line the eigenvalues of the matrix $L$ returned by FPCP and DNN, {as well as the eigenvalues of the original matrix $\Sigma$}, and for all the algorithms simultaneously in the second line. In the left figure, we {display} the first $17$ eigenvalues of the matrix $L$ where the forced rank is set to $k=15$. In the right figure, we plot the first $50$ eigenvalues where the forced rank is set to $k=88$. Notice that the input matrix $\Sigma$ has {some} \emph{negative} eigenvalues. {This phenomenon can happen in empirical correlation matrices when the data of the different variables are either not sampled over the same time frame or not with the same frequency}; we refer to \cite{higham2016bounds} for a further discussion on {this issue}. Our DNN algorithm and RPCA-GD, by setting $L:=MM^T$, avoid negative eigenvalues, although the original matrix $\Sigma$ is not {positive semidefinite}. {In contrast, the other algorithms might output a non-positive semidefinite matrix.}
\begin{center}
	\begin{figure}[h!]
		\begin{minipage}[c]{0.33\linewidth}%
			\centering \includegraphics[width=1\linewidth]{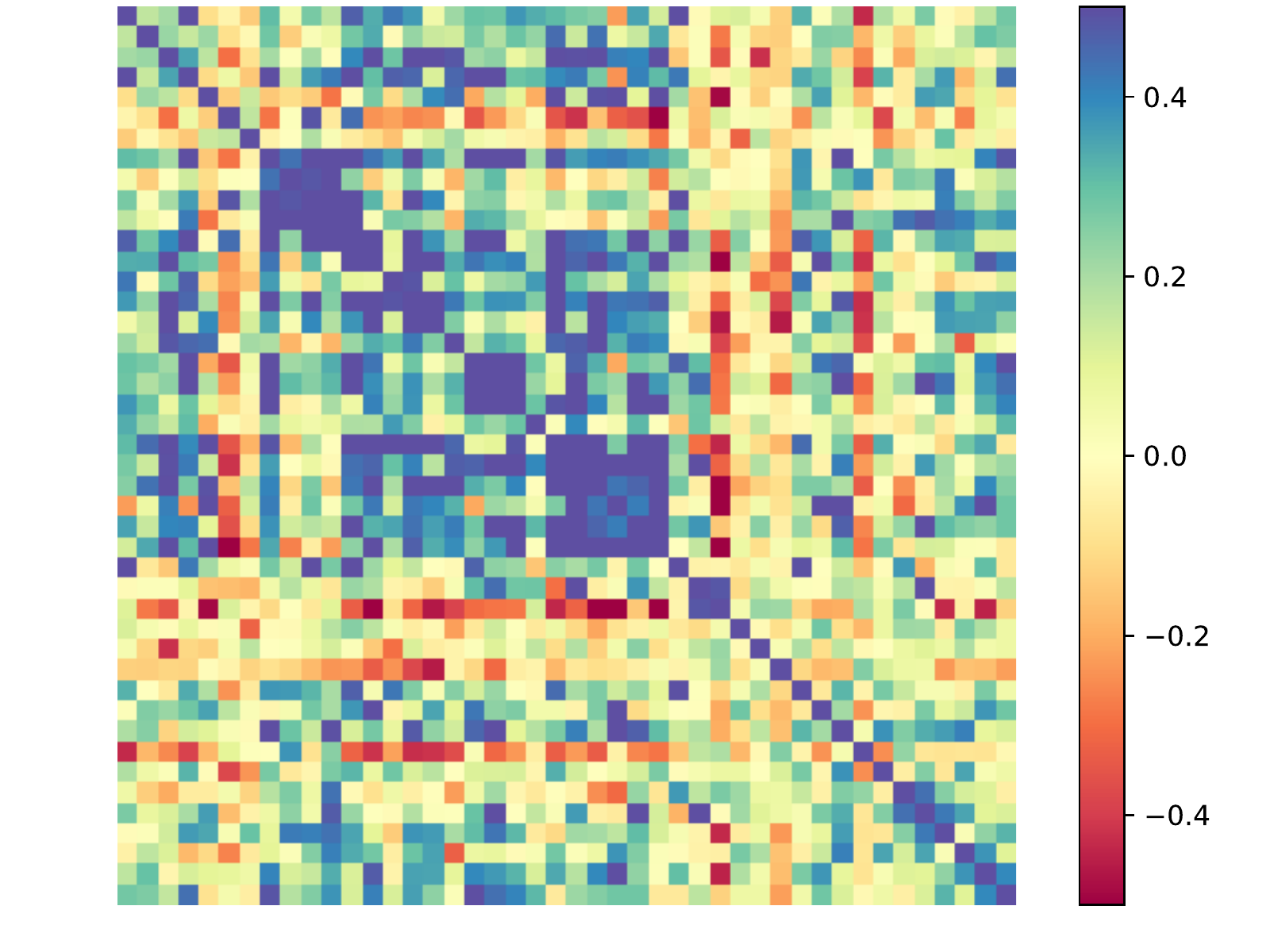}
			\vspace{-6mm}
			\caption*{$\Sigma$}
			\vspace{2mm}
			\includegraphics[width=1\linewidth]{images/real_estate/real_estate_perm}
			\vspace{-6mm}
			\caption*{{$\Sigma$}}
			\vspace{2mm}
			\includegraphics[width=1\linewidth]{images/real_estate/real_estate_perm}
			\vspace{-6mm}
			\caption*{$\Sigma$}
			\vspace{2mm}
			\includegraphics[width=1\linewidth]{images/real_estate/real_estate_perm}
			\vspace{-6mm}
			\caption*{$\Sigma$}
			\vspace{2mm}
			\includegraphics[width=1\linewidth]{images/real_estate/real_estate_perm}
			\vspace{-6mm}
			\caption*{$\Sigma$}
			\vspace{2mm}
		\end{minipage}%
		\begin{minipage}[c]{0.33\linewidth}%
			\centering 
			\vspace{0.3cm}
			\includegraphics[width=1\linewidth]{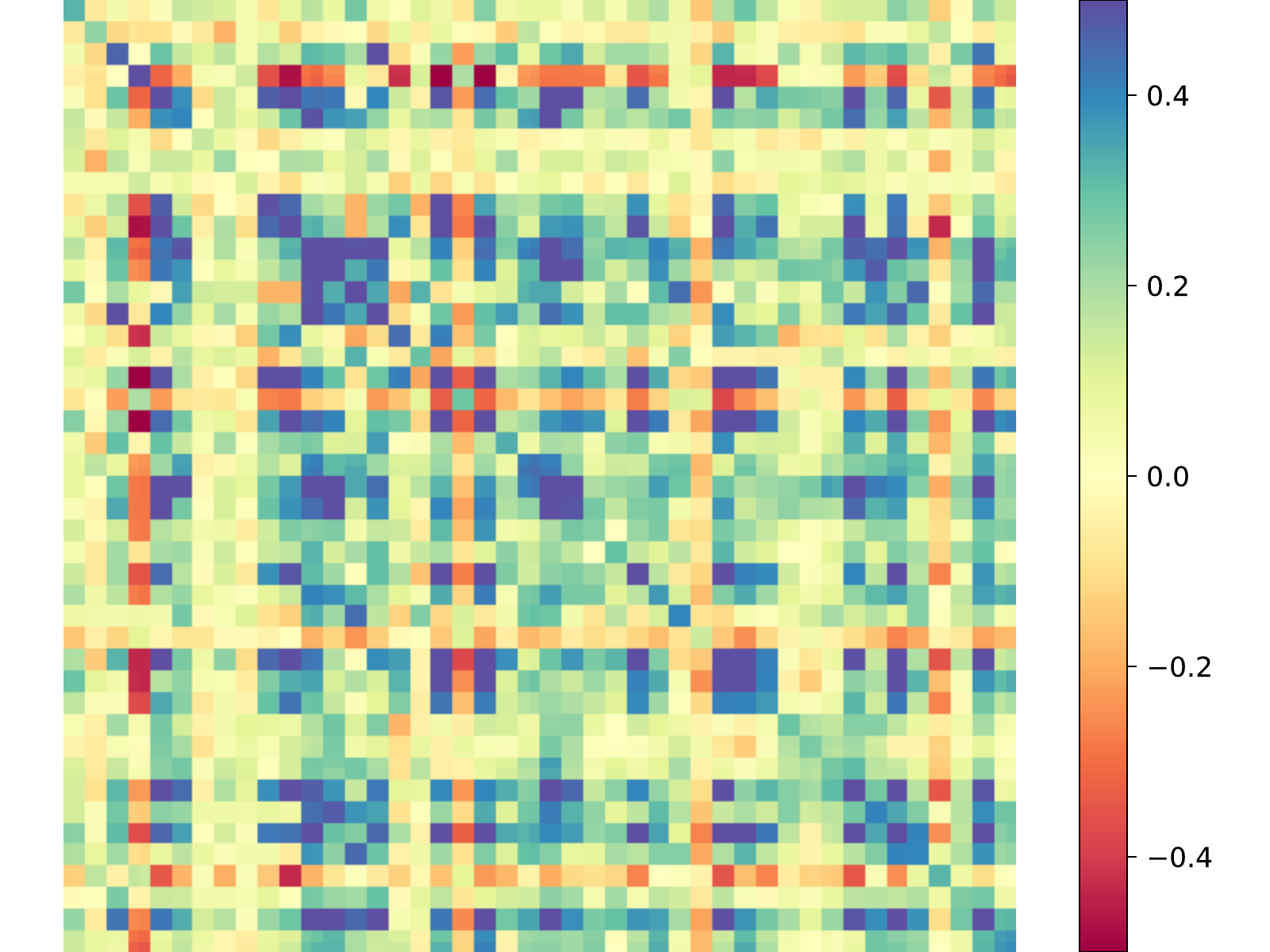}
						\vspace{-6mm}
			\caption*{$L$ (IALM)}
			\vspace{2mm}
			\includegraphics[width=1\linewidth]{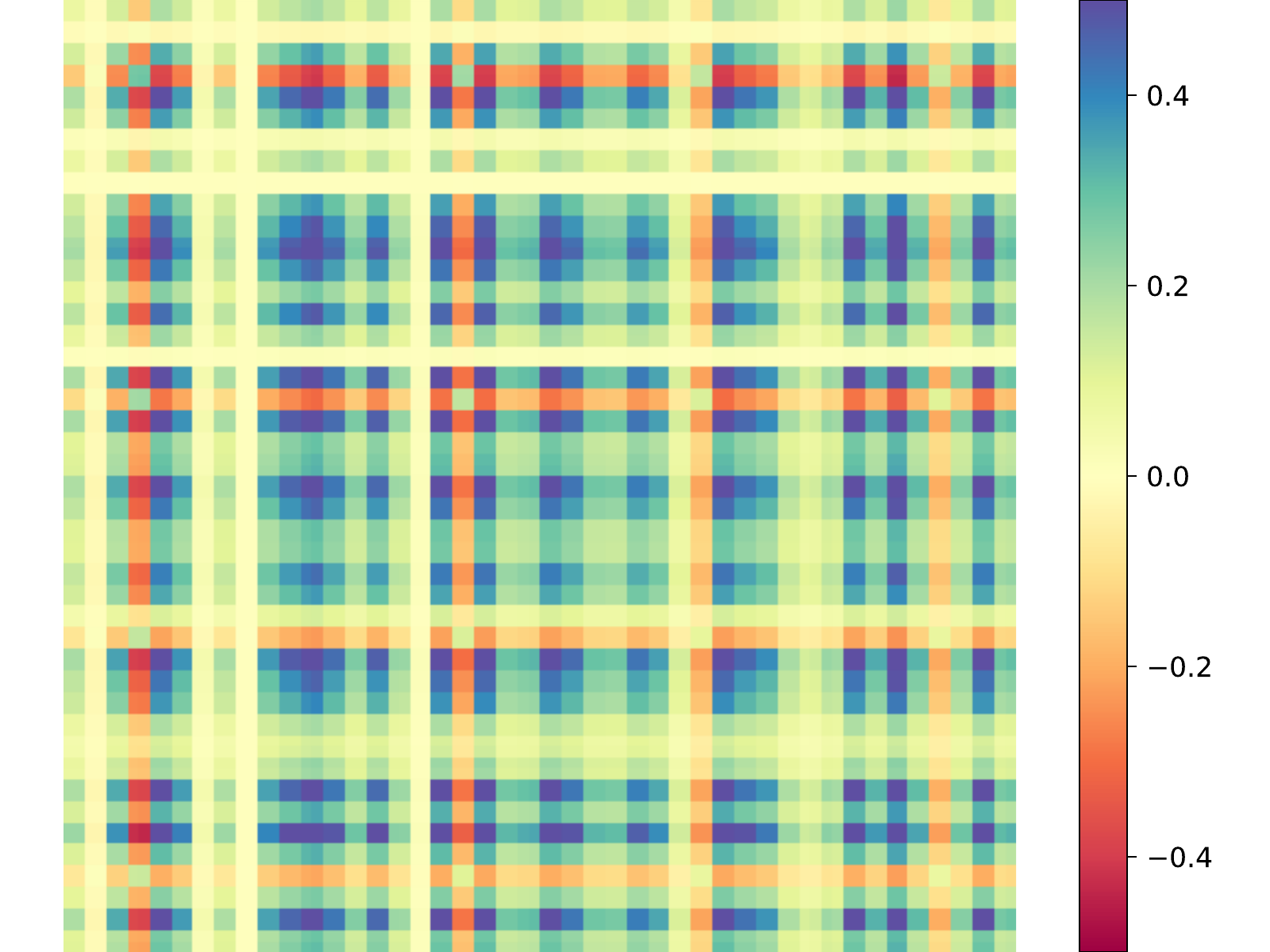}
			\vspace{-6mm}
			\caption*{$L$ (NC-RPCA)}
			\vspace{2mm}
			\includegraphics[width=1\linewidth]{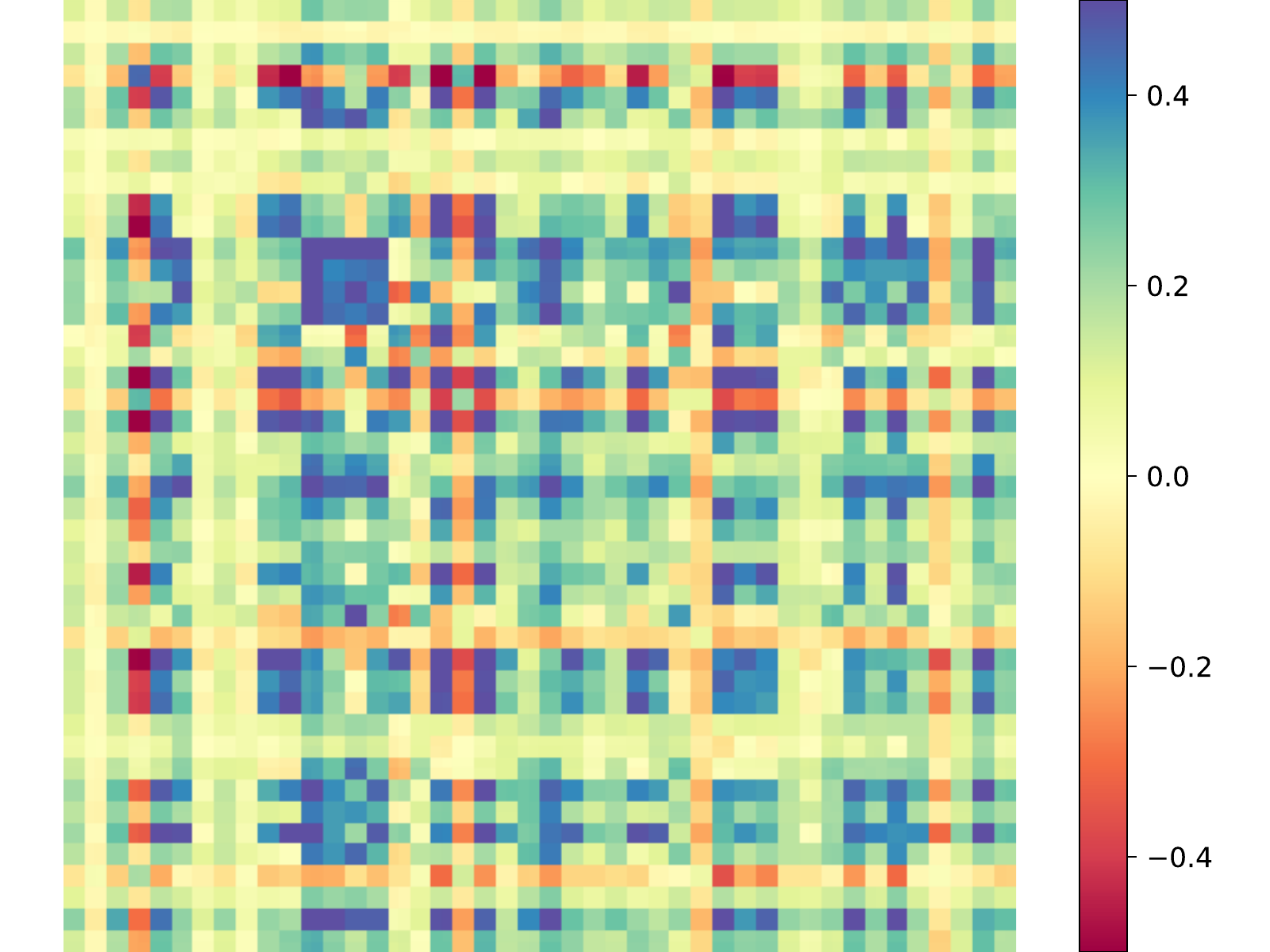}
			\vspace{-6mm}
			\caption*{$L$ (RPCA-GD)}
			\vspace{2mm}
			\includegraphics[width=1\linewidth]{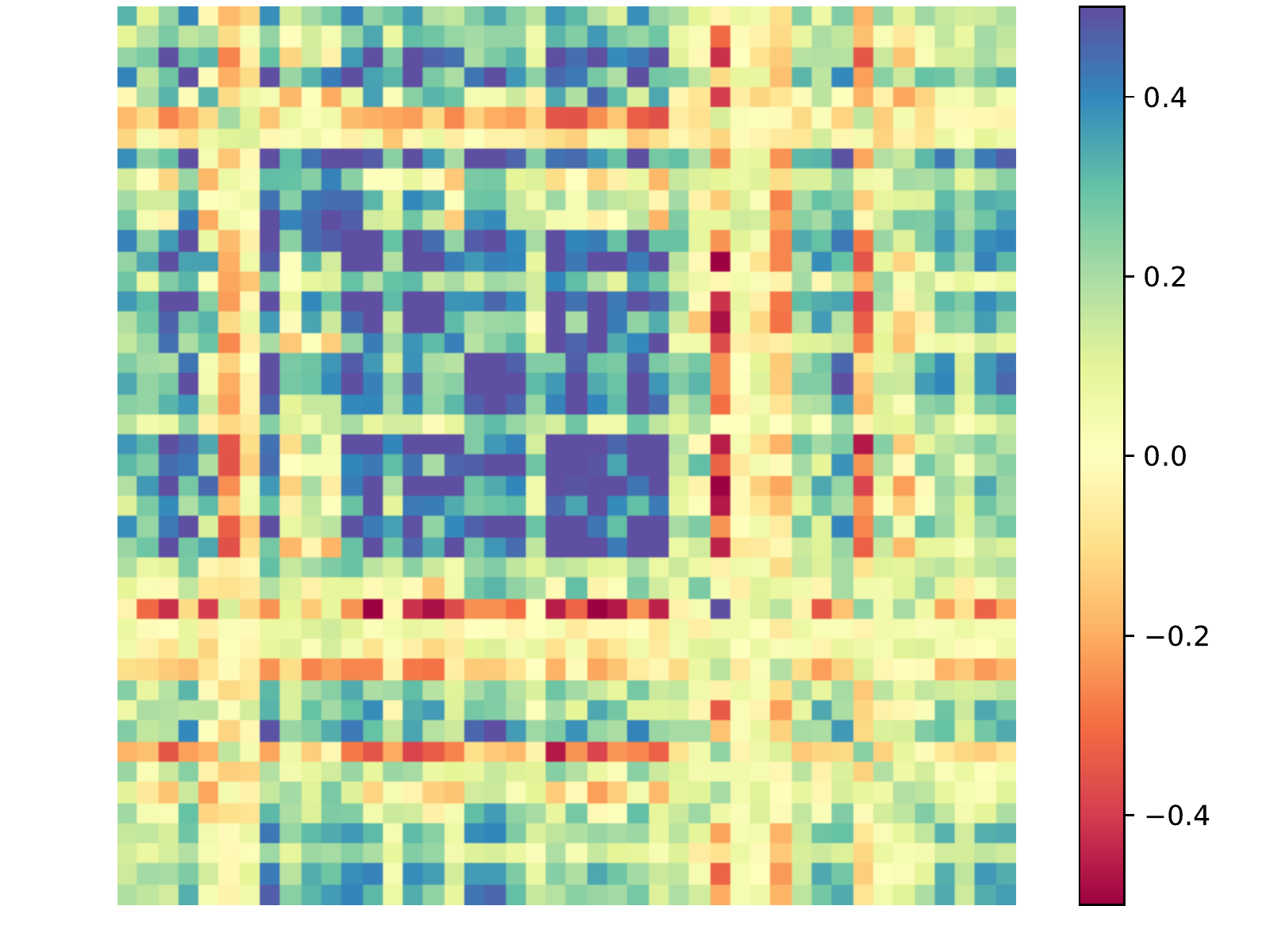}
			\vspace{-6mm}
			\caption*{$L$ (FPCP)}
			\vspace{2mm}
			\includegraphics[width=1\linewidth]{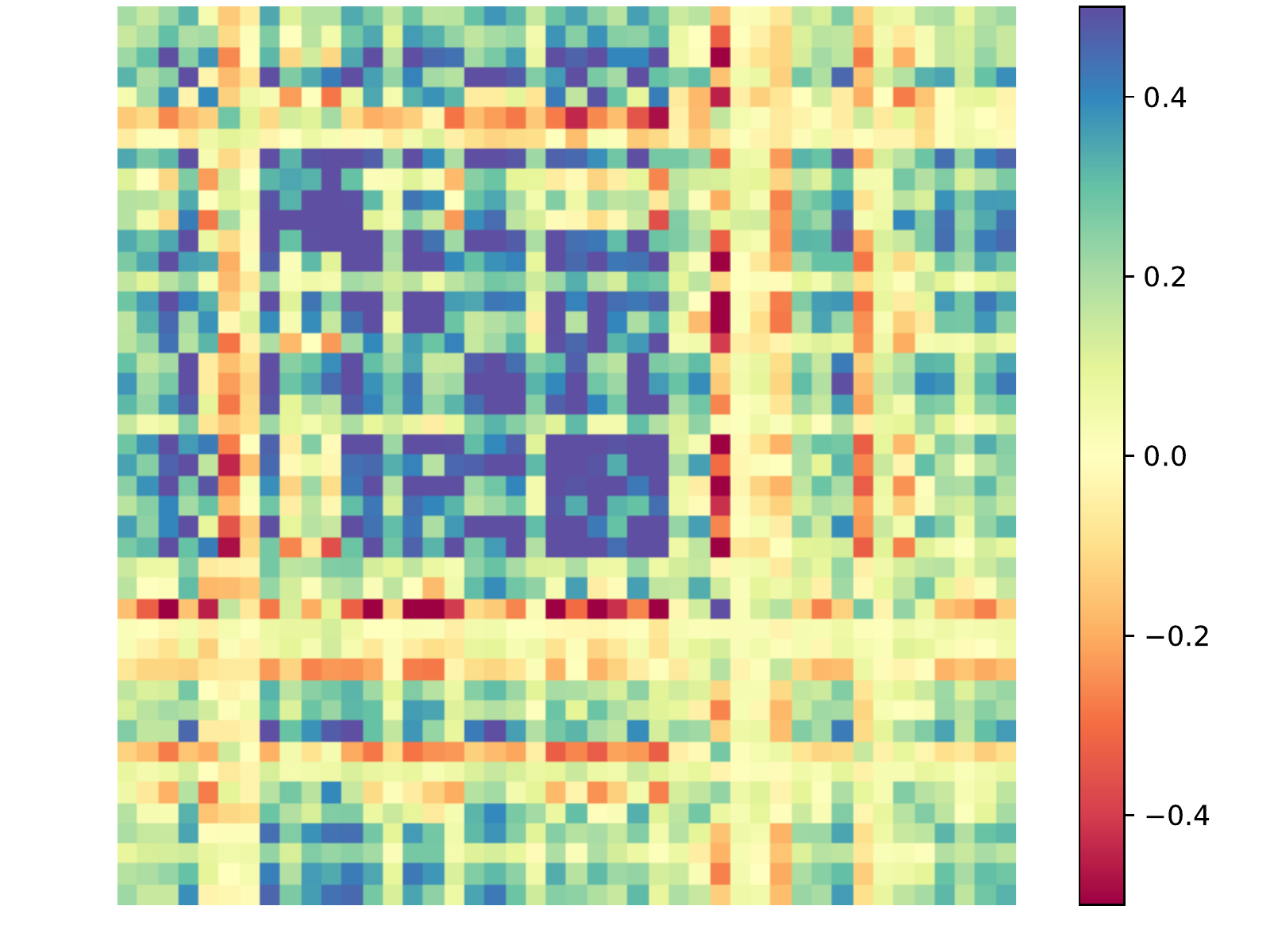} \vspace{-6mm}
			\caption*{$\phantom{ fdd}$ $L$ (DNN)  $\phantom{ fdd}$ \, }
			\vspace{2mm}
		\end{minipage}%
		\begin{minipage}[c]{0.33\linewidth}%
			\centering
			\vspace{0.3cm} \includegraphics[width=1\linewidth]{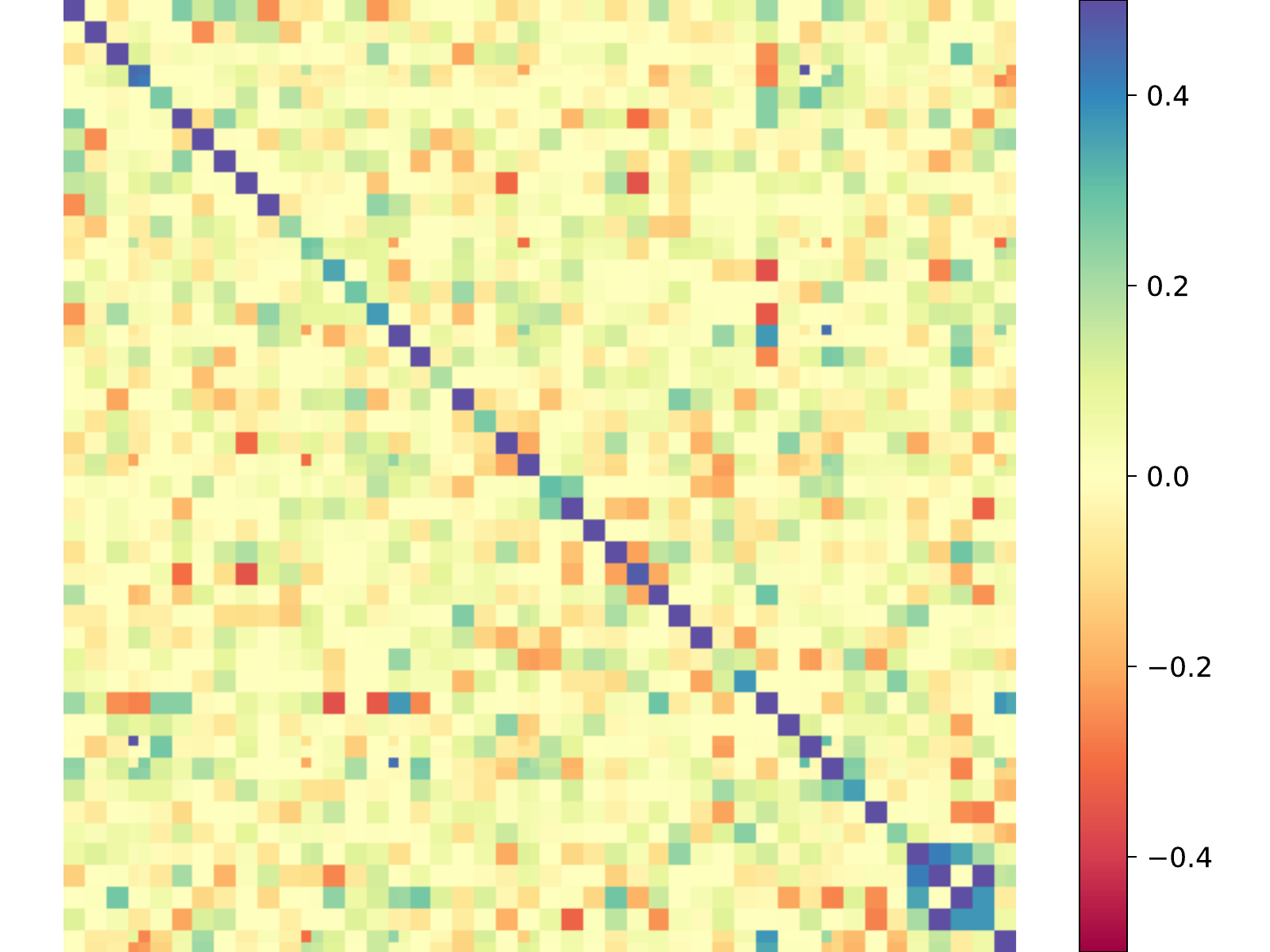}
			\vspace{-6mm}
			\caption*{$S$ (IALM)}
			\vspace{2mm}
			\includegraphics[width=1\linewidth]{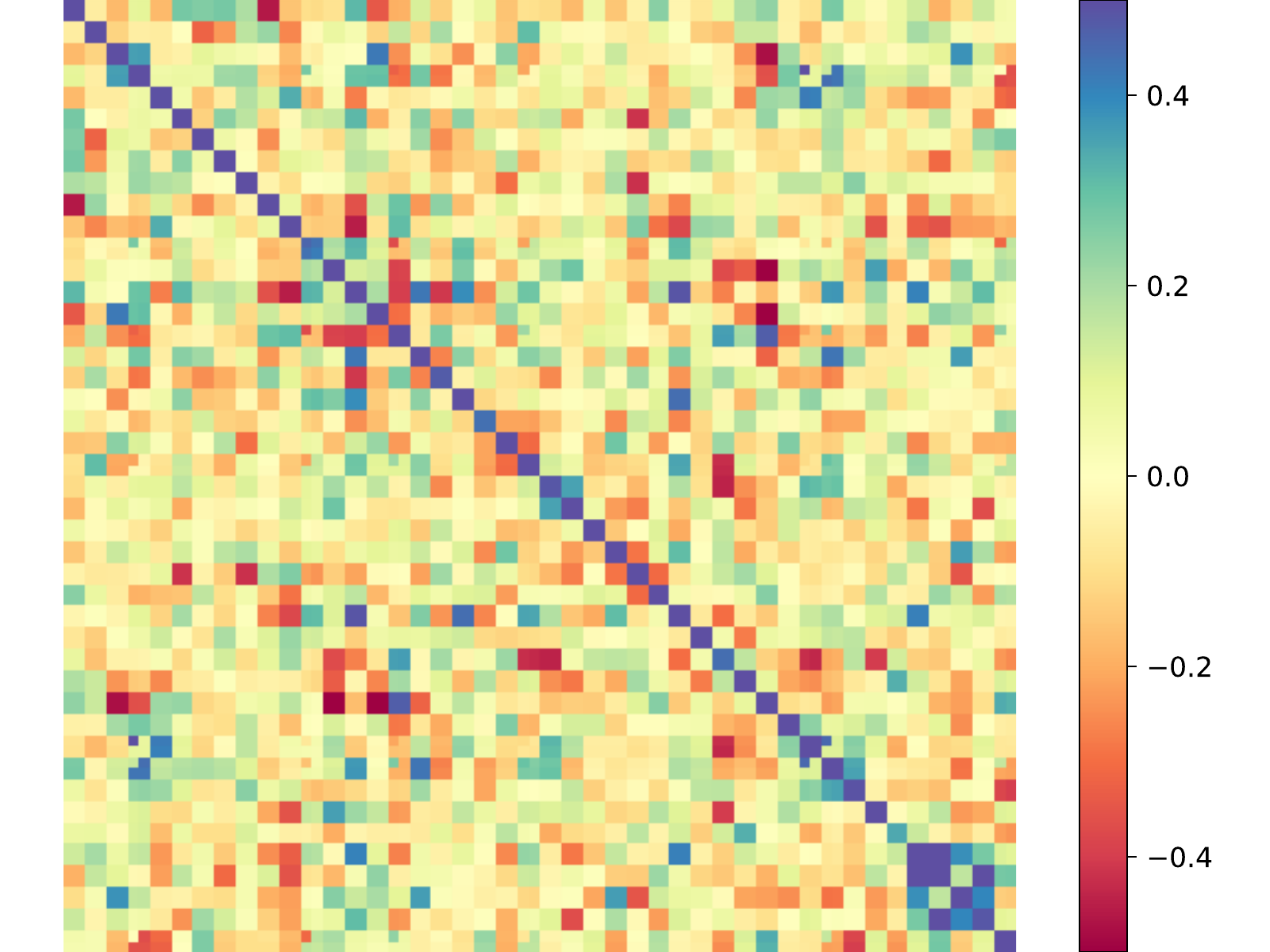}
			\vspace{-6mm}
			\caption*{$S$ (NC-RPCA)}
			\vspace{2mm}
			\includegraphics[width=1\linewidth]{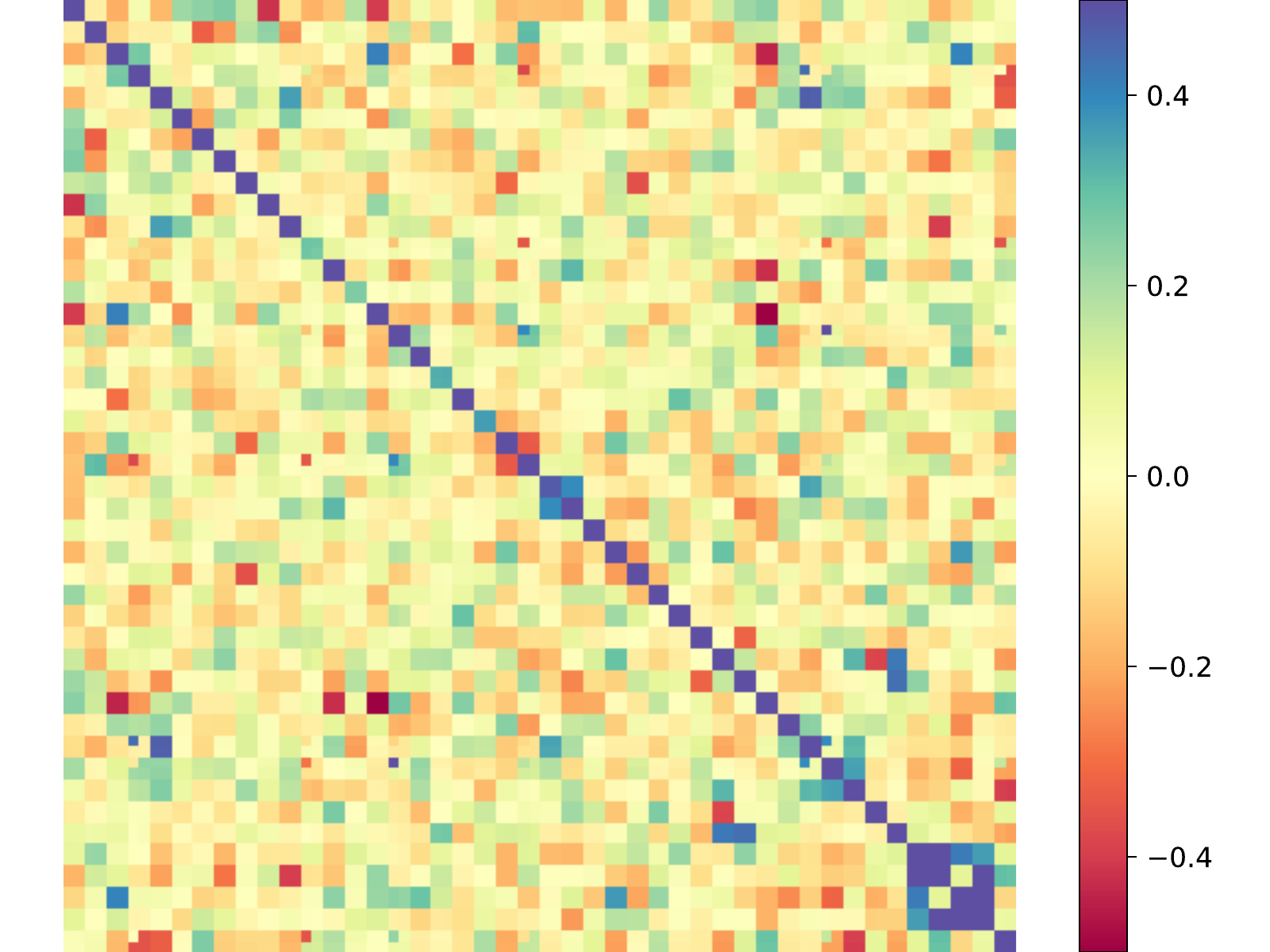}
			\vspace{-6mm}
			\caption*{$S$ (RPCA-GD)}
			\vspace{2mm}
			\includegraphics[width=1\linewidth]{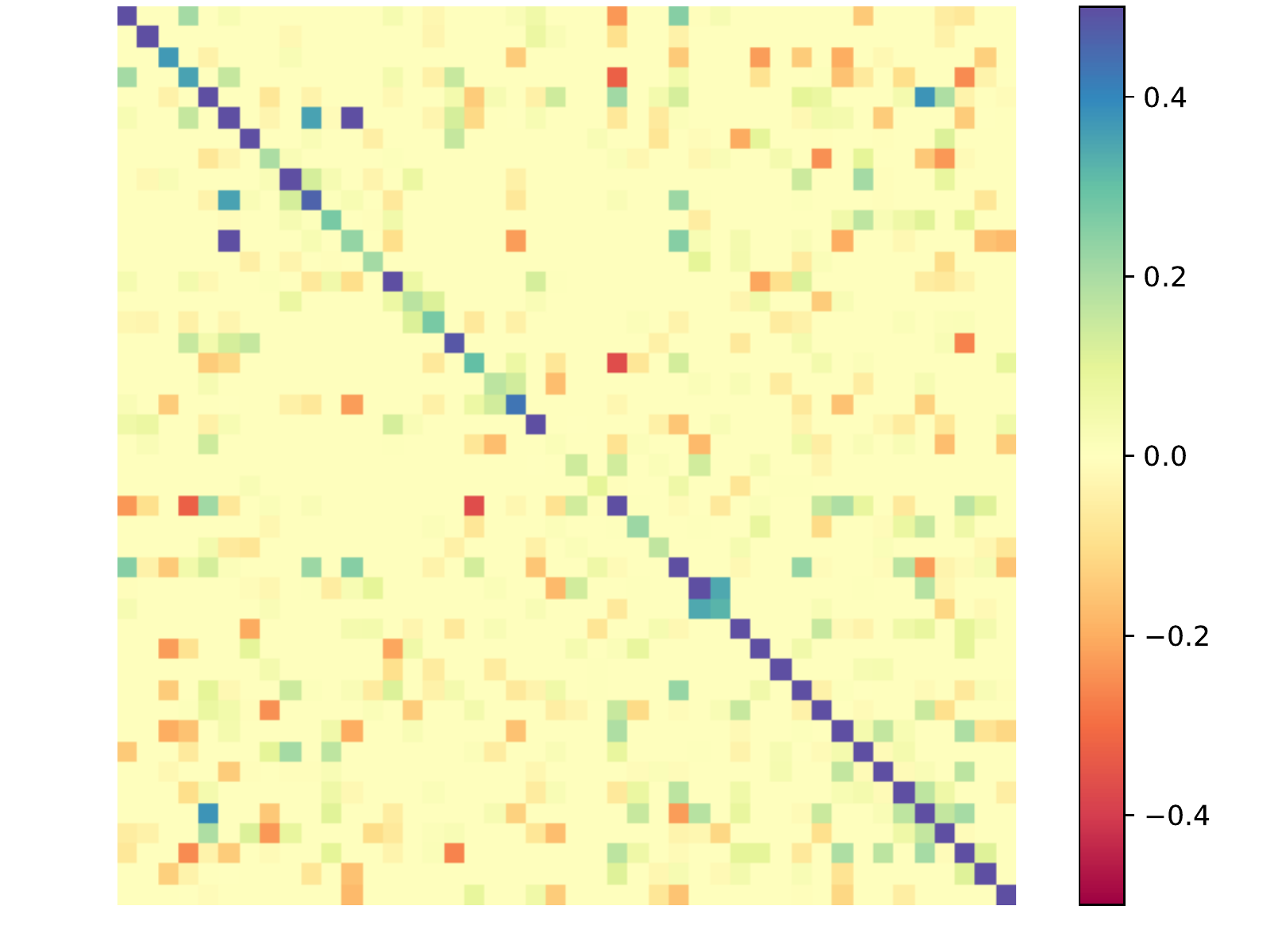}
			\vspace{-6mm}
			\caption*{$S$ (FPCP)}
			\vspace{2mm}
			\includegraphics[width=1\linewidth]{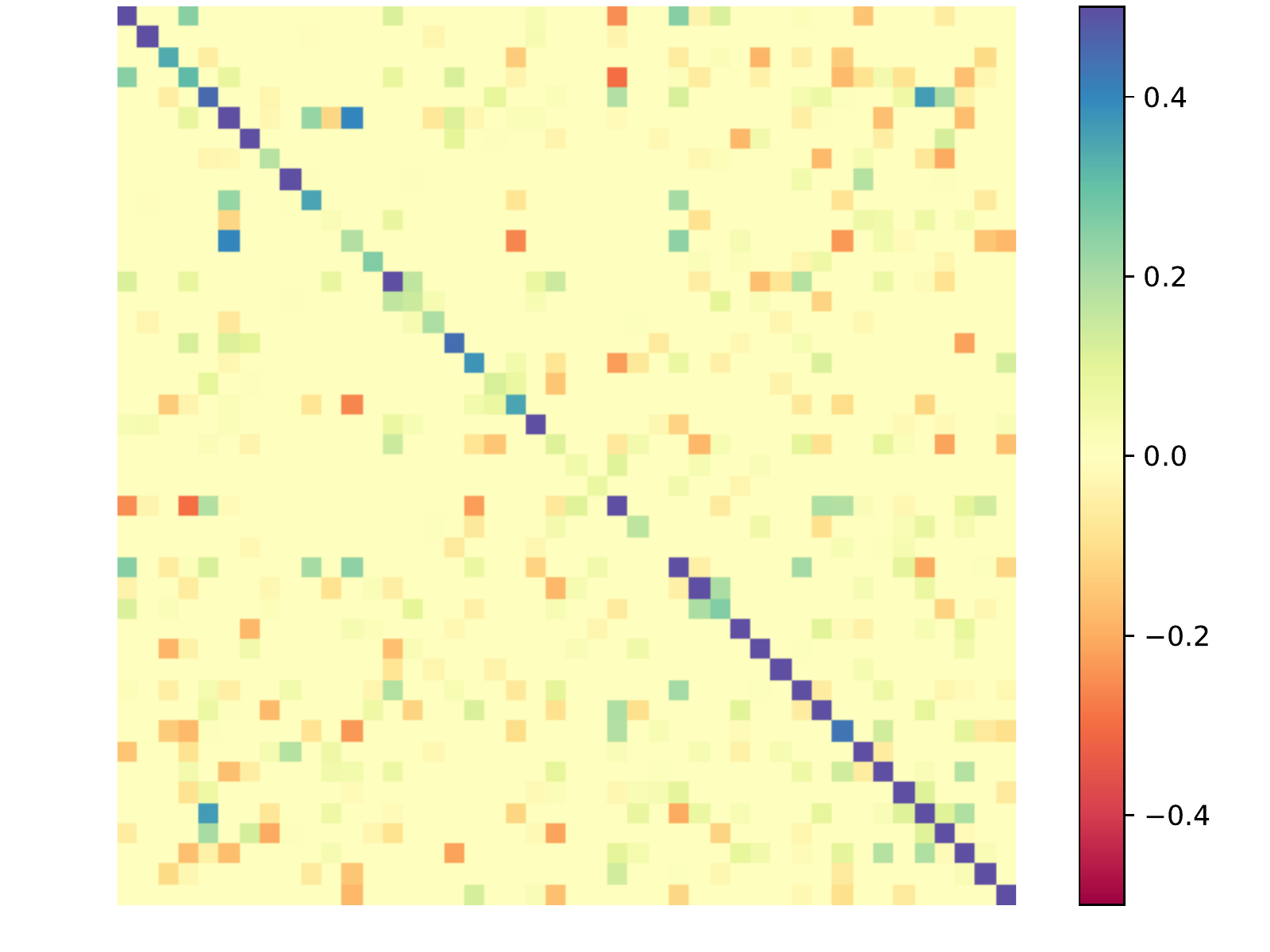} \vspace{-6mm}
			\caption*{$\phantom{ fdd}$ $S$ (DNN),  $\phantom{ fdd}$ \, with shrinkage}
			\vspace{2mm}
		\end{minipage} 
		\caption{Decomposition into a low-rank plus a sparse matrix of the correlation matrix of the real estate returns. The forced rank is set to $k=3$ for DNN and RCPA-GD. The matrix $L$ returned by FPCP has also a rank equal to $3$, while IALM gets a rank of 14 and Non-Convex RPCA gets a rank equal to 1. 
				The relative error of $L$ compared to $\Sigma$ is $0.41$ for DNN and $0.41$ for FPCP. The sparsity of $S$ is $0.65$ for FPCP and $0.76$ for DNN with shrinkage, $0.17$ without shrinkage. The relative error of $L + S$ compared to $\Sigma$ is $0.13$ for DNN and $0.24$ for FPCP. When shrinkage is not applied, this error drops down to $2.4 \cdot 10^{-8}$ with our algorithm.}
		\label{realestate}
	\end{figure}
\end{center}
%
%
%
%
\begin{center}
	\begin{figure}[!h]
		\begin{minipage}[c]{0.5\linewidth}%
			\centering \includegraphics[width=1\linewidth]{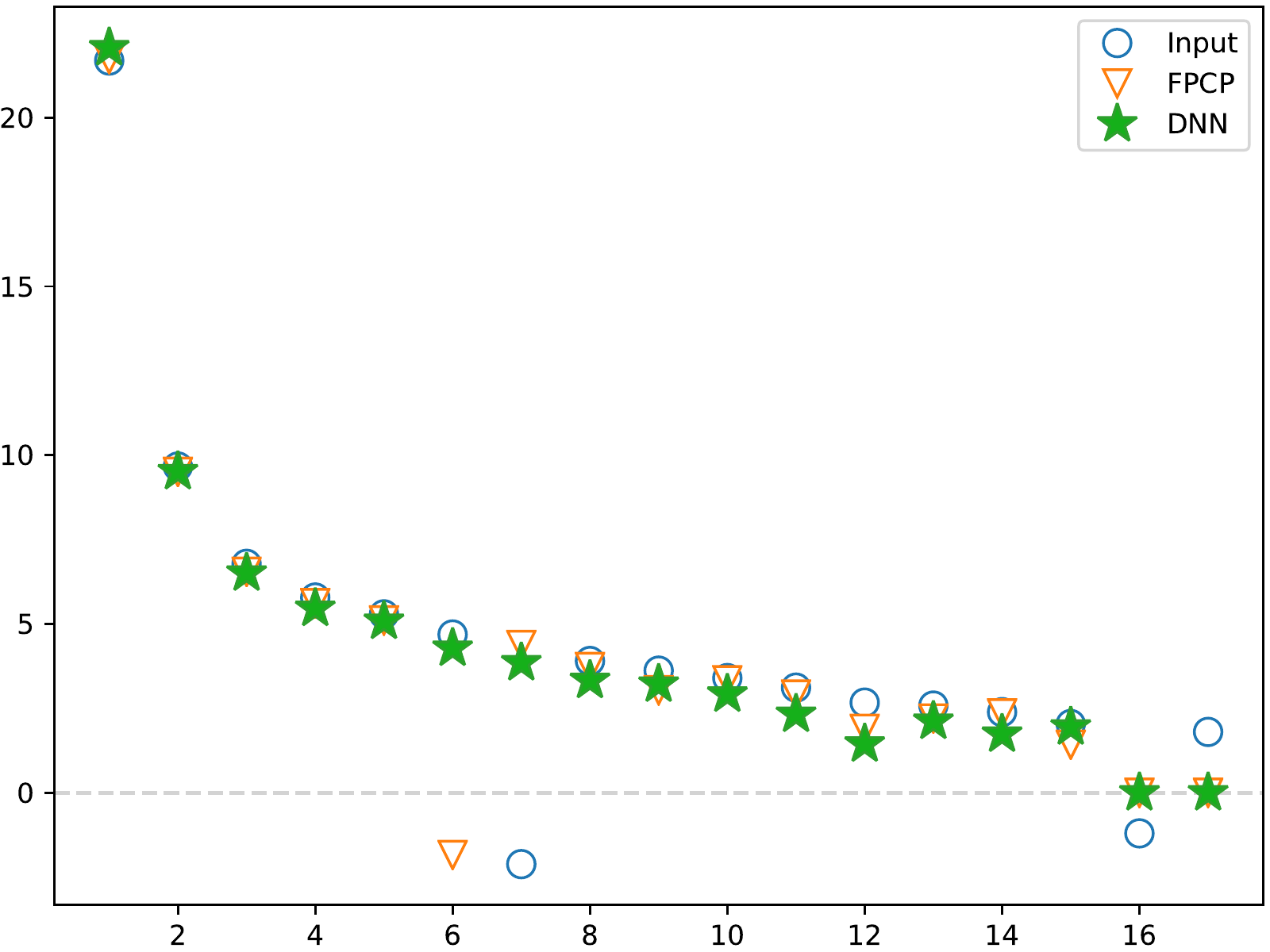}
			
		\end{minipage}%
		\begin{minipage}[c]{0.5\linewidth}%
			\centering \includegraphics[width=1\linewidth]{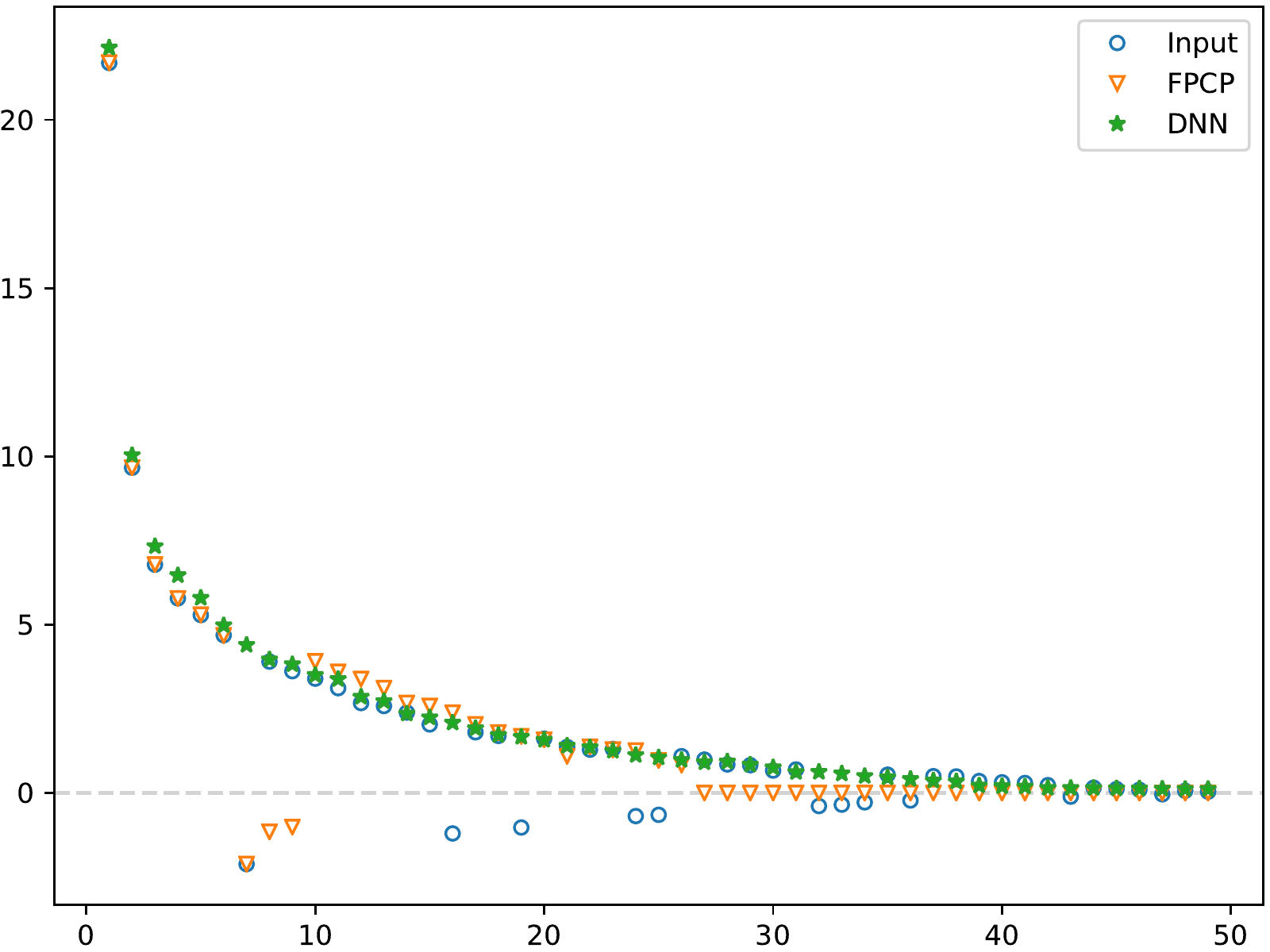}
			
		\end{minipage}%
		
		\begin{minipage}[c]{0.5\linewidth}%
			\centering \includegraphics[width=1\linewidth]{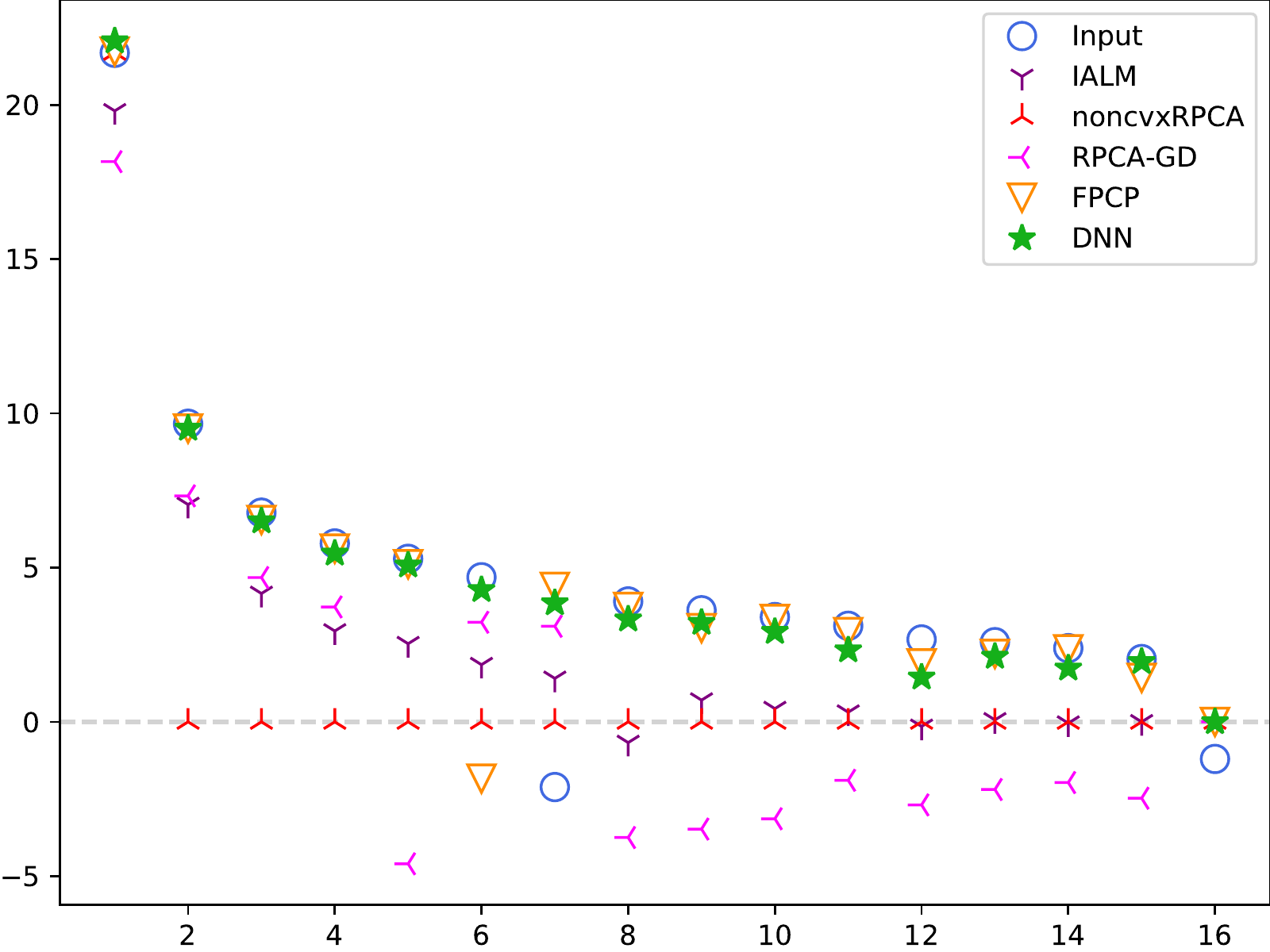}
			
		\end{minipage}%
		\begin{minipage}[c]{0.5\linewidth}%
			\centering \includegraphics[width=1\linewidth]{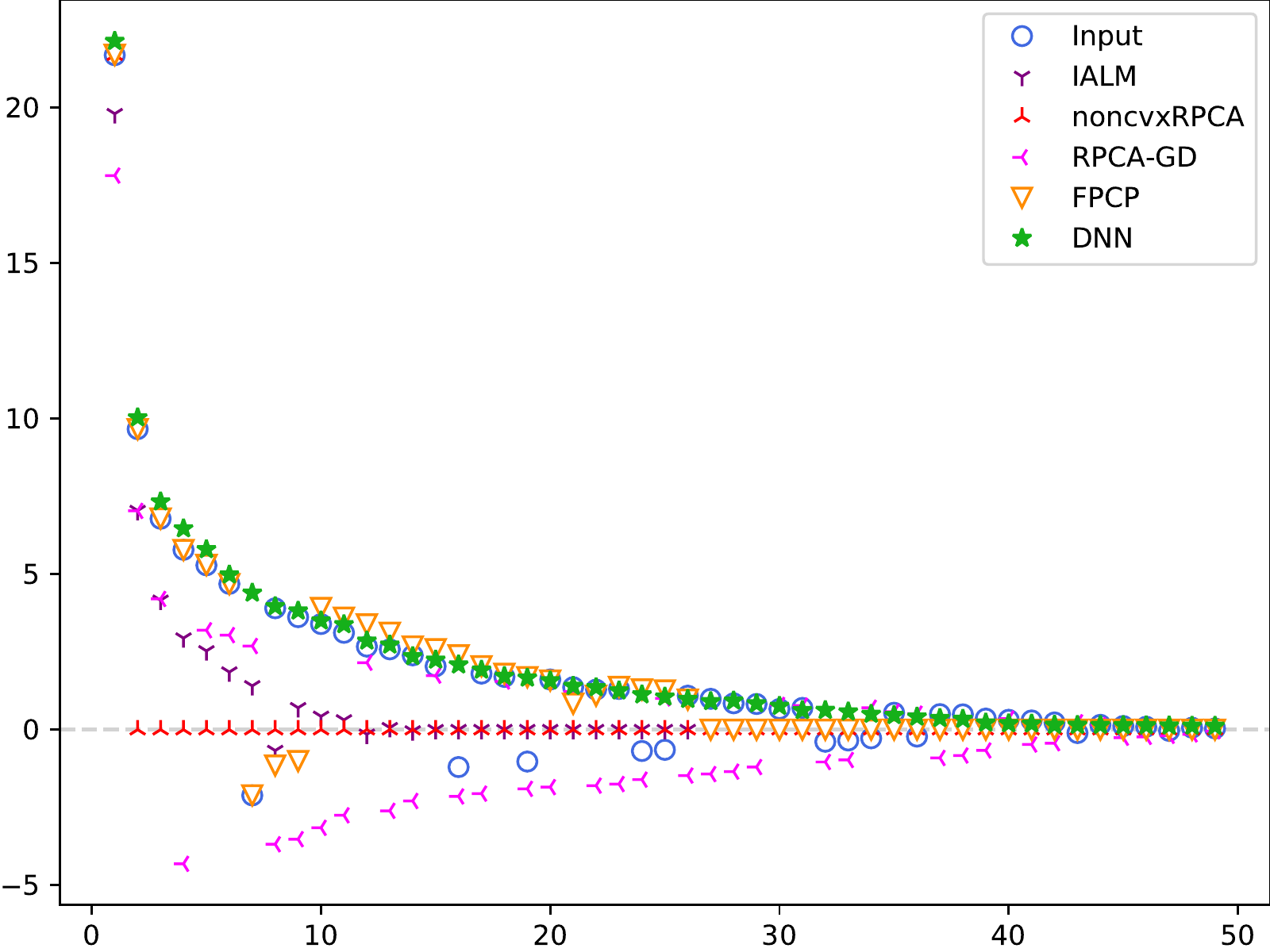}
			
		\end{minipage}%

		\caption{The first 17 (left panels) and 50 (right panels) eigenvalues of the correlation matrix of the real estate returns $\Sigma$ and of the low-rank
			matrix $L$ outputted. 
			The forced rank is set to $k=15$ (left) and to $k=88$ (right). 
			{For clarity, the first line compares our algorithm to FPCP only, while the second line compares all the algorithms.}
		}
		\label{realestate_eigen}
	\end{figure}
\end{center}
%
%
%
\subsection{Empirical verification of bounded parameters}\label{subsec:bound}
{To verify empirically} our assumption \eqref{ass-D} in Theorem~\ref{thm:convergence} that {the} parameters $(\Theta_j)_{j\in \N_0}$ generated by our algorithm \eqref{gradient-method} remain in a compact set, we plotted in Figure \ref{norm_weights} the running maximum $\max_{0\leq j \leq J} \|\Theta_j\|$ as a function of the number of iterations $J$ for both examples on the S\&P500 and the real estate data used in the previous section. 
For both cases, we observe, as desired, that the running maximum $\max_{0\leq j \leq J} \|\Theta_j\|$ converges, which means that at least empirically, $(\Theta_j)_{j\in \N_0}$ remains in a compact set. 
%
\begin{center}
	\begin{figure}[!h]
		\begin{minipage}[c]{0.52\linewidth}
			\centering \includegraphics[width=1\linewidth]{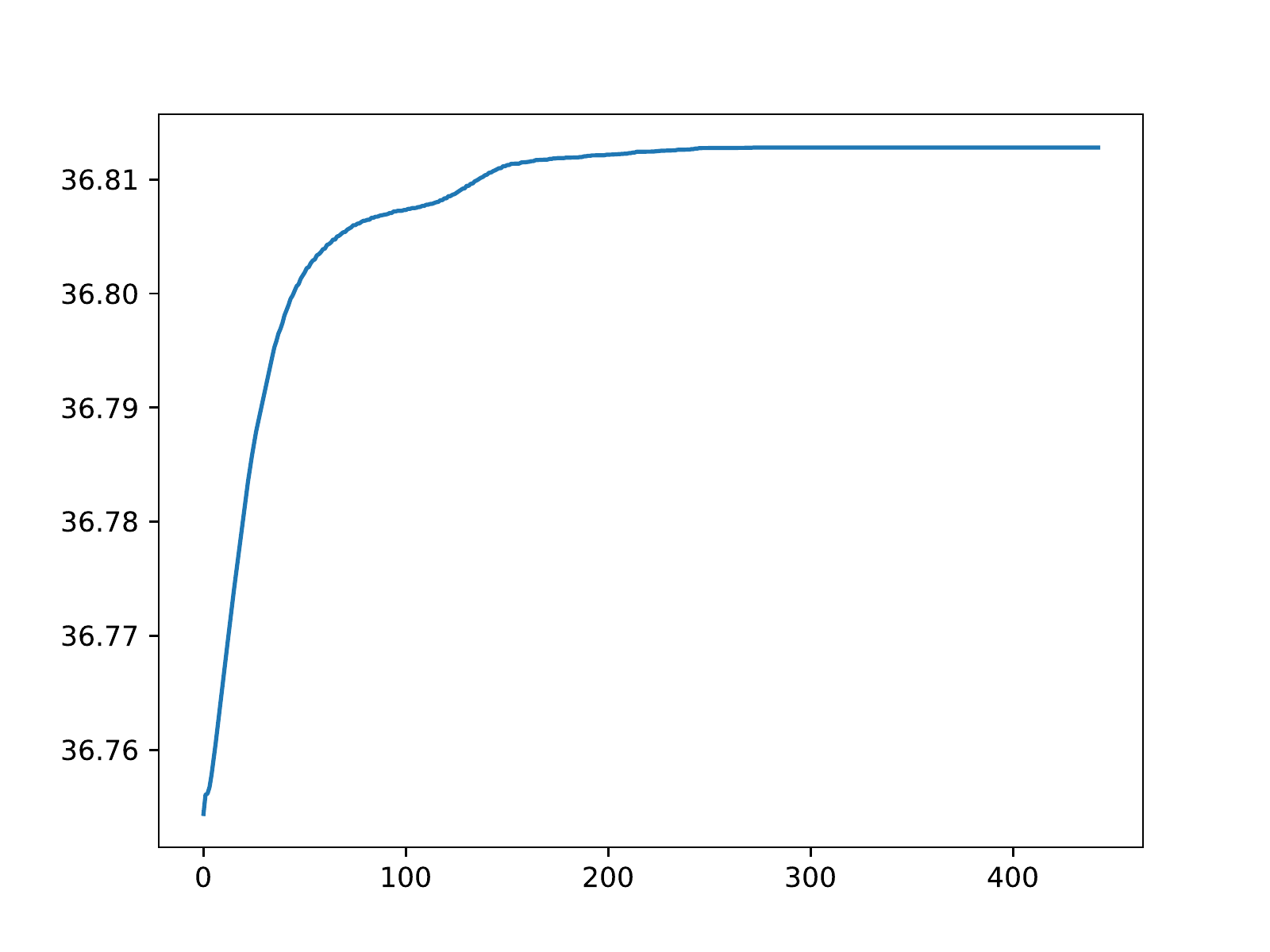}
			
		\end{minipage}%
		\begin{minipage}[c]{0.52\linewidth}
			\centering \includegraphics[width=1\linewidth]{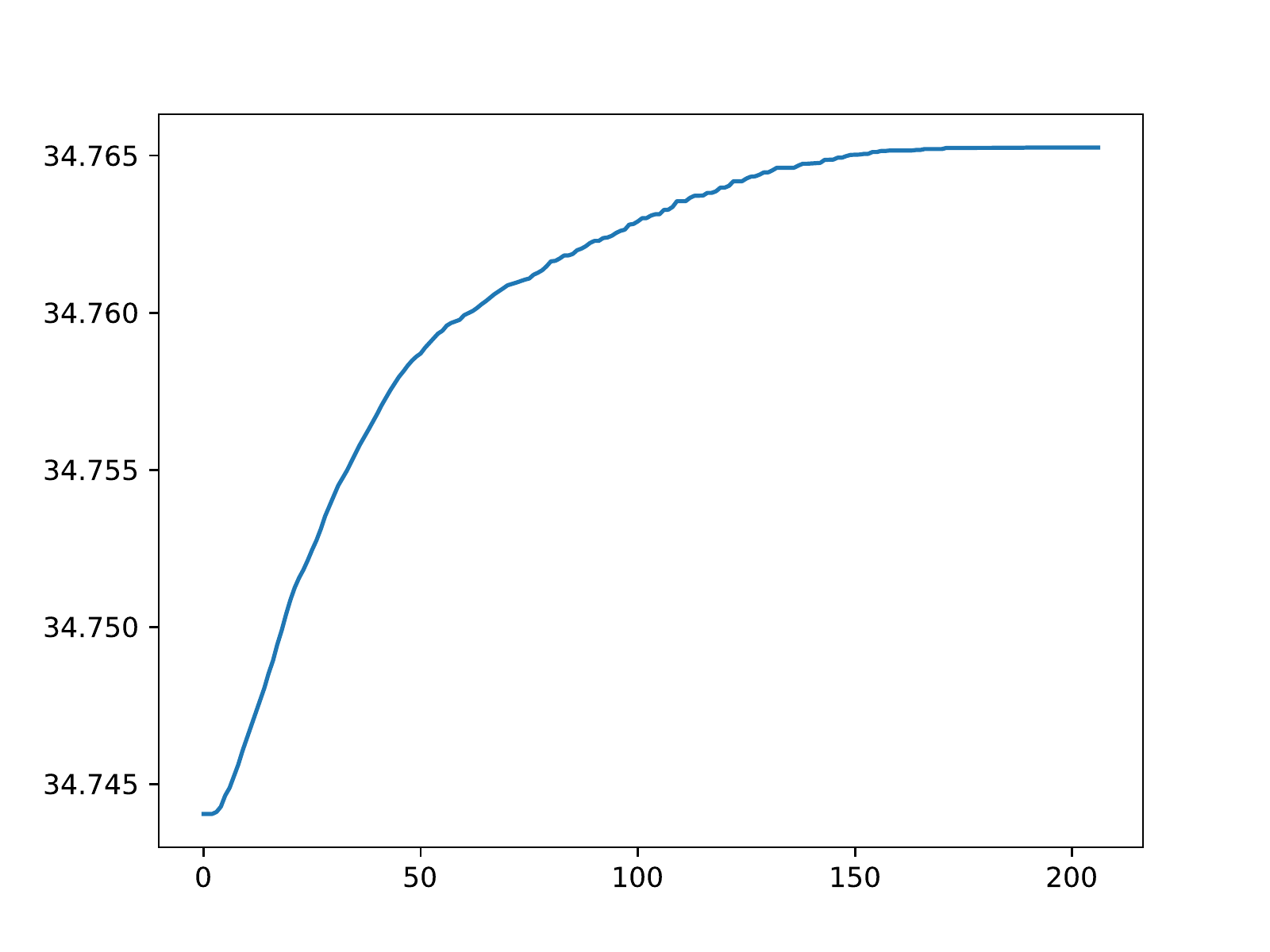}
		\end{minipage}%
		\caption{The running maximum $\max_{0\leq j \leq J} \|\Theta_j\|$ (y-axis) of our parameters $(\Theta_j)_{j \in \N_0}$  generated by our DNN algorithm, plotted as a function of iterations $J$ (x-axis) for both examples on  the S\&P500 stock prices  (left) and on the real estate returns (right).}
		\label{norm_weights}
	\end{figure}
\end{center}
\section{Proof of convergence}\label{sec:proof}

A vast majority of first-order methods for minimizing locally a non-convex
function with provable convergence rate are meant to minimize $L$-smooth
functions, that is, differentiable functions with a Lipschitz continuous
gradient. Also, the value of the Lipschitz constant with respect to
a suitable norm plays a prominent role in this convergence rate 
{ (see \cite{ghadimi2016accelerated} and references therein). As
a critical step in the convergence proof for the minimization procedure
of the function ${\varphi}$, we compute carefully a bound on the Lipschitz
constant of its gradient, also called the \emph{smoothness constant} below. 

For establishing this bound, we use in a critical way the recursive compositional nature of the neural network that represents $M$: we first consider the case of a neural network with a single layer. We compute explicitly the derivative of the corresponding objective function $\varphi(\Theta)$ with respect to each parameter in Lemmas \ref{lem:der-b} and \ref{lem:der-d}. Then, using an elementary lemma on the smoothness constant of a composition and a product of functions, we carefully derive an upper bound on the derivative of a single-layered network output with respect to the bias on the input in Lemma \ref{lem:Lip-b}. This result represents the key element of our proof. Lemma \ref{lem:Lip-d} completes the picture by computing such a bound for derivatives with respect to other coefficients. Lemma \ref{lem:L-smooth} merges these results and provides a bound on the overall smoothness constant for the case where $M$ is parametrized by a single-layer neural network. Finally, we deduce inductively a bound on this Lipschitz constant for a multi-layered neural network in the proof of Theorem \ref{thm:convergence}. }

So, we consider in the beginning of this section a neural
network with a single layer of $\ell$ neurons. These neurons have each $\sigma:\mathbb{R}\to[-1,1]$ as activation function, and $\sigma$ has its first and the second derivative uniformly bounded
by $\sigma'_{\textrm{max}}>0$ and $\sigma''_{\textrm{max}}$,
respectively. 
We let $r:=n(n+1)/2$, 
$A=\left[A_{i,j}\right]_{i,j}\in\mathbb{R}^{\ell\times r}$ be the
weights, and $b=\left[b_{i}\right]_{i}\in\mathbb{R}^{\ell}$ be
the bias on the input, and let 
\[
f_{1}^{A,b}:\mathbb{R}^{r}\to\mathbb{R^{\ell}},\qquad w\mapsto f_{1}^{A,b}(w)=Aw+b.
\]
The coefficients on the output are denoted by $C=\left[C_{i,j}\right]_{i,j}\in\mathbb{R}^{nk\times\ell}$
and the bias by $d=\left[d_{i}\right]_{i}\in\mathbb{R}^{nk}$. As above,
we define 
\[
f_{2}^{C,d}:\mathbb{R}^{\ell}\to\mathbb{R}^{nk}\qquad v\mapsto f_{2}^{C,d}(v)=Cv+d.
\]
With $\tilde{\sigma}:\mathbb{R}^{\ell}\to\mathbb{R}^{\ell}$, $u\mapsto\tilde{\sigma}(u)=\left(\sigma(u_{1})\cdots\sigma(u_{\ell})\right)^{T}$, and $\Theta:=(A,b,C,d)$
the single layer neural network $\cNt:\mathbb{R}^{r}\to\mathbb{R}^{nk}$
is then the composition of these three functions, that is: $\cNt=f_{2}^{C,d}\circ\tilde{\sigma}\circ f_{1}^{A,b}$.
%
For a given $\Sigma\in\mathbb{S}^{n}$, our approximated objective function with respect to the above single-layer neural network  is defined by
\begingroup\makeatletter\def\f@size{7}\check@mathfonts
\def\maketag@@@#1{\hbox{\m@th\normalsize\normalfont#1}}
\begin{equation}\label{eq:tilde phi}
\tilde{\varphi}(\Theta)=  \sum_{i,j=1}^{n}\mu\left(\left[g\left(\cNt\left(h(\Sigma)\right)\right)g\left(\cNt\left(h(\Sigma)\right)\right)^{T}-\Sigma\right]_{i,j}\right),
\end{equation}
\endgroup
where $\mu:\mathbb{R}\to\mathbb{R}$ is a smooth approximation of
the absolute value function with a derivative uniformly bounded by
$1$ and its second derivative bounded by $\mu''_{\max}$.

As announced above, we start our proof by computing the partial derivatives of $\tilde{\varphi}$.
For abbreviating some lengthy expressions, we use the following shorthand notation throughout this section. We fix $\Sigma \in \mathbb S^n$ and let for $1\leq i,j \leq n$,
\begin{equation}\label{shortcut-notation}
\begin{split}
\omega_{i,j}&\!:=\!\left[g\left(\cNt\left(h(\Sigma)\right)\right)g\left(\cNt\left(h(\Sigma)\right)\right)^{T}-\Sigma\right]_{i,j}\!,\\ 
X&:=\cNt\left(h(\Sigma)\right)\in\mathbb{R}^{nk},\\
Y&:=\tilde{\sigma}\circ f_{1}^{A,b}(h(\Sigma))\in\mathbb{R}^{\ell},\\
Z&:=f_{1}^{A,b}(h(\Sigma))\in\mathbb{R}^{\ell}.
\end{split}
\end{equation} 

\begin{lem}\label{lem:der-b}
	Let $\tilde{\varphi}$ be the function defined in \eqref{eq:tilde phi} and $1\leq \iota \leq \ell$. Then
	\begingroup\makeatletter\def\f@size{7}\check@mathfonts
	\def\maketag@@@#1{\hbox{\m@th\normalsize\normalfont#1}}
	\begin{equation*}
	\begin{split}
	\frac{\partial\tilde{\varphi}(\Theta)}{\partial b_{\iota}}=
	\!\sum_{i,j=1}^{n}&\!\mu'(\omega_{i,j})\sigma'(\!Z_{\iota}\!)\times\\
	&\left(\sum_{s=1}^{k}\!\left(C_{(i-1)k+s,\iota}X_{(j-1)k+s}\!+\!X_{(i-1)k+s}C_{(j-1)k+s,\iota}\!\right)\right).
	\end{split}
	\end{equation*}
	\endgroup
	Moreover,  for every $1\leq \iota \leq \ell$, $1\le\eta\le r$ we have
	\begin{equation*}	\frac{\partial\tilde{\varphi}(\Theta)}{\partial A_{\iota,\eta}}=[h(\Sigma)]_{\eta}\frac{\partial\tilde{\varphi}(\Theta)}{\partial b_{\iota}}.
	\end{equation*}
\end{lem} 
\begin{proof}
	Let $1\le\iota\le\ell$.
	By definition of $\tilde{\varphi}$, we have 
	\begingroup\makeatletter\def\f@size{7}\check@mathfonts
	\def\maketag@@@#1{\hbox{\m@th\normalsize\normalfont#1}}
	\begin{equation}
	\label{eq:pa-b-start}
	\begin{split}
	\frac{\partial\tilde{\varphi}(\Theta)}{\partial b_{\iota}} &=\sum_{i,j=1}^{n}\mu'(\omega_{i,j})\Bigg(\Bigg[\tfrac{\partial g\left(\cNt\left(h(\Sigma)\right)\right)}{\partial b_{\iota}}g\left(\cNt\left(h(\Sigma)\right)\right)^{T}\\
	%
	& \qquad\qquad\qquad  \quad
	+
	g\left(\cNt\left(h(\Sigma)\right)\right)
	\tfrac{\partial g\left(\cNt\left(h(\Sigma)\right)\right)^{T}}{\partial b_{\iota}}\Bigg]_{i,j}\Bigg).
	\end{split}
	\end{equation}
	\endgroup
	As $\cNt\equiv f_{2}^{C,d}\circ\tilde{\sigma}\circ f_{1}^{A,b}$, we have
	\begingroup\makeatletter\def\f@size{7}\check@mathfonts
	\def\maketag@@@#1{\hbox{\m@th\normalsize\normalfont#1}}
	\begin{equation}\label{eq:pa-b-comp}
	\tfrac{\partial g\left(\cNt\left(h(\Sigma)\right)\right)}{\partial b_{\iota}}=\nabla_{X}g(X)\cdot\nabla_{Z}\left(f_{2}^{C,d}\circ\tilde{\sigma}\right)(Z)\cdot\frac{\partial f_{1}^{A,b}(h(\Sigma))}{\partial b_{\iota}}.
	\end{equation}
	\endgroup
	Observe that
	\begingroup\makeatletter\def\f@size{7}\check@mathfonts
	\def\maketag@@@#1{\hbox{\m@th\normalsize\normalfont#1}}
	\begin{eqnarray}
	\frac{\partial f_{1}^{A,b}(h(\Sigma))}{\partial b_{\iota}}&=&\frac{\partial(Ah(\Sigma)+b)}{\partial b_{\iota}}=\left(\begin{array}{c}
	1_{\iota=1}\\
	\vdots\\
	1_{\iota=\ell}
	\end{array}\right),\label{eq:pa-b-comp-1}\\
	\left[\nabla_{Z}\left(f_{2}^{C,d}\circ\tilde{\sigma}\right)(Z)\right]_{\iota}
	&=&
	\left[\nabla_{Z}\left(C\tilde{\sigma}(Z)+d\right)\right]_{\iota}=\begin{pmatrix}C_{1,\iota}\sigma'(Z_{\iota})\\
	\vdots\\
	C_{nk,\iota}\sigma'(Z_{\iota})\nonumber
	\end{pmatrix}.
	\end{eqnarray}
	\endgroup
	The third order tensor $\nabla_{X}g(X)$, as an $nk$-dimensional vector of $n$-by-$k$ matrices, has for 
	$((i-1)k+j)$-th element the 
	matrix
	whose only nonzero element is a $1$ at position $(i,j)$, namely
	\begin{equation}
	\label{eq:pa-b-comp-2}
	\left[\nabla_{X}g(X)\right]_{(i-1)k+j}
	=
	\begin{pmatrix}0 & \cdots & 0\\
	\vdots & 1_{(i,j)} & \vdots\\
	0 & \cdots & 0
	\end{pmatrix}
	\in \R^{n\times k}.
	\end{equation}
	Plugging these in \eqref{eq:pa-b-comp}, we get
	\begin{equation*}
	\begin{split}
	\frac{\partial g\left(\cNt\left(h(\Sigma)\right)\right)}{\partial b_{\iota}}=\sigma'(Z_{\iota})\begin{pmatrix}C_{1,\iota} & \cdots & C_{k,\iota}\\
	\vdots &  & \vdots\\
	C_{(n-1)k+1,\iota} & \cdots & C_{nk,\iota}
	\end{pmatrix},
	\end{split}
	\end{equation*}
	so that $\frac{\partial}{\partial b_{\iota}}\tilde{\varphi}(\Theta)$ equals
	\begingroup\makeatletter\def\f@size{7}\check@mathfonts
	\def\maketag@@@#1{\hbox{\m@th\normalsize\normalfont#1}}
	\begin{align*}
	 &\sum_{i,j=1}^{n}\mu'(\omega_{i,j})\sigma'(Z_{\iota}) \times\\
	 & \quad \ 
	 \cdot\left(
	 \left[\begin{pmatrix}C_{1,\iota} & \cdots & C_{k,\iota}\\
	\vdots &  & \vdots\\
	C_{(n-1)k+1,\iota} & \cdots & C_{nk,\iota}
	\end{pmatrix}\begin{pmatrix}X_{_{1}} & \cdots & X_{(n-1)k+1}\\
	\vdots &  & \vdots\\
	X_{k} & \cdots & X_{nk}
	\end{pmatrix}\right]_{i,j}\right.
	\\
	& \quad\;\;\underbrace{\;\;\ 
	+\left.\left[\begin{pmatrix}X_{_{1}} & \cdots & X_{k}\\
	\vdots &  & \vdots\\
	X_{(n-1)k+1} & \cdots & X_{nk}
	\end{pmatrix}\begin{pmatrix}C_{1,\iota} & \cdots & C_{(n-1)k+1,\iota}\\
	\vdots &  & \vdots\\
	C_{k,\iota} & \cdots & C_{nk,\iota}
	\end{pmatrix}\right]_{i,j}\right)}_{\displaystyle\sum_{s=1}^{k}C_{(i-1)k+s,\iota}X_{(j-1)k+s}+X_{(i-1)k+s}C_{(j-1)k+s,\iota}},
	\end{align*}
	\endgroup
	and the first part is proved. For the second part, note that for every $1\le\eta\le r$ we have
	\begin{equation}\label{eq:pa-b-part2-1}
	\frac{\partial f_{1}^{A,b}(h(\Sigma))}{\partial A_{\iota,\eta}}=\frac{\partial(Ah(\Sigma)+b)}{\partial A_{\iota,\eta}}=\left(\begin{array}{c}
	[h(\Sigma)]_{\eta}1_{\iota=1}\\
	\vdots\\
	\left[h(\Sigma)\right]_{\eta}1_{\iota=\ell}
	\end{array}\right),
	\end{equation}
	which coincides with $[h(\Sigma)]_{\eta}\frac{\partial f_{1}^{A,b}(h(\Sigma))}{\partial b_{\iota}}$ by \eqref{eq:pa-b-comp-1}.
	Hence, using the same derivations as in \eqref{eq:pa-b-start} and \eqref{eq:pa-b-comp}, we get
	\begin{equation*}
	\frac{\partial\tilde{\varphi}(\Theta)}{\partial A_{\iota,\eta}}=[h(\Sigma)]_{\eta}\frac{\partial\tilde{\varphi}(\Theta)}{\partial b_{\iota}}.
	\end{equation*}
\end{proof}
%
%
%
\begin{lem}\label{lem:der-d}
	Let $\tilde{\varphi}$ be the function defined in \eqref{eq:tilde phi}, fix $1\le\alpha\le n$
	and $1\le\beta\le k$, and set $\nu:=(\alpha-1)k + \beta$. Then
	\begin{equation}\label{eq:in Lem 2}
	\frac{\partial\tilde{\varphi}(\Theta)}{\partial d_{\nu}}
	=
	2\sum_{j=1}^{n}\mu'(\omega_{\alpha,j})X_{(j-1)k+\beta}.
	\end{equation}
	Moreover, for every $1\leq \iota \leq \ell$, we have
	\[
	\frac{\partial\tilde{\varphi}(\Theta)}{\partial C_{\nu,\iota}}=Y_{\iota}\frac{\partial\tilde{\varphi}(\Theta)}{\partial d_{\nu}}.
	\]
\end{lem}
\begin{proof}
	Note that $\frac{\partial\tilde{\varphi}(\Theta)}{\partial d_{\nu}}$ equals
	\begingroup\makeatletter\def\f@size{7}\check@mathfonts
	\def\maketag@@@#1{\hbox{\m@th\normalsize\normalfont#1}}
	\begin{equation}\label{eq:pa-d-start}
	\sum_{i,j=1}^{n}\mu'(\omega_{i,j})\left[\tfrac{\partial g\left(f_{2}^{C,d}(Y)\right)}{\partial d_{\nu}}g\left(X\right)^{T}
	+
	g\left(X\right)\left(\tfrac{\partial g\left(f_{2}^{C,d}(Y)\right)^{T}}{\partial d_{\nu}}\right)\right]_{i,j}.
	\end{equation}
	\endgroup
	As $\nu=(\alpha-1)k + \beta$, the same calculation as in \eqref{eq:pa-b-comp-1}  and  \eqref{eq:pa-b-comp-2} ensures that
		\begingroup\makeatletter\def\f@size{7}\check@mathfonts
	\def\maketag@@@#1{\hbox{\m@th\normalsize\normalfont#1}}
	\begin{equation*}
	\begin{split}
	\frac{\partial g\left(f_{2}^{C,d}(Y)\right)}{\partial d_{\nu}}&=\left[\nabla_{X}g(X)\right]\frac{\partial\left(CY+d\right)}{\partial d_{\nu}}\\
	&=\left[\nabla_{X}g(X)\right]\left(\begin{array}{c}
	1_{\nu=1}\\
	\vdots\\
	1_{\nu=nk}
	\end{array}\right)=\begin{pmatrix}0 & \cdots & 0\\
	\vdots & 1_{(\alpha,\beta)} & \vdots\\
	0 & \cdots & 0
	\end{pmatrix},
	\end{split}
	\end{equation*}
	\endgroup
	which is a $n$-by-$k$ matrix. Hence, $\frac{\partial\tilde{\varphi}(\Theta)}{\partial d_{\nu}}$ equals
		\begingroup\makeatletter\def\f@size{7}\check@mathfonts
	\def\maketag@@@#1{\hbox{\m@th\normalsize\normalfont#1}}
	\begin{align*}
    &\sum_{i,j=1}^{n}\mu'(\omega_{i,j})\times\\
	& \quad \quad \quad \quad
	\left[\begin{pmatrix}0 & \cdots & 0\\
	\vdots & 1_{(\alpha,\beta)} & \vdots\\
	0 & \cdots & 0
	\end{pmatrix}\begin{pmatrix}X_{_{1}} & \cdots & X_{(n-1)k+1}\\
	\vdots &  & \vdots\\
	X_{k} & \cdots & X_{nk}
	\end{pmatrix}\right.\\
	& \quad \quad \quad \quad \quad
	\left. +
	\begin{pmatrix}X_{_{1}} & \cdots & X_{k}\\
	\vdots &  & \vdots\\
	X_{(n-1)k+1} & \cdots & X_{nk}
	\end{pmatrix}\begin{pmatrix}0 & \cdots & 0\\
	\vdots & 1_{(\beta,\alpha)} & \vdots\\
	0 & \cdots & 0
	\end{pmatrix}\right]_{i,j}\\
	& =\sum_{i,j=1}^{n}\mu'(\omega_{i,j})\times\\
	& \quad 
	\left[\underbrace{\begin{pmatrix}0 & \cdots & 0\\
		X_{\beta} & \cdots & X_{(n-1)k+\beta}\\
		0 & \cdots & 0
		\end{pmatrix}}_{\textrm{line}\;\alpha}+\underbrace{\begin{pmatrix}0 & \cdots & X_{\beta} & \cdots & 0\\
		\vdots &  & \vdots &  & \vdots\\
		0 & \cdots & X_{(n-1)k+\beta} & \cdots & 0
		\end{pmatrix}}_{\textrm{column}\;\alpha}\right]_{i,j}.
	\end{align*}
	\endgroup
	Since  $\omega_{i,j}=\omega_{j,i}$ for every $1\le i,j\le n$ we finally get
		\begingroup\makeatletter\def\f@size{7}\check@mathfonts
	\def\maketag@@@#1{\hbox{\m@th\normalsize\normalfont#1}}
	\begin{eqnarray*}
	\frac{\partial\tilde{\varphi}(\Theta)}{\partial d_{\nu}}&=& 2\mu'\!(\omega_{\alpha,\alpha}\!)X_{(\alpha-1)k+\beta}\\ 
	&& + \sum_{\substack{j=1\\
			j\neq\alpha
		}
	}^{n}\mu'\!(\omega_{\alpha,j}\!)X_{(j-1)k+\beta}
\!+\!\sum_{\substack{i=1\\
			i\neq\alpha
		}
	}^{n}\mu'\!(\omega_{i,\alpha}\!)X_{(i-1)k+\beta}\\ 
	&=&2\sum_{j=1}^{n}\mu'(\omega_{\alpha,j})X_{(j-1)k+\beta},
	\end{eqnarray*}
	\endgroup
	which proves the first part. For the second part, 
	we fix $1\le\iota\le\ell$. In the same manner as we established \eqref{eq:pa-b-part2-1}, we can use \eqref{eq:pa-b-comp-2} to see that
	\begin{eqnarray*}
	&&\frac{\partial g\left(f_{2}^{C,d}(Y)\right)}{\partial C_{\nu,\iota}}
	=
	[\nabla_{X}g(X)]\frac{\partial (CY+d)}{\partial C_{\nu,\iota}}\\
	&=&
	[\nabla_{X}g(X)]Y_{\iota}\left(\begin{array}{c}
	1_{\nu=1}\\
	\vdots\\
	1_{\nu=nk}
	\end{array}\right)=Y_{\iota}\frac{\partial g\left(f_{2}^{C,d}(Y)\right)}{\partial b_{\nu}},
	\end{eqnarray*}
	and thereby, with \eqref{eq:pa-d-start}, to conclude that
	\[
	\frac{\partial\tilde{\varphi}(\Theta)}{\partial C_{\nu,\iota}}=Y_{\iota}\frac{\partial\tilde{\varphi}(\Theta)}{\partial d_{\nu}}.
	\]
\end{proof}

One of the key tools in the derivation of bounds on the Lipschitz constant of $\tilde\varphi$ is the following elementary lemma. It shows how to infer the Lipschitz constant of some functions from the Lipschitz constant of other, simpler Lipschitz-continuous functions.

\begin{lem}\label{lem:Gen-Lip}
	Let $\phi_{1}:\mathbb{R}^{m}\to\mathbb{R}$ be Lipschitz continuous with Lipschitz constant $L_1$.
	Let $\phi_{2}:\mathbb{R}^{n}\to\mathbb{R}^m$ be a function for which there exists a some $L_2: \R^n\times \R^n \to (0,\infty)$ such that for all $x,y \in \R^n$ we have
	\begin{equation}\label{eq:Varying_Lipschitz_Constant}
	\|\phi_2(x)-\phi_2(y)\|\leq L_2(x,y) \|x-y\|.
	\end{equation}
	Assume that $\phi_1\circ\phi_2$ is bounded by a constant $B_{12}$.
	Finally, let  $\phi_{3}:\mathbb{R}^{n}\to\mathbb{R}$ be a function for which 
	\begin{enumerate} \item $|\phi_3(y)|\leq B_3(y)$  for all $y \in \R ^n$ for some positive function $B_3\colon\R^n \to (0,\infty)$;
		
		\item there exist three functions $L_{31}: \R^n\times \R^n \to (0,\infty)$, $L_{32}: \R^n\times \R^n \to (0,\infty)$, and  $f:\R^n \to \R^p$ such that for all $x,y \in \R^n$ we have
		\begingroup\makeatletter\def\f@size{7}\check@mathfonts
		\def\maketag@@@#1{\hbox{\m@th\normalsize\normalfont#1}}
		\begin{equation*}
		\|\phi_3(x)-\phi_3(y)\|\leq L_{31}(x,y) \|x-y\| +L_{32}(x,y)\|f(x)-f(y)\|.
		\end{equation*}
\endgroup	
\end{enumerate}
	Then, for every $x,y \in \R^n$, the function $\Phi:=(\phi_1 \circ \phi_2)\phi_3$ verifies
	\begingroup\makeatletter\def\f@size{7}\check@mathfonts
	\def\maketag@@@#1{\hbox{\m@th\normalsize\normalfont#1}}
	\begin{eqnarray*}
	|\Phi(x)-\Phi(y)|
	&\leq& \left(B_{12} L_{31}(x,y)+ B_3(y) L_1 L_2(x,y)\right) \|x-y\|\\
	 && +B_{12}L_{32}(x,y)\|f(x)-f(y)\|.
	\end{eqnarray*}
	\endgroup
\end{lem}
\begin{proof}
	For all $x,y\in\mathbb{R}^{n}$, we can write
	\begingroup\makeatletter\def\f@size{7}\check@mathfonts
	\def\maketag@@@#1{\hbox{\m@th\normalsize\normalfont#1}}
	\begin{equation*}
	\begin{split}
	&|\Phi(x)-\Phi(y)|=
	\left|\phi_{1}\left(\phi_{2}(x)\right)\phi_{3}(x)-\phi_{1}\left(\phi_{2}(y)\right)\phi_{3}(y)\right|\\
	& \leq  \left|\phi_{1}\left(\phi_{2}(x)\right)\left(\phi_{3}(x)-\phi_{3}(y)\right)\right|+\left|\left(\phi_{1}\left(\phi_{2}(x)\right)-\phi_{1}\left(\phi_{2}(y)\right)\right)\phi_{3}(y)\right|\\
	&\leq B_{12}\left|\phi_{3}(x)-\phi_{3}(y)\right|+B_{3}(y)\left|\phi_{1}\left(\phi_{2}(x)\right)-\phi_{1}\left(\phi_{2}(y)\right)\right|\\
	& \leq B_{12}\big(L_{31}(x,y) \|x\!-\!y\| \!+\!L_{32}(x,y)\|f(x)\!-\!f(y)\|\big)\\
	&\quad+\!B_{3}L_{1}\left\Vert \phi_{2}(x)\!-\!\phi_{2}(y)\right\Vert.
	\end{split}
	\end{equation*}
	\endgroup	
	It remains to use \eqref{eq:Varying_Lipschitz_Constant} on the last term to conclude.
\end{proof}

As defined in \eqref{eq:norm_definition}, the norm we use for a finite list of matrix of different sizes is the Euclidean norm of the vector constituted by their Frobenius norms.
For symmetric matrices $X$, we also use a dedicated norm defined
as $||X||_{S}:=||h(X)||$, where $h:\mathbb S^n \to \R^{n(n+1)/2}$ is the function defined in \eqref{eq:def-h}. Note also that $2||X||_{S}^{2}=||X||_{F}^{2}+||\textrm{diag}(X)||_{F}^{2}$.
%
%

The next lemma is a key step in determining a bound on the smoothness constant of our objective function $\tilde\varphi$.
\begin{lem}\label{lem:Lip-b}
	Let  $\tilde{\varphi}$ be the function defined in \eqref{eq:tilde phi} and let $D>0$ be any constant. We define
	\begingroup\makeatletter\def\f@size{7}\check@mathfonts
	\def\maketag@@@#1{\hbox{\m@th\normalsize\normalfont#1}}
	\begin{equation*}
	\mathcal{D}:=\big\{(A,b,C,d) \in \R^{l\times r} \times \R^\ell \times \R^{nk \times \ell} \times \R^{nk} \colon \|(A,b,C,d)\|\leq D\big\}.
	\end{equation*}
	\endgroup
	The function 
	\[
	\Theta=(A,d,C,d) \mapsto \frac{\partial \tilde \varphi(\Theta)}{\partial b} \in \R^\ell
	\]
	is Lipschitz continuous on $\mathcal{D}$ with a constant $L_b$ for which
	\begin{equation*}
	L^2_b = \mathcal{C}_bn^2D^2  \max\!\left\{\ell D^2L^2_Z, \ell^2D^6L_Z^2, \ell^3D^4,n\ell\right\},
	\end{equation*}
	where $L_Z:=\sqrt{1 + \|\Sigma\|_S^2}$ and $\mathcal{C}_b$ is a constant that only depends polynomially on  $\sigma'_{\max},\sigma''_{\max},\mu''_{\max}$.
\end{lem}
\begin{proof}
	We divide the proof into several steps, each of which establishing that some function is bounded and/or Lipschitz continuous. We set $\Theta=(A,b,C,d)$ and $\bar\Theta = (\bar A,\bar b,\bar C, \bar d)$ in $\mathcal{D}$. We make use abundantly of the shorthand notation defined in \eqref{shortcut-notation}, adopting the notation $\bar\omega_{ij}$, $\bar X$, etc\ldots when $\Theta$ is replaced by $\bar\Theta$.

		\noindent \textbf{Step 1.} Let $1\le\iota\le\ell$. In this first step, we focus on $Z_\iota = \left[f_1^{A,b}(h(\Sigma))\right]_\iota$. The Cauchy-Schwarz Inequality ensures that for every $(A,b), (\bar{A},\bar{b}) \in \R^{\ell\times r} \times \R^\ell$ we have
	\begingroup\makeatletter\def\f@size{9}\check@mathfonts
	\def\maketag@@@#1{\hbox{\m@th\normalsize\normalfont#1}}
		\begin{equation*}
		\begin{split}
		&\left(Z_{\iota}-\bar Z_{\iota}\right)^{2} =
		\Big((A_{\iota,:}-\bar{A}_{\iota,:})h(\Sigma)+b_{\iota}-\bar{b}_{\iota}\Big)^{2}\\
		&\leq
		\left\Vert \begin{array}{c}
		A_{\iota,:}-\bar{A}_{\iota,:}\\
		b_{\iota}-\bar{b}_{\iota}
		\end{array}\right\Vert ^{2}\left\Vert \begin{array}{c}
		h(\Sigma)\\
		1
		\end{array}\right\Vert ^{2} = L_Z^2 ||({A}_{\iota,:},{b}_{\iota})-(\bar {A}_{\iota,:},\bar{b}_{\iota})||^2.
		\end{split}
		\end{equation*}
		\endgroup

        \noindent \textbf{Step 2.} Let us apply Lemma \ref{lem:Gen-Lip} with 
		$\phi_{1}\equiv\sigma$,
		$\phi_{2}\equiv\left[f_{1}^{\cdot,\cdot}(h(\Sigma))\right]_{\iota}$,
		and $\phi_{3}\equiv 1$. We have immediately $L_{31}=L_{32}=0$, $B_{3}=1$, $L_{1}=\sigma'_{\max}$, and, by Step $1$, $L_{2}=L_{Z}$. Therefore,
	\begingroup\makeatletter\def\f@size{9}\check@mathfonts
	\def\maketag@@@#1{\hbox{\m@th\normalsize\normalfont#1}}
		\begin{equation*}
		\left|\sigma\left(Z_{\iota}\right)-\sigma\left(\bar Z_{\iota}\right)\right|\leq \sigma'_{\max} L_Z ||({A}_{\iota,:},{b}_{\iota})-(\bar {A}_{\iota,:},\bar{b}_{\iota})||.
		\end{equation*}
		\endgroup
Note also that $\|\tilde{\sigma}(\bar Z)\|\le\sqrt\ell$ because $|\sigma\left(\bar Z_{\iota}\right)|\le 1$ by assumption on $\sigma$.

		\noindent \textbf{Step 3.} Similarly, we get 
	\begingroup\makeatletter\def\f@size{9}\check@mathfonts
	\def\maketag@@@#1{\hbox{\m@th\normalsize\normalfont#1}}
		\begin{equation*}
		\left|\sigma'\left(Z_{\iota}\right)-\sigma'\left(\bar Z_{\iota}\right)\right|\leq \sigma''_{\max} L_Z ||({A}_{\iota,:},{b}_{\iota})-(\bar {A}_{\iota,:},\bar{b}_{\iota})||.
		\end{equation*}
		\endgroup
		Also, $\sigma'\left(Z_{\iota}\right)\le \sigma'_{\max}$ by assumption.

		\noindent \textbf{Step 4.} Let us fix $1\le\nu\le nk$. We focus in this step on $X_\nu = \left[f_{2}^{{C,}{d}}\circ\tilde{\sigma}\circ f_{1}^{{A},{b}}(h(\Sigma))\right]_{\nu}=\left[f_{2}^{{C,}{d}}(Y)\right]_{\nu}$, with $Y = \tilde\sigma(Z) = \tilde{\sigma}\circ f_{1}^{{A},{b}}(h(\Sigma))$. Observe first that
		\begingroup\makeatletter\def\f@size{9}\check@mathfonts
		\def\maketag@@@#1{\hbox{\m@th\normalsize\normalfont#1}}
		\begin{align*}
		&
		\left|X_{\nu}
		-\bar X_{\nu}\right| = \left|{C}_{\nu,:}{Y}+{d}_{\nu}-\bar{C}_{\nu,:}{\bar Y}-\bar{d}_{\nu}\right|\\
		&\le \sum_{\iota=1}^{\ell}{|C}_{\nu,\iota}|\cdot|Y_{\iota}-\bar{Y}_{\iota}|+\left\|C_{\nu,:}-\bar{C}_{\nu,:}\right\|\cdot\left\|\bar Y\right\|+\left|d_{\nu}-\bar{d}_{\nu}\right|\\
		&\le \sum_{\iota=1}^{\ell}|{C}_{\nu,\iota}|\sigma'_{\max}L_{Z}\cdot\left\Vert 
		\begin{array}{c}
		A_{\iota,:}-\bar{A}_{\iota,:} \nonumber\\
		b_{\iota}-\bar{b}_{\iota}
		\end{array}\right\Vert\\
		&\;\;+\left\|C_{\nu,:}-\bar{C}_{\nu,:}\right\|\left \|\bar Y\right\|+\left|d_{\nu}-\bar{d}_{\nu}\right|\qquad\textrm{(by Step $2$).}
        \end{align*}
        \endgroup
        Squaring both sides of this inequality and using Cauchy-Schwarz, we get
		\begingroup\makeatletter\def\f@size{9}\check@mathfonts
		\def\maketag@@@#1{\hbox{\m@th\normalsize\normalfont#1}}
		\begin{align}
		& \left(X_{\nu}
		-\bar X_{\nu}\right)^2\nonumber\\
		& \leq  \left(\left(\sum_{\iota=1}^{\ell}{C}_{\nu,\iota}^{2}\right)\left(\sigma'_{\max}\right)^{2}L_{Z}^{2}+||\bar Y||^{2}+1\right) \times\nonumber\\
		& \;\;\;\;\;\;\Big(||A-\bar{A}||^{2}+||b-\bar{b}||^{2}+||C_{\nu,:}-\bar{C}_{\nu,:}||^{2}+\left|d_{\nu}-\bar{d}_{\nu}\right|^{2}\Big)
		\nonumber\\
		& \le  L^2_\nu\left\|(A,b,C_{\nu,:},d_\nu)-(\bar A,\bar b,\bar C_{\nu,:}, \bar d_\nu)\right\|^2. \label{step-4-Lip}
		\end{align}
		\endgroup
		with 
        \begingroup\makeatletter\def\f@size{9}\check@mathfonts
		\def\maketag@@@#1{\hbox{\m@th\normalsize\normalfont#1}}
		\begin{equation}\label{eq:L_nu}
		L_\nu^2:=\left(||{C}_{\nu,:}||\sigma'_{\max}L_{Z}\right)^{2}+\ell+1.
        \end{equation}
            \endgroup
            This last inequality is also ensured by Step $2$, where we showed that $||\bar Y||\le \sqrt\ell$.
 
		In addition to providing an estimate for its Lipschitz constant, we can also obtain a bound for this function. From Step $2$, we get
		\begin{align}
		\left|X_{\nu}\right|&=\left|{C}_{\nu,:}{Y}+{d}_{\nu}\right|
		\leq
		||{C}_{\nu,:}||\sqrt{\ell}+|{d}_{\nu}|\nonumber\\
		&\le
		\sqrt{\ell+1}\cdot\sqrt{||{C}_{\nu,:}||^{2}+|{d}_{\nu}|^{2}}. \label{step-4-bound}
		\end{align}
We obtain with this step the Lipschitz constant of one output of a single-layer neural network. Would our objective function be the standard least square loss function to minimize when training a neural network, the remaining of our task would have been vastly simpler. The specific intricacies of our objective function \eqref{eq:tilde phi} ask for a more involved analysis.

        \noindent \textbf{Step 5.} Let $1\le i,j\le n$. In this step, we deal with the function
        \begingroup\makeatletter\def\f@size{9}\check@mathfonts
		\def\maketag@@@#1{\hbox{\m@th\normalsize\normalfont#1}}
        \begin{equation*}
        \omega_{i,j}:=\left[g\left(X\right)g\left(X\right)^{T}-\Sigma\right]_{i,j}.
        \end{equation*} 
        \endgroup
		We have
		\begingroup\makeatletter\def\f@size{7}\check@mathfonts
		\def\maketag@@@#1{\hbox{\m@th\normalsize\normalfont#1}}
		\begin{align}
		& \big|\omega_{i,j}-\bar{\omega}_{i,j}\big|^{2} = \left[g\left(X\right)g\left(X\right)^{T}-g\left(\bar X\right)g\left(\bar X\right)^{T}\right]^2_{i,j}\nonumber\\
		& \le  \left(\sum_{s=1}^{k}\Big[\left|X_{(i-1)k+s}\right|\cdot\left|X_{(j-1)k+s}-\bar{X}_{(j-1)k+s}\right|\right.\nonumber\\
		& \quad \quad \quad +\left|\bar{X}_{(j-1)k+s}\right|\cdot\left|X_{(i-1)k+s}-\bar{X}_{(i-1)k+s}\right|\Big]\Bigg)^{2}\label{eq:Step 5 initial}\\
		& \le  \;2\left(\sum_{s=1}^{k}X_{(i-1)k+s}^{2}\right)\cdot\left(\sum_{s=1}^{k}|X_{(j-1)k+s}-\bar{X}_{(j-1)k+s}|^{2}\right)\nonumber\\ &\quad
		+2\left(\sum_{s=1}^{k}\bar{X}_{(j-1)k+s}^{2}\right)\cdot\left(\sum_{s=1}^{k}|X_{(i-1)k+s}-\bar{X}_{(i-1)k+s}|^{2}\right).\nonumber
		\end{align}
		\endgroup
		Let us bound these four sums using \eqref{step-4-Lip} and \eqref{step-4-bound}.
		First, 
			\begingroup\makeatletter\def\f@size{7}\check@mathfonts		\def\maketag@@@#1{\hbox{\m@th\normalsize\normalfont#1}}
		\begin{equation*}
		 \sum_{s=1}^{k}X_{(i-1)k+s}^{2}\le (\ell +1)\left(\sum_{s=1}^{k}\left(\|C_{(i-1)k+s,:}\|^2 +d_{(i-1)k+s}^2\right)\right).
		\end{equation*}
		\endgroup
		The sum on the right-hand side can be conveniently rewritten using the vector-to-matrix operator $g$.
        \begingroup\makeatletter\def\f@size{7}\check@mathfonts
		\def\maketag@@@#1{\hbox{\m@th\normalsize\normalfont#1}}
		\begin{equation}
		\sum_{s=1}^{k}\|C_{(i-1)k+s,:}\|^2
		= \sum_{s=1}^{k} \sum_{\iota=1}^\ell C_{(i-1)k+s,\iota}^2=\sum_{\iota=1}^\ell \| g(C_{:,\iota})_{i,:}\|^2.\label{M-transf}
		\end{equation}
		\endgroup
		Similarly, we can write $\sum_{s=1}^{k}\|d_{(i-1)k+s}\|^2 = \|g(d)_{i,:}\|^2$, so that 
        \begingroup\makeatletter\def\f@size{9}\check@mathfonts
		\def\maketag@@@#1{\hbox{\m@th\normalsize\normalfont#1}}
		\begin{equation*}
		 \sum_{s=1}^{k}X_{(i-1)k+s}^{2}\le (\ell +1)\left(\sum_{\iota=1}^\ell \| g(C_{:,\iota})_{i,:}\|^2 +\|g(d)_{i,:}\|^2\right).
		\end{equation*}
		\endgroup
        Second, with $ \Theta_{I(\nu)}:=(A,b,C_{\nu,:},d_\nu)$ and writing $\nu(s):=(j-1)k+s$ for $1\le s\le k$, we have
			\begingroup\makeatletter\def\f@size{9}\check@mathfonts		\def\maketag@@@#1{\hbox{\m@th\normalsize\normalfont#1}}
        \begin{equation*}
		 \sum_{s=1}^{k}|X_{\nu(s)}-\bar{X}_{\nu(s)}|^{2}\le\sum_{s=1}^{k}L^2_{\nu(s)}\left\|\Theta_{I(\nu(s))}-\bar\Theta_{I(\nu(s))}\right\|^2.
		\end{equation*}
		\endgroup
		We can use the crude bound 
			\begingroup\makeatletter\def\f@size{9}\check@mathfonts		\def\maketag@@@#1{\hbox{\m@th\normalsize\normalfont#1}}
        \begin{equation*}
		 \sum_{s=1}^{k}\left\|\Theta_{I(\nu(s))}-\bar\Theta_{I(\nu(s))}\right\|^2\le \left\|\Theta-\bar\Theta\right\|^2
		\end{equation*}
		\endgroup
		to get
			\begingroup\makeatletter\def\f@size{9}\check@mathfonts		\def\maketag@@@#1{\hbox{\m@th\normalsize\normalfont#1}}
        \begin{align*}
&\sum_{s=1}^{k}L^2_{\nu(s)}\left\|\Theta_{I(\nu(s))}-\bar\Theta_{I(\nu(s))}\right\|^2\\
&\le \left(\left(\sigma'_{\max}L_{Z}\right)^{2}\sum_{s=1}^{k}||{C}_{\nu(s),:}||^{2}+\ell+1\right)\left\|\Theta-\bar\Theta\right\|^2\\
&=\left(\left(\sigma'_{\max}L_{Z}\right)^{2}\sum_{\iota=1}^\ell \| g(C_{:,\iota})_{j,:}\|^2+\ell+1\right)\left\|\Theta-\bar\Theta\right\|^2
\end{align*}
		\endgroup
		by \eqref{M-transf}. We can plug our two estimates into our bound for $|\omega_{i,j}-\bar{\omega}_{i,j}|^{2}$ to get
        \begingroup\makeatletter\def\f@size{9}\check@mathfonts
		\def\maketag@@@#1{\hbox{\m@th\normalsize\normalfont#1}}
		\begin{equation*}
		|\omega_{i,j}-\bar{\omega}_{i,j}| \le L_{\omega,i,j}\left\|\Theta-\bar\Theta\right\|
		\end{equation*}
		\endgroup
with 
        \begingroup\makeatletter\def\f@size{7}\check@mathfonts
		\def\maketag@@@#1{\hbox{\m@th\normalsize\normalfont#1}}
		\begin{align}
		 &\left.L_{\omega,i,j}^2\middle/\left(2\left(\ell+1\right)\left(\sigma'_{\max}L_{Z}\right)^{2}\right)\right.\label{step-5-Lip}\\
		 &=\left(\sum_{\iota=1}^{\ell}||g({C}_{:,\iota})_{i,:}||^{2}+||g({d})_{i,:}||^{2}\right)\left(\sum_{\iota=1}^{\ell}||g({C}_{:,\iota})_{j,:}||^{2}+\kappa\right)\nonumber\\
		 &\;\;+\left(\sum_{\iota=1}^{\ell}||g({\bar C}_{:,\iota})_{j,:}||^{2}+||g(\bar{d})_{j,:}||^{2}\right)\left(\sum_{\iota=1}^{\ell}||g({C}_{:,\iota})_{i,:}||^{2}+\kappa\right),\nonumber
		\end{align}
		where, $\kappa:=(\ell+1)/\left(\sigma'_{\max}L_{Z}\right)^{2}$.
		\endgroup

		\noindent \textbf{Step 6.}  In this step, we fix $1\le i,j\le n$ and $1\le\iota\le\ell$ and focus on $\mu'\left({\omega}_{i,j}\right)\sigma'({Z}_{\iota})$. We therefore seek to apply Lemma~\ref{lem:Gen-Lip} with $\phi_{1}\equiv\mu'$, $\phi_{2}\equiv{\omega}_{i,j}$, and $\phi_{3}\equiv\sigma'({Y}_{\iota})$. The assumptions on $\mu$ allow us to take $L_{1}=\mu''_{\max}$ and $B_{12}=1$. Step $5$ shows that $L_{2}$ can be taken as $L_{\omega,i,j}$. Step $3$ allows us to pick $L_{31}=\sigma''_{\max}L_{Z}$ with $L_{32}=0$ and $B_{3}=\sigma'_{\max}$ by our assumptions on $\sigma$.

		We obtain that
		\begingroup\makeatletter\def\f@size{9}\check@mathfonts
		\def\maketag@@@#1{\hbox{\m@th\normalsize\normalfont#1}}
		\begin{equation*}
		\left|\mu'\left({\omega}_{i,j}\right)\sigma'({Z}_{\iota}) - 
		\mu'\left(\bar{\omega}_{i,j}\right)\sigma'(\bar{Z}_{\iota})\right|\\ 
		\leq 
		L_{\mu'\sigma',i,j} \|\Theta-\bar \Theta\|
		\end{equation*} 
		\endgroup
		with 
		\begingroup\makeatletter\def\f@size{9}\check@mathfonts
		\def\maketag@@@#1{\hbox{\m@th\normalsize\normalfont#1}}
		\begin{equation*}
		L_{\mu'\sigma',i,j} := \sigma''_{\max}L_{Z}+ \sigma'_{\max}\mu''_{\max}L_{\omega,i,j}.
		\end{equation*} 
		\endgroup
		Evidently, we also have $\mu'\left({\omega}_{i,j}\right)\sigma'({Z}_{\iota})\le \sigma'_{\max}$.

        \noindent \textbf{Step 7.} The formula for $\frac{\partial\tilde{\varphi}(\Theta)}{\partial b_{\iota}}$ obtained in Lemma \ref{lem:der-b} appears as a weighted sum of sums. Step $6$ did focus on the weights of this combination. In this step, we deal with these sums. We fix again $1\le i,j\le n$ and $1\le\iota\le\ell$ and define  
	\begingroup\makeatletter\def\f@size{9}\check@mathfonts
	\def\maketag@@@#1{\hbox{\m@th\normalsize\normalfont#1}}
	\begin{equation*}
	V_{i,j,\iota}:=\sum_{s=1}^{k}C_{(i-1)k+s,\iota}X_{(j-1)k+s}.
	\end{equation*}
    \endgroup
    We define similarly $\bar V_{i,j,\iota}$ according to our usual convention to replace $\Theta$ by $\bar\Theta$. Using the same argument as in Step $5$, we can write
        \begingroup\makeatletter\def\f@size{7}\check@mathfonts
		\def\maketag@@@#1{\hbox{\m@th\normalsize\normalfont#1}}
		\begin{align*}
		& \big|V_{i,j,\iota}-\bar V_{i,j,\iota}\big|^{2}\\
		& \le  \left(\sum_{s=1}^{k}\Big[\left|C_{(i-1)k+s,\iota}\right|\cdot\left|X_{(j-1)k+s}-\bar{X}_{(j-1)k+s}\right|\right.\\
		& \quad \quad \quad +\left|\bar{X}_{(j-1)k+s}\right|\cdot\left|C_{(i-1)k+s,\iota}-\bar{C}_{(i-1)k+s,\iota}\right|\Big]\Bigg)^{2}\\
		& \le  \;2\left(\sum_{s=1}^{k}C_{(i-1)k+s,\iota}^{2}\right)\cdot\left(\sum_{s=1}^{k}|X_{(j-1)k+s}-\bar{X}_{(j-1)k+s}|^{2}\right)\\ &\quad
		+2\left(\sum_{s=1}^{k}\bar{X}_{(j-1)k+s}^{2}\right)\cdot\left(\sum_{s=1}^{k}|C_{(i-1)k+s,\iota}-\bar{C}_{(i-1)k+s,\iota}|^{2}\right).
		\end{align*}
		\endgroup
        We can bound the sums involving $X$ or $\bar X$ as in Step $5$.  The sums with $C,\bar C$ can we rewritten using \eqref{M-transf}; in particular, 
	\begingroup\makeatletter\def\f@size{9}\check@mathfonts
	\def\maketag@@@#1{\hbox{\m@th\normalsize\normalfont#1}}
	\begin{equation*}
	\sum_{s=1}^{k}|C_{(i-1)k+s,\iota}-\bar{C}_{(i-1)k+s,\iota}|^{2} = ||g(C_{:,\iota})_{i,:}-g(\bar{C}_{:,\iota})_{i,:}||^{2}.
	\end{equation*}
    \endgroup
    We obtain
            \begingroup\makeatletter\def\f@size{7}\check@mathfonts
		\def\maketag@@@#1{\hbox{\m@th\normalsize\normalfont#1}}
		\begin{align*}
		& \big|V_{i,j,\iota}-\bar V_{i,j,\iota}\big|^{2}\\
		& \le L_{CdX,i,j,\iota}^{2} \left\|\Theta-\bar\Theta\right\|^2+L_{XdC,i,j,\iota}^{2}||g(C_{:,\iota})_{i,:}-g(\bar{C}_{:,\iota})_{i,:}||^{2}.
		\end{align*}
		\endgroup       
    with 
		\begingroup\makeatletter\def\f@size{9}\check@mathfonts
		\def\maketag@@@#1{\hbox{\m@th\normalsize\normalfont#1}}
		\begin{align*}
		L^2_{CdX,i,j,\iota}
		&:=2||g({C}_{:,\iota})_{i,:}||^{2}\times\\
		& \quad 
		\bigg(\left(\sigma'_{\max}L_{Z}\right)^{2}\textstyle \sum\limits_{\jmath=1}^{\ell}||g({C}_{:,\jmath})_{j,:}||^{2}+\ell+1\bigg),\\
		L_{XdC,j,\iota}^{2}
		&:= 2(\ell+1)
		\bigg(\textstyle\sum\limits_{\jmath=1}^{\ell}||g(\bar{C}_{:,\jmath})_{j,:}||^{2}+||g(\bar{d})_{j,:}||^{2}\bigg).
		\end{align*}
		\endgroup
		
We can also obtain an upper bound to $\big|V_{i,j,\iota}\big|$ 
due to
			\begingroup\makeatletter\def\f@size{9}\check@mathfonts
		\def\maketag@@@#1{\hbox{\m@th\normalsize\normalfont#1}}
		\begin{align*}
		V_{i,j,\iota}^2
		&\leq \left(\sum_{s=1}^k C_{(i-1)k+s,\iota}^2\right) \left(\sum_{s=1}^k X_{(j-1)k+s}^2\right)\\
		&\le||g({C}_{:,\iota})_{i,:}||^2(\ell+1)\left(\sum_{\jmath=1}^{\ell}||g({C}_{:,\jmath})_{j,:}||^{2}+||g({d})_{j,:}||^{2}\right)\\
		&=:B^2_{CX,i,j,\iota}
		\end{align*} 
		\endgroup
		by using Cauchy-Schwarz, \eqref{M-transf}, and a bound obtained in Step $5$.

		\noindent \textbf{Step 8.} We start this final step by applying Lemma \ref{lem:Gen-Lip} to the function $\Upsilon_{i,j,\iota}:=\mu'(\omega_{i,j})\sigma'(Z_{\iota})\left(V_{i,j,\iota}+V_{j,i,\iota}\right)$ by taking $\phi_1\equiv id_{\R}$, $\phi_2\equiv \mu'(\omega_{i,j})\sigma'(Z_{\iota})$, and 
		$\phi_3\equiv V_{i,j,\iota}+V_{j,i,\iota}$. Evidently $L_1=1$. In Step $6$, we have shown that we can take $B_{12}=\sigma'_{\max}$ and $L_2= L_{\mu'\sigma',i,j}$. By Step $7$, we can set $B_3=B_{CX,i,j,\iota}+B_{CX,j,i,\iota}$, $L_{31}=L_{CdX,i,j,\iota}+L_{CdX,j,i,\iota}$, and $L_{32}=0$. We obtain that 
        \begingroup\makeatletter\def\f@size{9}\check@mathfonts
        \def\maketag@@@#1{\hbox{\m@th\normalsize\normalfont#1}}
		\begin{align*}
 		&\left|\Upsilon_{i,j,\iota} - \bar\Upsilon_{i,j,\iota}\right|\\
 		&\leq
		L_{\mu'\sigma',i,j}
		\big(B_{CX,i,j,\iota}+B_{CX,j,i,\iota}\big)||\Theta-\bar{\Theta}||\\
		& \quad +\sigma'_{\max}(L_{CdX,i,j,\iota}+L_{CdX,j,i,\iota}\big)||\Theta-\bar{\Theta}||\\
		& \quad 
		+\sigma'_{\max}\, L_{XdC,j,\iota}\,
		||g(C_{:,\iota})_{i,:}-g(\bar{C}_{:,\iota})_{i,:}||\\
		& \quad
		+
		\sigma'_{\max} \,L_{XdC,i,\iota}\,
		||g(C_{:,\iota})_{j,:}-g(\bar{C}_{:,\iota})_{j,:}||.		
		\end{align*}
		\endgroup
		Since in view of Lemma \ref{lem:Lip-b}, $\frac{\partial\tilde{\varphi}(\Theta)}{\partial b_{\iota}}$ is the sum of $\Upsilon_{i,j,\iota}$ over $i,j$, we obtain immediately the Lipschitz continuity of this partial derivative. Squaring the above bound and using Cauchy-Schwarz on the 6-terms right-hand side, we get
		\begingroup\makeatletter\def\f@size{7}\check@mathfonts
		\def\maketag@@@#1{\hbox{\m@th\normalsize\normalfont#1}}
		\begin{align*}
		&\sum_{\iota=1}^{\ell}\left(\frac{\partial\tilde{\varphi}(\Theta)}{\partial b_{\iota}}-\frac{\partial\tilde{\varphi}(\bar{\Theta})}{\partial b_{\iota}}\right)^{2} = \sum_{\iota=1}^{\ell}\left(\sum_{i,j=1}^n\Upsilon_{i,j,\iota}-\bar\Upsilon_{i,j,\iota}\right)^{2}
		\\
		&\le  6n^{2}\sum_{\iota=1}^{\ell}\sum_{i,j=1}^{n}L_{\mu'\sigma',i,j}^{2} B_{CX,i,j,\iota}^{2}||\Theta-\bar{\Theta}||^{2}\\
		& \quad
		+
		6n^{2}\sum_{\iota=1}^{\ell}\sum_{i,j=1}^{n}L_{\mu'\sigma',i,j}^{2}B_{CX,j,i,\iota}^{2}||\Theta-\bar{\Theta}||^{2}\\
		& \quad  +6\left(n\sigma'_{\max}\right)^{2}\sum_{\iota=1}^{\ell}\sum_{i,j=1}^{n}L_{CdX,i,j,\iota}^{2}||\Theta-\bar{\Theta}||^{2}\\
		& \quad +6\left(n\sigma'_{\max}\right)^{2}\sum_{\iota=1}^{\ell}\sum_{i,j=1}^{n}L_{CdX,j,i,\iota}^{2}||\Theta-\bar{\Theta}||^{2}
		\\
		& \quad
		+6\left(n\sigma'_{\max}\right)^{2}\sum_{\iota=1}^{\ell}\sum_{i,j=1}^{n}L_{XdC,j,\iota}^{2}||g(C_{:,\iota})_{i,:}-g(\bar{C}_{:,\iota})_{i,:}||^{2}
		\\
		& \quad 
		+6\left(n\sigma'_{\max}\right)^{2}\sum_{\iota=1}^{\ell} \sum_{i,j=1}^{n}L_{XdC,i,\iota}^{2}||g(C_{:,\iota})_{j,:}-g(\bar{C}_{:,\iota})_{j,:}||^{2}.
		\end{align*}
		\endgroup
		It remains to estimate these six sums. For the first one (and the second one by symmetry), we can proceed as follows.
		\begingroup\makeatletter\def\f@size{7}\check@mathfonts
		\def\maketag@@@#1{\hbox{\m@th\normalsize\normalfont#1}}
		\begin{align*}
		& \sum_{\iota=1}^{\ell}\sum_{i,j=1}^{n}L_{\mu'\sigma',i,j}^{2}B_{CX,i,j,\iota}^{2}\\
		&=  2\left(\ell+1\right)\left(\sigma''_{\max}L_{Z}\right)^{2}\times\\
		&\qquad\sum_{\iota=1}^{\ell}\sum_{i,j=1}^{n}||g({C}_{:,\iota})_{i,:}||^{2}\left(\sum_{\jmath=1}^{\ell}\!||g({C}_{:,\jmath})_{j,:}||^{2}\!+\!||g({d})_{j,:}||^{2}\right)\\
		& \quad +2\left(\ell+1\right)\left(\sigma'_{\max}\mu''_{\max}\right)^{2}\times\\
		& \qquad\sum_{\iota=1}^{\ell}\sum_{i,j=1}^{n}L_{\omega,i,j}^{2}||g({C}_{:,\iota})_{i,:}||^{2}\left(\sum_{\jmath=1}^{\ell}||g({C}_{:,\jmath})_{j,:}||^{2}+||g({d})_{j,:}||^{2}\right)\\
		& \leq  2\left(\ell+1\right)\left(\sigma''_{\max}L_{Z}\right)^{2}||{C}||^{2}\left(||{C}||^{2}+||{d}||^{2}\right)\\
		& \quad +8\left(\ell+1\right)^{2}\left(\sigma'_{\max}\mu''_{\max}\right)^{2}\left(||{C}||^{2}\left(\sigma'_{\max}L_{Z}\right)^{2}+\ell+1\right)\times\\
		& \qquad  \|C\|^2\left(||{C}||^{2}+||{d}||^{2}\right)^{2}.
		\end{align*}
		\endgroup
		To estimate the third and fourth sum, notice that
			\begingroup\makeatletter\def\f@size{7}\check@mathfonts
		\def\maketag@@@#1{\hbox{\m@th\normalsize\normalfont#1}}
		\begin{align*}
		&\sum_{\iota=1}^{\ell}\sum_{i,j=1}^{n}L_{CdX,i,j,\iota}^{2}\\
		&=2\sum_{\iota=1}^\ell\sum_{i,j=1}^{n}||g({C}_{:,\iota})_{i,:}||^{2}\left(\left(\sigma'_{\max}L_Z\right)^{2}\sum_{\jmath=1}^{\ell}||g({C}_{:,\jmath})_{j,:}||^{2}+\ell+1\right)\\
		&\le 2\sum_{j=1}^{n}||{C}||^{2}\left(\sigma'_{\max}L_Z\right)^{2}
		\sum_{\jmath=1}^{\ell}||g({C}_{:,\jmath})_{j,:}||^{2}+ 2\sum_{j=1}^{n}||{C}||^{2}\left(\ell+1\right)\\
		&\leq 2 ||{C}||^{2}\left(\left(||{C}||\sigma'_{\max}L_{Z}\right)^{2}+n\left(\ell+1\right)\right).
		\end{align*}
		\endgroup
		For the last two terms, we can write
			\begingroup\makeatletter\def\f@size{7}\check@mathfonts
		\def\maketag@@@#1{\hbox{\m@th\normalsize\normalfont#1}}
		\begin{align*}
		& \sum_{\iota=1}^{\ell}\sum_{i,j=1}^{n}L_{XdC,j,\iota}^{2}||g(C_{:,\iota})_{i,:}-g(\bar{C}_{:,\iota})_{i,:}||^{2}\\
		&=  2(\ell+1)\times\\
		& 
		\quad\sum_{\iota,\jmath=1}^{\ell}\sum_{i,j=1}^{n}\left(||g(\bar{C}_{:,\jmath})_{j,:}||^{2}+||g(\bar{d})_{j,:}||^{2}\right)||g(C_{:,\iota})_{i,:}-g(\bar{C}_{:,\iota})_{i,:}||^{2}\\
		&= 2(\ell+1)\left(||\bar{C}||^{2}+||\bar{d}||^{2}\right)||C-\bar C||^{2}.
		\end{align*}
		\endgroup
		Therefore, using all the estimates for the six terms, we arrive at
		\begingroup\makeatletter\def\f@size{7}\check@mathfonts		\def\maketag@@@#1{\hbox{\m@th\normalsize\normalfont#1}}
		\begin{equation*}
		\begin{split}
		&\sum_{\iota=1}^{\ell}\left(\frac{\partial\tilde{\varphi}(\Theta)}{\partial b_{\iota}}-\frac{\partial\tilde{\varphi}(\bar{\Theta})}{\partial b_{\iota}}\right)^{2}\\
		&\leq
		\!24n^{2}||\bar{\Theta}||^{2}
		\Bigg(\!\left(\ell+1\right)\left(\sigma''_{\max}L_{Z}\right)^{2}||\bar{\Theta}||^{2}\\
		& \quad \quad \quad \quad \ \
		+
		4\left(\ell+1\right)^{2}\left(\mu''_{\max}L_{Z}\right)^{2}\left(\sigma'_{\max}\right)^{4}||\bar{\Theta}||^{6}
	\\
	&\quad \quad \quad \quad \ \
		+
		4\left(\ell+1\right)^{3}\left(\sigma'_{\max}\mu''_{\max}\right)^{2}||\bar{\Theta}||^{4}
		\\
		& \quad \quad \quad \quad \ \ 
		+
		||\bar{\Theta}||^{2}\!\left(\!\sigma'_{\max}\!\right)^{4}\!L_{Z}^{2}
		\!+\!
		(n\!+\!1)\left(\ell\!+\!1\right)\left(\sigma'_{\max}\!\right)^{2}\Bigg)
		\|\Theta\!-\!\bar \Theta\|^2.
		\end{split}
		\end{equation*}
		\endgroup
		We conclude that
		\begingroup\makeatletter\def\f@size{7}\check@mathfonts
		\def\maketag@@@#1{\hbox{\m@th\normalsize\normalfont#1}}
		\begin{equation*}
		\begin{split}
		L^2_b 
		&= O\Bigg(n^2D^2 \max\!\bigg\{\ell \left(D\sigma''_{\max}L_Z\right)^2 , \left(\ell\mu''_{\max}L_Z\right)^2D^6(\sigma'_{\max})^4\,\\
 		& \quad \quad \quad \ell^3D^4\left(\sigma'_{\max}\mu''_{\max}\right)^{2},(DL_Z)^2(\sigma'_{\max})^4,n\ell(\sigma'_{\max})^2\bigg\}
 		\Bigg)\\
		&= O\Bigg(n^2D^2 \max\!\bigg\{\ell (DL_Z)^2 \max\!\big\{(\sigma''_{\max})^2,(\sigma'_{\max})^4\big\} ,\\
		& \quad \quad  \left(\ell\mu''_{\max}L_Z\right)^2D^6\!(\sigma'_{\max})^4, \!\ell^3\!D^4\left(\sigma'_{\max}\,\mu''_{\max}\right)^{2}\!\!,n\ell(\sigma'_{\max})^2\bigg\}
		\Bigg)\\
		&= O\Bigg(\mathcal{C}_b\,n^2D^2  \max\!\big\{\ell D^2L^2_Z, \ell^2D^6L_Z^2, \ell^3D^4,n\ell\big\}
		\Bigg),
		\end{split}
		\end{equation*}
		\endgroup
		where $\mathcal{C}_b\equiv\mathcal{C}_b(\sigma'_{\max},\sigma''_{\max},\mu''_{\max})$ is a constant that only depends polynomially on  $\sigma'_{\max},\sigma''_{\max},\mu''_{\max}$.
\end{proof}

In the next lemma, we determine a bound for the Lipschitz constant of $
\frac{\partial \tilde \varphi(\Theta) }{\partial d_\nu}$. Lemmas \ref{lem:der-b} and \ref{lem:der-d} show how to deduce from the previous lemma a bound on the Lipschitz constant of $
\frac{\partial \tilde \varphi(\Theta)}{\partial A_{\iota,\eta}}$ and of $
\frac{\partial \tilde \varphi(\Theta)}{\partial C_{\nu,\iota}}$ from the next lemma. We use the same set $\mathcal{D}$ as in Lemma \ref{lem:Lip-b}.
\begin{lem}\label{lem:Lip-d}
Let  $\tilde{\varphi}$ be the function defined in \eqref{eq:tilde phi} and let us consider the set $\mathcal{D}$ defined in the statement of Lemma \ref{lem:Lip-b} for some $D>0$. The function 
	\[
	\Theta=(A,d,C,d) \mapsto \frac{\partial \tilde \varphi(\Theta)}{\partial d}\in \R^{nk}
	\]
is Lipschitz continuous on $\mathcal{D}$ with a constant $L_d$ 
	satisfying
	\begin{equation*}
	L^2_d = \mathcal{C}_d \max\Big\{n^2D^2L_Z^2,n^3k\ell,n\ell^2D^6L_{Z}^{2}, n^2\ell^3D^4\Big\},
	\end{equation*}
	where $L_Z:=\sqrt{1 + \|\Sigma\|_S^2}$ and $\mathcal{C}_d$ is a constant that depends polynomially on  $\sigma'_{\max}$ and $\mu''_{\max}$.
\end{lem}
\begin{proof}
We set $\Theta=(A,b,C,d)$ and $\bar\Theta = (\bar A,\bar b,\bar C, \bar d)$ two distinct points in $\mathcal{D}$.
Let us fix $1\le \nu\le nk$ and set $\alpha$ and $\beta$ to be the unique numbers for which $\nu=(\alpha-1)k + \beta$, $1\le\alpha\le n$ and $1\le \beta\le k$. 
We computed in Lemma~\ref{lem:der-d} 
	\begingroup\makeatletter\def\f@size{9}\check@mathfonts	\def\maketag@@@#1{\hbox{\m@th\normalsize\normalfont#1}}	\begin{equation}\label{eq:lip-d-1}
	\frac{\partial\tilde{\varphi}(\Theta)}{\partial d_{\nu}}
	=
	2\sum_{j=1}^{n}\mu'(\omega_{\alpha,j})X_{(j-1)k+\beta},
	\end{equation}
	\endgroup
	with $\omega_{\alpha,j}$ and $X$ as defined in \eqref{shortcut-notation}.
	
	Let $1\le j\le n$ and set $\eta(j):=(j-1)k+\beta$. To obtain a bound on the Lipschitz constant of $\Gamma_{\alpha, \beta, j}:=\mu'(\omega_{\alpha,j})X_{\eta(j)}$, we apply Lemma \ref{lem:Gen-Lip} with 
	$\phi_{1}\equiv\mu'$, $\phi_{2}\equiv{\omega}_{\alpha,j}$, and
	$\phi_{3}\equiv X_{\eta(j)}$. By assumption, we know that
	$L_1:=\mu''_{\max}$ and $B_{12}:=\mu'_{\max}\leq 1$.
	By Step $5$ in the proof of Lemma \ref{lem:Lip-b}, we can take 
	$L_2:=L_{\omega,\alpha,j}$, while Step $4$ ensures we can set $L_{31}:=L_{\eta(j)}$ as in \eqref{eq:L_nu} (note that $L_2$ and $L_{31}$ are functions of $\Theta$ and $\bar\Theta$). Finally, we can let $B_3:=\sqrt{\ell+1}\cdot\sqrt{||{C}_{\eta(j),:}||^{2}+|{d}_{\eta(j)}|^{2}}$ and $L_{32} = 0$. We deduce that 
	\begingroup\makeatletter\def\f@size{9}\check@mathfonts	\def\maketag@@@#1{\hbox{\m@th\normalsize\normalfont#1}}
	\begin{equation}
	\label{eq:lip-d-2}
	\big|\Gamma_{\alpha,\beta,j}-\bar\Gamma_{\alpha,\beta,j}\big|\le  L_{\Gamma,\alpha,\beta,j}||\Theta-\bar\Theta||
	\end{equation}
	\endgroup
with 	
	\begingroup\makeatletter\def\f@size{7}\check@mathfonts	\def\maketag@@@#1{\hbox{\m@th\normalsize\normalfont#1}}
	\begin{equation*}
	L_{\Gamma,\alpha,\beta,j}=
	L_{\eta(j)}
	+\mu''_{\max}\sqrt{\ell+1}\cdot\sqrt{||{C}_{\eta(j),:}||^{2}+|{d}_{\eta(j)}|^{2}}L_{\omega,\alpha,j}.
	\end{equation*}
	\endgroup
    Since
	\begingroup\makeatletter\def\f@size{9}\check@mathfonts	\def\maketag@@@#1{\hbox{\m@th\normalsize\normalfont#1}}
	\begin{eqnarray*}
	\left(\tfrac{\partial\tilde{\varphi}(\Theta)}{\partial d_{\nu}}-\tfrac{\partial\tilde{\varphi}(\bar\Theta)}{\partial d_{\nu}}\right)^2&\le&4\bigg(\textstyle\sum\limits_{j=1}^nL_{\Gamma,\alpha,\beta,j}\bigg)^2||\Theta-\bar\Theta||^2,\\
	&\le& 4n\textstyle\sum\limits_{j=1}^nL^2_{\Gamma,\alpha,\beta,j}||\Theta-\bar\Theta||^2,
	\end{eqnarray*}
	\endgroup
	the function $\frac{\partial \tilde \varphi(\Theta)}{\partial d_\nu}$ is Lipschitz continuous. Now, the Lipschitz constant of $\frac{\partial \tilde \varphi(\Theta)}{\partial d}$ can be estimated by 
	\begingroup\makeatletter\def\f@size{7}\check@mathfonts	\def\maketag@@@#1{\hbox{\m@th\normalsize\normalfont#1}}
 	\begin{align*}
L_d^2&\le4n\sum_{\nu=1}^{nk}\sum_{j=1}^nL^2_{\Gamma,\alpha,\beta,j}=4n\sum_{\beta=1}^k\sum_{\alpha,j=1}^nL^2_{\Gamma,\alpha,\beta,j}\\
& \le 8n^2 \sum_{\beta=1}^k\sum_{j=1}^n L^2_{\eta(j)}+8n(\ell+1)\left(\mu''_{\max}\right)^2\times\\
 &\quad\left(\sum_{\beta=1}^k\sum_{\alpha,j=1}^n\left(||{C}_{\eta(j),:}||^{2}+|{d}_{\eta(j)}|^{2}\right)L^2_{\omega,\alpha,j}\right).
 	\end{align*}	
	\endgroup
	For the first sum above, we can write first
	\begingroup\makeatletter\def\f@size{9}\check@mathfonts	\def\maketag@@@#1{\hbox{\m@th\normalsize\normalfont#1}}
	\begin{eqnarray*}
	\sum_{j=1}^nL^2_{\eta(j)}&=&\left(\sigma'_{\max}L_Z\right)^2\sum_{j=1}^n||C_{\eta(j),:}||^2+n(\ell+1)\\
	&=& \left(\sigma'_{\max}L_Z\right)^2\sum_{\iota=1}^\ell||g(C_{:,\iota})_{:,\beta}||^2+n(\ell+1),
	\end{eqnarray*}
	\endgroup	
	where this last equality, 
	similarly to \eqref{M-transf}, follows from 
	        \begingroup\makeatletter\def\f@size{7}\check@mathfonts
		\def\maketag@@@#1{\hbox{\m@th\normalsize\normalfont#1}}
		\begin{equation*}
		\sum_{j=1}^n\|C_{\eta(j),:}\|^2
		= \sum_{j=1}^n \sum_{\iota=1}^\ell C_{(j-1)k+\beta,\iota}^2=\sum_{\iota=1}^\ell \| g(C_{:,\iota})_{:,\beta}\|^2.
		\end{equation*}
		\endgroup
    Therefore, 
	\begingroup\makeatletter\def\f@size{9}\check@mathfonts	\def\maketag@@@#1{\hbox{\m@th\normalsize\normalfont#1}}
	\begin{equation*}
	\sum_{\beta=1}^k\sum_{j=1}^nL^2_{\eta(j)}=\left(\sigma'_{\max}L_Z\right)^2||C||^2+nk(\ell+1).
	\end{equation*}
	\endgroup	
For the second sum above, we have from \eqref{step-5-Lip}:
        \begingroup\makeatletter\def\f@size{9}\check@mathfonts
		\def\maketag@@@#1{\hbox{\m@th\normalsize\normalfont#1}}
		\begin{align*}
		 &\left.\sum_{\alpha=1}^nL_{\omega,\alpha,j}^2\middle/\left(2\left(\ell+1\right)\left(\sigma'_{\max}L_{Z}\right)^{2}\right)\right.\\
		 &=\left(||C||^{2}+||d||^{2}\right)\left(\sum_{\iota=1}^{\ell}||g({C}_{:,\iota})_{j,:}||^{2}+\kappa\right)\nonumber\\
		 &\;\;+\left(\sum_{\iota=1}^{\ell}||g({\bar C}_{:,\iota})_{j,:}||^{2}+||g(\bar{d})_{j,:}||^{2}\right)\left(||C||^{2}+n\kappa\right),\nonumber
		 \end{align*}
		\endgroup
	with $\kappa:=(\ell+1)/\left(\sigma'_{\max}L_{Z}\right)^{2}$. Observe that
        \begingroup\makeatletter\def\f@size{7}\check@mathfonts
		\def\maketag@@@#1{\hbox{\m@th\normalsize\normalfont#1}}
		\begin{align*}
		 &\sum_{j=1}^n\sum_{\beta=1}^k\left(||{C}_{\eta(j),:}||^{2}+|{d}_{\eta(j)}|^{2}\right)\sum_{\iota=1}^{\ell}||g({C}_{:,\iota})_{j,:}||^{2}\\
		 &\le\sum_{j=1}^n\sum_{\beta=1}^k\left(||{C}_{\eta(j),:}||^{2}+|{d}_{\eta(j)}|^{2}\right)\left(\sum_{\iota=1}^{\ell}||g({C}_{:,\iota})_{j,:}||^{2} + ||g(d)_{j,:}||^{2}\right)\\
		 & =\sum_{j=1}^n\left(\sum_{\iota=1}^{\ell}||g({C}_{:,\iota})_{j,:}||^{2}+||g(d)_{j,:}||^{2}\right)^2\le D^4
		 \end{align*}
	\endgroup
	by the classical inequality $\sum_i\lambda_i^2\le\left(\sum_i|\lambda_i|\right)^2$.	Thus, 
	\begingroup\makeatletter\def\f@size{7}\check@mathfonts	\def\maketag@@@#1{\hbox{\m@th\normalsize\normalfont#1}}
 	\begin{align*}
&\left.\sum_{\beta=1}^k\sum_{\alpha,j=1}^n\left(||{C}_{\eta(j),:}||^{2}+|{d}_{\eta(j)}|^{2}\right)L^2_{\omega,\alpha,j}\middle/\left(2\left(\ell+1\right)\left(\sigma'_{\max}L_{Z}\right)^{2}\right)\right.\\
&\le D^2\left(D^4+D^2\kappa\right) + D^4(D^2+n\kappa) = 2D^6+(n+1)D^4\kappa.
 	\end{align*}	
	\endgroup
Putting everything together, we arrive at
	\begingroup\makeatletter\def\f@size{7}\check@mathfonts	\def\maketag@@@#1{\hbox{\m@th\normalsize\normalfont#1}}
\begin{align*}
	&\left.\sum_{\nu=1}^{nk}\left(\tfrac{\partial\tilde{\varphi}(\Theta)}{\partial d_{\nu}}-\tfrac{\partial\tilde{\varphi}(\bar\Theta)}{\partial d_{\nu}}\right)^2\middle/||\Theta-\bar\Theta||^2\right.\\
	&\le 8n^2\left(\left(\sigma'_{\max}DL_Z\right)^2+nk(\ell+1)\right)\\
	&\;\;\; + 16n(\ell+1)^2\left(\mu''_{\max}\right)^2\left(2\left(\sigma'_{\max}L_{Z}\right)^{2}D^6+(n+1)D^4(\ell+1)\right).
	\end{align*}
	\endgroup
Therefore,
	\begingroup\makeatletter\def\f@size{7}\check@mathfonts
	\def\maketag@@@#1{\hbox{\m@th\normalsize\normalfont#1}}
	\begin{equation*}
	\begin{split}
	L^2_d
	&=O\Bigg(\max\Big\{n^2D^2\left(\sigma'_{\max}L_Z\right)^2,n^3k\ell,n\ell^2D^6\left(\mu''_{\max}\sigma'_{\max}L_{Z}\right)^{2},\\
	&\quad \quad \quad\quad\quad n^2\ell^3D^4\left(\mu''_{\max}\right)^2\Big\}\Bigg)\\
	&=\mathcal{C}_d \max\Big\{n^2D^2L_Z^2,n^3k\ell,n\ell^2D^6L_{Z}^{2}, n^2\ell^3D^4\Big\},
	\end{split}
	\end{equation*}
	\endgroup
	where $\mathcal{C}_d$ is a constant that only depends polynomially on  $\sigma'_{\max},\mu''_{\max}$.
\end{proof}

We can now merge all our previous results to determine a bound on the Lipschitz constant of $\nabla\tilde\varphi$. Again, we use the set $\mathcal{D}$ as in Lemma \ref{lem:Lip-b} and 
$L_Z:=\sqrt{1+||\Sigma||_S^2}$.
\begin{lem}\label{lem:L-smooth}
	The function 
	\begin{equation*}
     \Theta = (A,b,C,d) \mapsto \nabla \tilde \varphi(\Theta)
	\end{equation*}
	is Lipschitz continuous with Lipschitz constant $L_{\tilde\varphi}$ satisfying
	\begingroup\makeatletter\def\f@size{7}\check@mathfonts
	\def\maketag@@@#1{\hbox{\m@th\normalsize\normalfont#1}}
	\begin{equation*}
	\begin{split}
	L^2_{\tilde\varphi}
	&=
	\mathcal{C}_{\tilde\varphi} n^2\ell
	\max\!\bigg\{D^4 L^4_Z, \ell D^8L_Z^4, \ell^2D^6L^2_Z,n D^2L^2_Z,nk\ell,\ell^3D^4\bigg\},
	\end{split}
	\end{equation*}
	\endgroup
	where $\mathcal{C}_{\tilde\varphi}$ is a constant that only depends polynomially on  $\sigma'_{\max}$, $\sigma''_{\max}$, and $\mu''_{\max}$.
\end{lem}
\begin{proof}
Let $\Theta=(A,b,C,d)$ and $\bar\Theta = (\bar A,\bar b,\bar C, \bar d)$ two distinct points in $\mathcal{D}$. In view of Lemmas \ref{lem:der-b} and \ref{lem:der-d} 
	\begingroup\makeatletter\def\f@size{9}\check@mathfonts
	\def\maketag@@@#1{\hbox{\m@th\normalsize\normalfont#1}}
	\begin{align*}
	&||\nabla\tilde{\varphi}(\Theta)-\nabla\tilde{\varphi}(\bar{\Theta})||^{2}\\
	 &=
	 \textstyle\sum\limits_{\iota=1}^{\ell}\left(\frac{\partial\tilde{\varphi}(\Theta)}{\partial b_{\iota}}-\frac{\partial\tilde{\varphi}(\bar{\Theta})}{\partial b_{\iota}}\right)^{2}\bigg(1+\sum\limits_{\eta=1}^{r}[h(\Sigma)]_{\eta}^{2}\bigg)\\
	 & \quad
	 +\textstyle\sum\limits_{\nu=1}^{nk}\left(\frac{\partial\tilde{\varphi}(\Theta)}{\partial d_{\nu}}-\frac{\partial\tilde{\varphi}(\bar{\Theta})}{\partial d_{\nu}}\right)^{2}\left(1+\sum\limits_{\iota=1}^{\ell}Y_{\iota}^{2}\right)\\
	 &\le \left(L_b^2(1+||\Sigma||^2_S)+L_d^2(1+\ell)\right)||\Theta-\bar\Theta||^2,
	\end{align*}
	\endgroup
    because $Y_\iota=\sigma([Ah(\Sigma)+b]_\iota)\le 1$ for all $1\le \iota\le \ell$. It remains now to use our estimates for $L_b$ and $L_d$ obtained in Lemma \ref{lem:Lip-b} and Lemma \ref{lem:Lip-d}.
\end{proof}

Now that the single-layered neural network parametrization has been completely treated, we can turn our attention to a multi-layered neural network  configuration. 
The set $\mathcal{D}$ is now the one defined in the statement of Theorem~\ref{thm:convergence}, and the notation follows that of Section~\ref{sec:main}.
\begin{proof}[Proof of Theorem~\ref{thm:convergence}]${}$\\
Let $\Theta,\bar\Theta\in\mathcal{D}$. For $0\le u\le v\le m$, we define $\mathcal{N}^{u:v}:\R^{\ell_u}\to\R^{\ell_v}$ so that $\mathcal{N}^{u:u}(x) = x$ and $\mathcal{N}^{u:v}(x):=\sigma^{(v)}\left(A^{(v-1)}\mathcal{N}^{u:v-1}(x)+b^{(v-1)}\right)$ for $u<v$. Observe that $\mathcal{N}^{u:v}(x)=\mathcal{N}^{c:v}\circ\mathcal{N}^{u:c}(x)$ when $0\le u\le c\le v\le m$ and that $\mathcal{N}^\Theta_m(x)=A^{(m)}\mathcal{N}^{0:m}(x)+b^{(m)}$. 

For convenience, we replace $\sigma^{(u)}$ by $\tilde\sigma$ when no confusion is possible. For all $||v||\le D$, we set $F(v):=\sum_{i,j}^n\mu\left(\left[g(v)g(v)^T-\Sigma\right]_{i,j}\right)$. Replicating \eqref{eq:Step 5 initial} in Step 5 of Lemma \ref{lem:Lip-b}, we deduce the bound $L_F:=2\sqrt{n}D$ for the Lipschitz constant of $F$. Similarly 
following the proof of Lemma \ref{lem:Lip-d} with $C=0$, we can deduce one for $\frac{\partial F}{\partial v_\nu}$ as $L_{F'}:=2n(1+\mu''_{\max}D^2)$ by \eqref{eq:in Lem 2} and Lemma \ref{lem:Gen-Lip}.

We fix $0\le u\le m$, $1\le i \le \ell_{u+1}$, and $1\le j\le \ell_u$.  As in Lemma \ref{lem:der-b}, we can check that 
	\begingroup\makeatletter\def\f@size{7}\check@mathfonts
	\def\maketag@@@#1{\hbox{\m@th\normalsize\normalfont#1}} 
	\begin{align}
	&\frac{\partial\varphi(\Theta)}{\partial A_{i,j}^{(u)}} = \frac{\partial F(A^{(m)}\mathcal{N}^{u+1:m}\circ\mathcal{N}^{0:u+1}(h(\Sigma))+b^{(m)})}{\partial A_{i,j}^{(u)}}\nonumber \\
 	&=
 	\Bigg[\left|\nabla_x F(A^{(m)}\mathcal{N}^{u+1:m}(x)+b^{(m)})\right|_{x=\mathcal{N}^{0:u+1}(h(\Sigma))}\Bigg]_i\times\nonumber\\
 	&\quad \Bigg[\tilde\sigma'\left(A^{(u)}\mathcal{N}^{0:u}(h(\Sigma))+b^{(u)}\right)\Bigg]_{i,i}\left[\mathcal{N}^{0:u}(h(\Sigma))\right]_j\nonumber\\
 	&=\frac{\partial\varphi(\Theta)}{\partial b_i^{(u)}}\left[\mathcal{N}^{0:u}(h(\Sigma))\right]_j.\label{eq:partial_der_phi}
	\end{align}
    \endgroup
Note that $\tilde\sigma'(v)$ is an $\ell_{u+1}\times\ell_{u+1}$ diagonal matrix. Since $||\mathcal{N}^{0:u}(h(\Sigma))||^2\le \ell_{u}$ for $u>0$, we get
	\begingroup\makeatletter\def\f@size{7}\check@mathfonts
	\def\maketag@@@#1{\hbox{\m@th\normalsize\normalfont#1}} 
	\begin{equation}\label{eq:the ell_u}
	\sum_{j=1}^{\ell_{u}}\bigg(\frac{\partial\varphi(\Theta)}{\partial A_{i,j}^{(u)}}-\frac{\partial\varphi(\bar\Theta)}{\partial A_{i,j}^{(u)}}\bigg)^2 \le \ell_{u}\left(\frac{\partial\varphi(\Theta)}{\partial b_{i}^{(u)}}-\frac{\partial\varphi(\bar\Theta)}{\partial b_{i}^{(u)}}\right)^2.
	\end{equation}
    \endgroup
    For $u=0$, we need to replace $\ell_u$ by $||\Sigma||^2_S$. 
    Denoting $\mathcal{N}^\Theta_u:=A^{(u)}\mathcal{N}^{0:u}(h(\Sigma))+b^{(u)}$ and expanding \eqref{eq:partial_der_phi}, we have for $u<m$:
 	\begingroup\makeatletter\def\f@size{7}\check@mathfonts
	\def\maketag@@@#1{\hbox{\m@th\normalsize\normalfont#1}} 
	\begin{equation}\label{eq:big product}
		\frac{\partial\varphi(\Theta)}{\partial b_{i}^{(u)}} = \left[\tilde\sigma'(\mathcal{N}^\Theta_{u})   
		A^{(u+1)T}
		\cdots\tilde\sigma'(\mathcal{N}^\Theta_{m-1})A^{(m)T}F'(\mathcal{N}^\Theta_m)\right]_i.
	\end{equation}
    \endgroup
Using exactly the same reasoning as in Step 4 of Lemma \ref{lem:Lip-b}, a Lipschitz constant $\mathcal{L}_{u}$ bound for $\mathcal{N}^\Theta_{u}$ satisfies
	$\mathcal{L}^2_{u} :=D^2_+\mathcal{L}^2_{u-1}+\ell_{u}+1$ with $\mathcal{L}^2_{0}=1+||\Sigma||_S^2=L_Z^2$ and $D_+:=\max\{D\sigma'_{\max},1.001\}>1$, so that 	
	\begingroup\makeatletter\def\f@size{7}\check@mathfonts
	\def\maketag@@@#1{\hbox{\m@th\normalsize\normalfont#1}} 
   \begin{equation*}
	\mathcal{L}^2_{u} = D_+^{2u}L_Z^2 + \sum_{j=1}^{u}D_+^{2(u-j)}(\ell_{j}+1)\leq D_+^{2u}\left(L_Z^2+\frac{\ell_{\max}+1}{D_+^2-1}\right).
	\end{equation*}
	\endgroup
To estimate the Lipschitz constant of \eqref{eq:big product}, Lemma \ref{lem:Gen-Lip} with 
$\phi_1\equiv \frac{\partial F}{\partial v_\nu}$, $\phi_2\equiv \mathcal{N}^\Theta_{m}$, $\phi_3\equiv A^{(m)}_{\nu,i}$ gets us a Lipschitz constant bound of $|A^{(m)}_{\nu,i}|L_{F'}\mathcal{L}_m$ for $A^{(m)}_{\nu,i}\frac{\partial F(\mathcal{N}^\Theta_{m})}{\partial v_\nu}$ and a bound of $|A^{(m)}_{\nu,i}|L_{F}$. With $\phi_1\equiv\tilde\sigma'$, $\phi_2\equiv \mathcal{N}^\Theta_{m-1}$, and $\phi_3\equiv A^{(m)T}_{:,i}F'(\mathcal{N}^\Theta_{m})$, we get a Lipschitz constant estimate of $\left(\sigma''_{\max}\mathcal{L}_{m-1}L_F + \mathcal{L}_{m}L_{F'}\right)||A^{(m)}_{:,i}||$ for $[\sigma'(\mathcal{N}^\Theta_{m-1})A^{(m)T}F'(\mathcal{N}^\Theta_m)]_i$. Pursuing as follows, we end with a Lipschitz constant bound for $\frac{\partial\varphi(\Theta)}{\partial b_{i}^{(u)}}$ and $u<m$ of 
\begingroup\makeatletter\def\f@size{9}\check@mathfonts
\def\maketag@@@#1{\hbox{\m@th\normalsize\normalfont#1}} 
   \begin{align*}
	L_i^{(u)}:=&||A^{(u+1)}_{:,i}||\cdots||A^{(m)}||\left(\sigma'_{\max}\right)^{m-u-2}\times\\
	&\quad\left(L_{F'}\mathcal{L}_{m}\sigma'_{\max}+L_F\sigma''_{\max}\textstyle\sum\limits_{k=u}^{m-1}\mathcal{L}_{k}\right).
	\end{align*}
	\endgroup
	For $u = m$, we can simply take $L_i^{(m)}:=L_{F'}$.
Note that 
\begingroup\makeatletter\def\f@size{9}\check@mathfonts
\def\maketag@@@#1{\hbox{\m@th\normalsize\normalfont#1}} 
   \begin{align*}
	&\textstyle\sum\limits_{k=u}^{m-1}\mathcal{L}_{k}\le \frac{D_+^m-D_+^u}{D_+-1}\left(L_Z+\frac{\sqrt{\ell_{\max}+1}}{D_+-1}\right).
	\end{align*}
	\endgroup
Using \eqref{eq:the ell_u} and summing up over $i$ yields
     	\begingroup\makeatletter\def\f@size{7}\check@mathfonts	\def\maketag@@@#1{\hbox{\m@th\normalsize\normalfont#1}} 
	\begin{align*}
	&\sum_{i=1}^{\ell_{u+1}}\sum_{j=1}^{\ell_{u}}\bigg(\frac{\partial\varphi(\Theta)}{\partial A_{i,j}^{(u)}}-\frac{\partial\varphi(\bar\Theta)}{\partial A_{i,j}^{(u)}}\bigg)^2 +\left(\frac{\partial\varphi(\Theta)}{\partial b_{i}^{(u)}}-\frac{\partial\varphi(\bar\Theta)}{\partial b_{i}^{(u)}}\right)^2\\
	&\le(1+\ell_u)\sum_{i=1}^{\ell_{u+1}}\left(L_i^{(u)}\right)^2||\Theta-\bar\Theta||^2.
	\end{align*}
    \endgroup 
      Finally, by summing $u$ from $0$ to $m$, we obtain that
      	\begingroup\makeatletter\def\f@size{7}\check@mathfonts	\def\maketag@@@#1{\hbox{\m@th\normalsize\normalfont#1}} 
	\begin{align*}
    L^2_m = \mathcal{C}\max\Big\{kn^3D^4\ell_{\max},&nD^{4m+2}\max\{\ell_{\max},L_Z^2\}\times\\
    & \max\{nD^2L_Z^2,nD^2\ell_{\max}
    \Big\},
	\end{align*}
    \endgroup   
where $\mathcal{C}$ is a constant that only depends polynomially on $\mu_{\max}''$, $\sigma_{\max}'$, and $\sigma_{\max}''$ (with powers in $O(m)$).

     The second part of Theorem~\ref{thm:convergence} is well-known in optimization theory; see, e.g., \cite[Section~1.2.3]{nesterov2013introductory}.
\end{proof}
%
 \section*{Acknowledgment}
	\noindent
C.\ Herrera gratefully acknowledges the support from ETH-foundation and from the Swiss National Foundation for the project \textit{Mathematical Finance in the light of machine learning}, SNF Project 172815.\\ A.\ Neufeld gratefully acknowledges the financial support by his Nanyang Assistant Professorship (NAP) Grant \textit{Machine Learning based Algorithms in Finance and Insurance}.\\
The authors would like to thank
 Florian Krach, Hartmut Maennel, Maximilian Nitzschner, and Martin Stefanik 
for their careful reading and suggestions.
Special thanks go to Josef Teichmann for his numerous helpful suggestions, ideas, and discussions.

\bibliographystyle{plain}


      \end{document}